\documentclass[letterpaper, 10pt, dvipsnames]{article}

\usepackage{bold-extra}
\usepackage[utf8]{inputenc}    
\usepackage[T1]{fontenc}
\usepackage[english]{babel}
\usepackage[normalem]{ulem}
\usepackage{dsfont, bm, pifont, mathrsfs}
\usepackage{stmaryrd}
\usepackage{comment}
\usepackage{bbm}
\usepackage{enumitem}
\usepackage{float, graphicx, caption, wrapfig, tikz}
\usepackage[margin=30pt]{subcaption}
\usepackage{amsthm, amsmath, amsfonts, amssymb, mathrsfs, mathtools}
\usepackage{amsrefs}
\usepackage{hyperref, xcolor, titlesec} 
\hypersetup{colorlinks = true, urlcolor = RoyalBlue!70!Black,linkcolor=BrickRed, citecolor=Plum, bookmarksopen = true}
\usepackage[nomarginpar]{geometry}
\numberwithin{equation}{section}
\usepackage{cleveref}
\renewcommand{\ge}{\geqslant}
\renewcommand{\le}{\leqslant}

\DeclareMathOperator{\Id}{Id}

\let \d \relax
\newcommand{\d}{\mathrm{d}}
\newcommand{\I}{\mathrm{i}}
\newcommand{\scp}[2]{\langle #1,#2\rangle}
\newcommand{\Q}{\mathcal{Q}}
\newcommand{\eb}{\mathbf{e}}
\let\O\relax
\newcommand{\O}[1]{\mathcal{O}\left(#1\right)}
\newcommand{\E}[1]{\mathds{E}\left[#1\right]}
\newcommand{\EL}[1]{\mathds{E}\left[#1\middle\vert\bm{\lambda}\right]}

\newcommand{\unn}[2]{[\![#1,#2]\!]}
\let\Im\relax
\DeclareMathOperator{\Im}{Im}
\let\Re\relax
\DeclareMathOperator{\Re}{Re}
\DeclareMathOperator{\Tr}{Tr}
\DeclareMathOperator{\tr}{Tr}

\newcommand{\e}{\mathrm{e}}
\newcommand{\one}{\mathbbm{1}}
\renewcommand{\epsilon}{\varepsilon}
\renewcommand{\tilde}{\widetilde}
\newcommand{\mom}[4]{\scp{\q_{\alpha_1}}{\u_{#1}}\scp{\q_{\beta_1}}{\u_{#2}}\scp{\q_{\alpha_2}}{\u_{#3}}\scp{\q_{\beta_2}}{\u_{#4}}}

\def\bet{\begin{thm}}
\def\eet{\end{thm}}
\def\bel{\begin{lem}}
\def\eel{\end{lem}}
\def\bas{\begin{ass}}
\def\eas{\end{ass}}
\def\bec{\begin{cor}}
\def\eec{\end{cor}}
\def\bed{\begin{defn}}
\def\eed{\end{defn}}
\def\bep{\begin{prop}}
\def\eep{\end{prop}}
\def\beq{\begin{equation}}
\def\eeq{\end{equation}}
\def\bea{\begin{equation*}}
\def\eea{\end{equation*}}
\def\bex{\begin{ex}}
\def\eex{\end{ex}}
\def\bp{\begin{proof}}
\def\ep{\end{proof}}

\def\1{{\mathbbm 1}}

\def\benr{\begin{enumerate}[label=(\roman*)]}
\def\eenr{\end{enumerate}}
\def\A{\mathcal{A}}
\def\B{\mathcal{B}}
\def\F{\mathcal{F}}
\def\M{\mathcal{M}}
\def\D{\mathcal{D}}
\def\N{\mathbb{N}}

\def\R{\mathbb{R}}
\def\C{\mathbb{C}}
\def\P{\mathbb{P}}
\def\E{\mathbb{E}}
\def\q{\mathbf{q}}
\newcommand\x{\mathbf{x}}
\def\S{\mathbb{S}}

\renewcommand{\u}{\mathbf{u}}
\renewcommand{\v}{\mathbf{v}}

\def\msc{m_{\mathrm{sc} } }

\def\one{{\mathbbm 1}}
\def\eps{\varepsilon}

\def\scrho{\rho_{\mathrm{sc} }}
\newcommand{\bma}{\begin{bmatrix}}
\newcommand{\ema}{\end{bmatrix}}

\def \wt {\widetilde}
\def\phi{\varphi}
\def\matn{\mathrm{Mat}_{N} }
\DeclareMathOperator{\GOE}{GOE}

\renewcommand{\hat}{\widehat}
\newcommand{\ximax}{\widetilde{\bm{\xi}}}

\newcommand{\fa}{\mathfrak{a}}

\newtheorem{ccounter}{ccounter}[section]
\newtheorem{thm}[ccounter]{Theorem}
\newtheorem{lem}[ccounter]{Lemma}
\newtheorem{cor}[ccounter]{Corollary}
\newtheorem{defn}[ccounter]{Definition}
\newtheorem{prop}[ccounter]{Proposition}
\newtheorem{ass}[ccounter]{Assumption}
\newtheorem{ex}[ccounter]{Example}

\theoremstyle{definition}
\newtheorem{rmk}[ccounter]{Remark}
\titleformat{\paragraph}[runin]{\itshape\normalsize}{\theparagraph}{}{}
\titleformat{\subparagraph}[runin]{\itshape\normalsize}{\theparagraph}{0em}{}
\titleformat{\section}[block]{\normalfont\filcenter}{\Large\thesection .}{.7em}{\Large\scshape}
\titleformat{\subsection}[runin]{\normalfont}{\large \bf \thesubsection .}{.5em}{\large\bf}[.]
\titleformat{\subsubsection}[runin]{\normalfont}{\bf \thesubsubsection .}{.5em}{\bf}[.]

\usepackage{tocloft}

\begin{document}
\tikzset{every node/.style={draw,circle,fill=black, scale=.5}}

\title{\scshape\bfseries{Fluctuations in local quantum unique ergodicity for generalized Wigner matrices}}
\author{L. \textsc{Benigni}\\\vspace{-0.15cm}\footnotesize{\it{University of Chicago}}\\\footnotesize{\it{lbenigni@math.uchicago.edu}}\and P. \textsc{Lopatto\footnote{P.L. is supported by National Science Foundation grant DMS-1926686.}}\\\vspace{-0.15cm}\footnotesize{\it{Institute for Advanced Study}}\\\footnotesize{\it{lopatto@ias.edu}}}
\date{}
\maketitle

\renewcommand{\abstractname}{\normalsize\normalfont\scshape Abstract}
\renewcommand{\contentsname}{}

\begin{abstract}
\small{
We study the eigenvector mass distribution for generalized Wigner matrices on a set of coordinates $I$,
where $N^\eps \le | I |  \le N^{1- \eps}$, and prove it converges to a Gaussian at every energy level, including the edge, as $N\rightarrow \infty$.
The key technical input is a four-point decorrelation estimate for eigenvectors of matrices with a large Gaussian component.
Its proof is an application of the maximum principle to a new set of moment observables satisfying
parabolic evolution equations.
Additionally, we prove high-probability Quantum Unique Ergodicity and Quantum Weak Mixing bounds for all eigenvectors and all deterministic sets of entries using a novel bootstrap argument.
}
\end{abstract}

{
  \hypersetup{linkcolor=black}
  \tableofcontents
}

\section{Introduction}
Wigner random matrices are the simplest example of a fully delocalized random Hamiltonian, and consequently, they are a fundamental model in mathematical physics. 
The goal of this work is to prove a strong form of delocalization for such matrices, which we characterize by the averages over subsets of eigenvector entries. 
When averaging over coordinate sets asymptotically smaller than the size of the matrix,
one expects that the entries should behave as independent random variables. To make this notion precise, we begin by considering a matrix of Gaussians, which is amenable to direct computation.

Recall that the Gaussian Orthogonal Ensemble (GOE) is defined as the $N \times N$ real symmetric random matrix $\mathrm{GOE}_N = \{ g_{ij} \}_{1 \le i,j \le N}$ whose upper triangular entries $g_{ij}$ are mutually independent Gaussian random variables with variances $( 1 + \mathds{1}_{i =j}) N^{-1}$. 
Its distribution is invariant under multiplication by orthogonal matrices, which implies its $\ell^2$-normalized eigenvectors are uniformly distributed on the unit sphere $\S^{N-1}$.
Using this fact, one can directly calculate the mass distribution of these eigenvectors. 
Fix $\alpha \in (0,1)$, let $I_N$ denote a set $I_N \subset \unn{1}{N}$ with $| I_N | = \lfloor N^\alpha \rfloor$,
and let the norm $\| \cdot \|_{I_N}^2$ denote the $\ell^2$ norm restricted to the coordinates in $I_N$.
A straightforward computation shows that if $\mathbf{v}_N$ is uniformly distributed on $\S^{N-1}$, then
\beq\label{clt}
\sqrt{\frac{N^2}{2 \left|I_N\right|}} \left( \left\| \mathbf{v}_N \right\|^2_{I_N} - \frac{\left|I_N\right|}{N} \right)\rightarrow \mathcal N(0,1)
\eeq
in distribution as $N \rightarrow \infty$, where $\mathcal{N}(0,1)$ is a standard normal random variable \cite{o2016eigenvectors}*{Theorem 2.4}. 
This central limit theorem makes precise the notion that the $\mathrm{GOE}_N$ eigenvector entries 
fluctuate independently on all sub-global scales. 

This paper is devoted to a proof of \eqref{clt} for a broad class of mean-field random matrices, including Wigner matrices.

\subsection{Main results}

We first define generalized Wigner matrices. 

\bed \label{d:wigner} A generalized Wigner matrix $H$ is a real symmetric or complex Hermitian $N\times N$ matrix whose upper triangular elements $\{h_{ij}\}_{i\le j}$ are independent random variables with mean zero and variances $\sigma_{ij}^2=\E\left[|h_{ij}|^2\right]$ that satisfy 

\beq \label{e:stochasticvar}
\sum_{i=1}^N \sigma_{ij}^2 =1 \quad \text{ for all } j\in \llbracket1,N\rrbracket\eeq
and
\beq 
\frac{c}{N} \le \sigma_{ij}^2 \le \frac{C}{N} \quad  \text{ for all } i,j \in \llbracket1,N\rrbracket\eeq
for some constants $C, c >0$. 
Further, we suppose that the normalized entries have finite moments, uniformly in $N$, $i$, and $j$, in the sense that for all $p\in \N$ there exists a constant $\mu_p$ such that
\beq\label{finitemoments}
\E\left|  \sqrt{N} h_{ij}  \right|^p \le \mu_{p}.\eeq
for all $N$, $i$, and $j$.
\eed

Our first main result determines the mass distribution for a single eigenvector on all local scales. 

\bet\label{t:main1}
Let $H$ be a generalized Wigner matrix and fix $\eps >0$. 
Let $I = I_N \subset \unn{1}{N}$ be a deterministic sequence of subsets satisfying $N^\eps \le | I | \le N^{1- \eps}$, let $i = i_N \in \unn{1}{N}$ be a sequence of indices, and let 
$\u  = \u_{i_N}^{(N)}$ be the corresponding sequence of $\ell^2$-normalized eigenvectors of $H$. Let $(\q_\alpha)_{\alpha\in I}=(\q_{\alpha}^{(N)})_{\alpha\in I}$ be a deterministic sequence of sets of orthogonal vectors in $\S^{N-1}$. Then
\begin{equation}
\sqrt{\frac{\beta N^2}{2 | I |}} \left( \sum_{\alpha\in I} \scp{\q_\alpha}{\u}^2   -  \frac{|I|}{N} \right) \rightarrow \mathcal N(0,1)
\end{equation}
with convergence in the sense of moments.
Here $\mathcal N(0,1)$ is a standard real Gaussian random variable; we take $\beta = 1$ if $H$ is real symmetric, or $\beta = 2$ if it is complex Hermitian.
\eet

\begin{rmk} The proof gives an explicit rate of convergence. The $n$th moment converges with rate $N^{-c}$ for some constant $c(\eps, n) >0$.
\end{rmk}

As an intermediate step in the proof of \Cref{t:main1}, we establish the following theorem. It provides high-probability quantum unique ergodicity and quantum weak mixing estimates for all eigenvectors, on all scales, with an error bound that is optimal up to a factor of $N^\eps$.

\bet\label{t:main2}
Let $H$ be a generalized Wigner matrix.
Let $I = I_N \subset \unn{1}{N}$ be a deterministic sequence of subsets, let $(i ,j )=( i_N,j_N) \in \unn{1}{N}^2$ be a sequence of index pairs with $i_N \neq j_N$ for all $N$, and let 
$(\u,\mathbf{v})  = (\u_{i_N}^{(N)}, \mathbf{v}_{i_N}^{(N)})$ be the corresponding sequence of eigenvector pairs of $H$. 
Let $(\q_\alpha)_{\alpha\in I}=(\q_{\alpha}^{(N)})_{\alpha\in I}$ be a deterministic sequence of sets of orthogonal vectors in $\S^{N-1}$.
Then for every $D, \eps >0$, there exists $C (D,\eps) >0$ such that 
\begin{equation}
\P \left(  \left|\sum_{\alpha\in I} \scp{\q_\alpha}{\u}^2   -  \frac{|I|}{N} \right| \ge N^\eps\frac{\sqrt{\vert I\vert}}{N} \right) \le C N^{-D}
\end{equation}
and 
\begin{equation}
\P \left(  \left|\sum_{\alpha \in I} \scp{\q_\alpha}{\u}\scp{\q_\alpha}{\v} \right| \ge N^\eps\frac{\sqrt{\vert I\vert}}{N}  \right) \le C N^{-D}.
\end{equation}
\eet
\begin{rmk}
Given the conclusion of \Cref{t:main2}, the factor of $N^\eps$ could likely be improved to the optimal logarithmic factor by mimicking the bootstrap argument used to achieve optimal eigenvector delocalization bounds in \cite{benigni2020optimal}. For brevity we do not take this up here.
\end{rmk}

\subsection{Background}

In a disordered quantum system, delocalization can be understood in many ways. 
Perhaps the most familiar perspective is that a delocalized eigenstate is globally ``flat,'' with a small $\ell^\infty$ norm. 
For Wigner matrices, the first strong $\ell^\infty$
bounds were achieved in the seminal papers \cites{erdos2009semicircle, erdos2009local, erdos2010wegner}, which showed $\Vert \sqrt{N}\u_\ell\Vert_\infty\leqslant (\log N)^{4}$ with high probability, assuming strong smoothness and decay conditions on the matrix entries, and that $\u_\ell$ corresponds to an eigenvalue in the bulk of the spectrum. 
This result was then successively improved in a number of directions, including weakening the assumptions on the entries 
\cites{gotze2018local, gotze2019local,aggarwal2019bulk, aggarwal2018goe}, treating eigenvectors at the edge \cites{erdos2012rigidity, tao2010random}, and strengthening the supremum bound \cites{tao2011random, vu2015random, erdos2012bulk, rudelson2015delocalization, o2016eigenvectors}.
Recently, \cite{benigni2020optimal} achieved the optimal $\ell^\infty$ delocalization rate $\sqrt{\log N}$ throughout the spectrum, with optimal constants.

While these results concern delocalization as measured in the standard coordinate basis, 
it is natural to consider arbitrary projections $\langle \q , \u_\ell \rangle$ for unit vectors $\q  \in \S^{N-1}$, known as isotropic delocalization. High-probability isotropic delocalization bounds were first shown in \cites{knowles2013isotropic,bloemendal2014isotropic} and improved to the optimal rate in \cite{benigni2020optimal}.

One can also ask about the behavior of individual eigenvector entries of (generalized) Wigner matrices. Preliminary investigations into this question were made in \cites{knowles2013eigenvector, tao2012random}. Shortly after, \cite{bourgade2013eigenvector} obtained joint normality of any finite number of entries $\big((\u_k(\alpha_1), \dots, \u_k(\alpha_j) \big)$ of a fixed eigenvector $\u_k$, and also joint normality of collections such as $\big(\u_{k_1}(\alpha), \u_{k_2}(\alpha), \dots, \u_{k_j}(\alpha)\big)$, concerning a fixed coordinate of multiple eigenvectors. Further, \cite{bourgade2013eigenvector} was a conceptual breakthrough because it introduced a dynamical approach to the study of eigenvectors via the \emph{eigenvector moment flow}, which we will discuss in detail below. This approach was later applied to sparse matrices \cite{bourgade2017eigenvector}, L\'evy matrices \cite{aggarwal2020eigenvector}, and deformed Wigner matrices \cite{benigni2020eigenvectors}; it was also used to prove the optimal delocalization of eigenvectors in \cite{benigni2020optimal}.  
Recently, \cite{marcinek2020high} introduced the \emph{colored eigenvector moment flow} and gave a sophisticated analysis of it using the energy method, from which the authors derived the joint normality of \emph{any} finite collection of eigenvector entries, from any row or column of the matrix of eigenvectors.

Another form of delocalization, \emph{Quantum Unique Ergodicity} (QUE), was introduced as a conjecture in \cite{rudnick1994behaviour}.
In its original form, it states that for any compact manifold $\mathcal M$ of negative curvature, the eigenstates of the Laplacian become uniformly distributed with respect to the volume form; that is, for any open $U \subset \mathcal M$:
\beq
\int_U \left| \psi_k(x) \right|^2 \d_{\mathrm{vol}} \rightarrow \int_U \d_{\mathrm{vol}} \quad \text{as $k\rightarrow\infty$.}
\eeq
Earlier notions of \emph{quantum ergodicity} demanded only that \emph{most} eigenstates tend to uniformity
\cites{de1985ergodicite,shnirel1974ergodic,zelditch1987uniform}. 
Geometric quantum ergodicity, QUE, and related questions continue to be subjects of active research \cites{
holowinsky2010sieving,
lindenstrauss2006invariant,
anantharaman2008entropy,
anantharaman2015quantum,
brooks2013non,
brooks2014joint,
anantharaman2019quantum,
schubert2008rate,
luo1995quantum,
schubert2006upper,
holowinsky2010mass,
eckhardt1995approach,
marklof2000quantum}.

In the random matrix setting, QUE was proved in \cite{bourgade2013eigenvector} in the following form for all $I \subset \unn{1}{N}$. For any generalized Wigner matrix and trace zero diagonal matrix $A$ supported on $I$ coordinates with $\| A \| \le 1$, the mass average $N | I |^{-1} \langle \u , A \u \rangle$ tends to zero in probability. It represents a form of eigenvector delocalization on all scales, including those intermediate between single entry normality and global $\ell^\infty$ control.
The convergence rate was later strengthened in \cite{bourgade2018random} for Gaussian divisible ensembles, where it was used to prove universality results for random band matrices; see \cite{bourgade2018survey} for a recent survey. Recently, \cite{benigni2021fermionic} proved decorrelation between the QUE averages for distinct eigenvectors by introducing a new set of moment observables following a variant of the eigenvector moment flow. We also note that a QUE estimate was obtained for random $d$-regular graphs and Erd\H{o}s--R\'enyi graphs in \cite{bauerschmidt2017local} using the exchangeability of the entries. 

One can also consider the stronger \emph{Eigenstate Thermalization Hypothesis} (ETH), which states that for a disordered quantum system, general observables become uniformly distributed \cites{deutsch2018eigenstate, srednicki1994chaos}.
Very recently, \cite{erdHos2020eigenstate} obtained a version of the ETH
for Wigner matrices with identically distributed entries. 
Their result states that if $H$ is a Wigner matrix, then for any deterministic matrix $A$ with $\| A \| \le 1$, 
\beq\label{eth}
\max_{i,j} \left | \langle \u_i, A \u_j \rangle -  \delta_{ij} N^{-1} \tr A \right | \le N^{ - 1/2 + \eps}
\eeq
with high probability (in the sense of \Cref{t:main2}); a similar statement holds with $\overline{\u_j}$ replacing $\u_j$. 
In follow-up work, \cite{cipolloni2021normal} proved Gaussian fluctuations in QUE for global observables. More precisely, for a Wigner matrix $H$ and deterministic traceless $A$ with $\| A \| \le 1$ and $\tr A^2 \ge c N$ for some $c>0$, 
\beq\label{globalfluctuations}
\sqrt{\frac{\beta N^2}{2 \tr A^2 }} \langle \u_i , A \u_i \rangle \rightarrow \mathcal N(0,1)
\eeq
for any eigenvector $\u_i$ corresponding to an eigenvalue $\lambda_i$ in the bulk of the spectrum. 

Lastly, we discuss how our results relate to these previous works. Compared to \Cref{t:main1}, \eqref{globalfluctuations} identifies fluctuations in a complementary regime.
\Cref{t:main1} studies QUE fluctuations on local scales, $| I | \ll N$, while \eqref{globalfluctuations} studies the global scale 
and identifies the corrections coming from $\ell^2$ normalization. Additionally, there are significant differences in the hypotheses.
\Cref{t:main1} is stated for any matrix with variance structure \eqref{e:stochasticvar}, and applies to any eigenvector.
The result \eqref{globalfluctuations} requires that $\u_i$ be a bulk eigenvector and $H$ has identically distributed entries (the diagonal and off-diagonal entries may take different distributions). %Both results assume all moments are finite.

Regarding \Cref{t:main2}, we observe that \eqref{eth} gives the same bounds for sets $I$ satisfying $\vert I\vert \ge c N$ and $A$ taken from a specific type of observable. However, when $\vert I\vert \ll N$, \Cref{t:main2} gives a stronger bound for this choice of observable. This bound is optimal up to the $N^\eps$ factor and thus captures the true size of the fluctuations for QUE on all scales.

\subsection{Proof strategy}
The proof of \Cref{t:main1} is based on the dynamical approach to random matrix theory, which was introduced in \cite{erdos2011universality}. It consists of the following three steps; our novel contributions come in the second and third steps.

	\paragraph{{Step 1: Local law and rigidity.}} To begin our analysis, we require \emph{a priori} estimates on the eigenvalues and eigenvectors of Wigner matrices. The \emph{local semicircle law} controls the resolvent of a generalized Wigner matrix and identifies the spectral distribution on all scales asymptotically larger than the 	typical eigenvalue spacing. An important consequence is the \emph{rigidity} of eigenvalues, first derived in \cite{erdos2012rigidity}. It states that with high probability, all eigenvalues are simultaneously close to their classical (deterministic) locations. Another useful corollary is the high-probability delocalization bound $\| \u \|^2_\infty \le N^{-1 + \eps}$.

	\paragraph{{Step 2: Relaxation by Dyson Brownian motion.}}We next use the \emph{a priori} estimates on generalized Wigner matrices from the first step to study such matrices as they are evolved through time by a matrix Ornstein--Uhlenbeck process. The moments of the QUE observables in our main result belong to a class of \emph{perfect matching observables} of eigenvectors. The time evolution of the matrix induces a parabolic differential equation with random coefficients on these observables, the \emph{eigenvector moment flow}. 
	
	We aim to use the maximum principle to show that the moments of the QUE observables at times $t=N^{-\tau}$ for small $\tau>0$ relax to their equilibrium state, the corresponding Gaussian moments.
However, the estimates from the first step alone do not suffice to complete this argument. It is necessary to see that the eigenvectors of the matrix process decorrelate in a quantitative way at times $t = N^{-\tau}$, so that their correlation functions are much smaller than the bound obtained from $\ell^\infty$ delocalization. For this, we require a novel four-point decorrelation estimate (see \eqref{decor} below). To prove it, we introduce two new sets of symmetrized observables and study their relaxation to equilibrium using the maximum principle. We also remark that in the final analysis, $\tau$ must be chosen as a function of both $n$ and the parameter $\eps$ in the bound $| I | \le N^{1-\eps}$, so that $n\tau$ is small relative to $\eps$.

	\paragraph{{Step 3: Comparison.}} 
The previous step shows that the $n$th moment of the quantity in \Cref{t:main1} converges to the $n$th moment of a Gaussian for all matrices of the form $\sqrt{1 - t }  H + \sqrt{t} \GOE_N$, if $t \ll 1$ is chosen appropriately.
Using the fact that matrices with small additive Gaussian noise are ``dense'' in the space of all generalized Wigner matrices, we complete the argument by showing this convergence extends to any matrix satisfying \Cref{d:wigner}.
	
	Our approach here is based on the four-moment matching technique \cite{tao2010random}. Given a generalized Wigner matrix $H$, we use a standard result to exhibit a generalized Wigner matrix $\tilde H$ and time $t\ll 1$ such that $H$ and $\hat H = \sqrt{1 - t } \tilde H + \sqrt{t} \GOE_N$ have entries whose first four moments match, up to a small error. If the normalized eigenvector sums in \Cref{t:main1} were regular with respect to small perturbations in the matrix entries, it would be straightforward to transfer the moment convergence from $\hat H$ to $H$. However, because derivatives of eigenvector entries with respect to the matrix entries are highly singular,  this cannot be done directly. We instead construct a smoothed analogue of these sums, which closely approximates them but has better regularity. The derivatives of this substitute observable are controlled using \Cref{t:main2}, and the resulting estimates suffice to complete the moment matching. We discuss the details of the smoothed observable, and the proof of \Cref{t:main2}, below.

	\paragraph{} 
	
	We now discuss the details of the second and third steps. Our primary technical contribution in the second step is the proof of eigenvector correlation estimates for pairwise distinct entries, such as 
	
	\begin{equation}\label{decor}
	\E \left[ N^2\scp{\q_{\alpha_1}}{\u_k^s}\scp{\q_{\alpha_2}}{\u_k^s}\scp{\q_{\alpha_3}}{\u_k^s}\scp{\q_{\alpha_4}}{\u_k^s} \right] = \O{\frac{N^{\delta}}{N s^{3/2} }},
	\end{equation}
	for eigenvectors of the Ornstein--Uhlenbeck matrix dynamics at times $s \gg N^{-1/3}$ (see \Cref{p:emf1}). For times $s \approx 1$, this bound is order $\O{N^{-1 + \delta}}$, which improves the  $\O{N^{ \delta}}$ estimate given by $\ell^\infty$ delocalization bounds, and therefore captures the decorrelating influence of the dynamics. 
	A major obstacle to obtaining bounds like \eqref{decor} is that the evolution equations for general products of eigenvector entries under the eigenvector moment flow are not represented by a positive $L^\infty$ operator. Therefore, the maximum principle cannot be applied to deduce short-time equilibration, as it was in \cite{bourgade2013eigenvector} for products of even moments of entries. 
	Further, while general product observables were analyzed using the energy method in \cite{marcinek2020high}, the convergence rate given there is not strong enough to deduce \eqref{decor}.
	
	To prove eigenvector decorrelation, we introduce two systems of \emph{symmetrized} eigenvector moment observables containing sums of products of eigenvector entries averaged over symmetry groups. 
	These symmetries are chosen so that the observables satisfy parabolic evolutions, which permits the use of the maximum principle to give a short proof of optimal decorrelation bounds for large times. The first set of moment observables, given in Definition \ref{d:momobs1}, can be seen as an extension to pairwise distinct indices of the observables from \cite{bourgade2013eigenvector} and follow the same flow. The second set, given in Definition \ref{d:momobs2}, are signed and are used to partially relax the symmetries in the first set. They follow the Fermionic eigenvector moment flow introduced in \cite{benigni2021fermionic}. A limitation of this approach is that we cannot in general access expectations of individual products of eigenvector projections, but only expectations of averages over certain symmetries; the bound  \eqref{decor} is an exceptional case where control of a single product is possible. Fortunately, these symmetrized bounds suffice to complete our analysis of the dynamics of the perfect matching observables, and therefore complete the dynamical step of our argument.
	
	 We note that we do not use the ETH proved in \cite{erdHos2020eigenstate}, but instead rely entirely on symmetrized moment observables following parabolic equations and QUE bounds proved dynamically in \cites{bourgade2018random, benigni2021fermionic}. For this reason, we believe that our dynamical analysis can be refined to admit a much more general class of initial data by following similar generalizations for other eigenvector observables \cites{marcinek2020high,aggarwal2020eigenvector,bourgade2017eigenvector,bourgade2018random}. We also remark that the proof of short-time equilibration does not use the intricate energy method from \cite{marcinek2020high}, which played a crucial role in the proof of  \eqref{globalfluctuations} in \cite{cipolloni2021normal}, but only the maximum principle.
	 
	 \paragraph{} For the third step of the argument, the main task is to exhibit an appropriate regularized observable to which four-moment matching can be applied. We begin by reviewing the regularization of individual eigenvector entries, which was carried out in \cite{knowles2013eigenvector}. 
Let $\u_\ell$ be an eigenvector with corresponding eigenvalue $\lambda_\ell$. For any $\eta >0$, we have
\beq
| \scp{\q_\alpha}{\u_\ell} |^2 = 
\frac{\eta}{\pi} \int_{\R} \frac{| \scp{\q_\alpha}{\u_\ell} |^2\, \d E}{(E - \lambda_\ell)^2 + \eta^2 } \approx 
\frac{\eta}{\pi} \int_{I} \frac{| \scp{\q_\alpha}{\u_\ell} |^2\, \d E}{(E - \lambda_\ell)^2 + \eta^2 } \label{g1},
\eeq
where $I$ is an interval with length $|I| \gg \eta$ centered at $\lambda_\ell$, and we used that the Poisson kernel integrates to unity. 
Recall that the resolvent for $H$ is defined by $G = (H - z)^{-1}$ for $z \in \mathbb{H}$. Now suppose that $\eta$ is taken smaller than the typical size of the eigenvalue gap $\lambda_{\ell+1} - \lambda_\ell$. Then, with probability $1 - o(1)$,
\beq
\frac{\eta}{\pi} \int_{I} \frac{| \scp{\q_\alpha}{\u_\ell} |^2\, \d E}{(E - \lambda_\ell)^2 + \eta^2 } 
\approx
\frac{\eta}{\pi} \int_{I} \sum_j \frac{| \scp{\q_\alpha}{\u_j} |^2\, \d E}{(E - \lambda_j)^2 + \eta^2 }  =
\int_I \Im \scp{\q_\alpha}{G(E + \I \eta)\q_\alpha}\, \d E,\label{g2}
\eeq
where we used the fact that $\lambda_j$ is typically outside of $I$ for $j\neq \ell$ and therefore the corresponding term in the sum is negligible. 
Combining \eqref{g1} and \eqref{g2} shows that $| \u_\ell(\alpha)|^2$ can be written in terms of entries of $G$ evaluated at points with imaginary part slightly below the typical eigenvalue spacing. 
These can be effectively controlled using the estimates from the first step of the outline, and because derivatives of resolvents can be written in terms of their entries using the identity $\partial_{H_{ij}} G_{ab} = -G_{ai}G_{jb} - G_{aj}G_{ib}$, their derivatives are also straightforward to control. 

For the normalized sums of eigenvector entries $ | I |^{-1/2} \sum_{\alpha \in I} |\scp{\q_\alpha}{\u_i}|^2$ appearing in \eqref{t:main1}, it is natural to regularize the sum by replacing each term by its regularization in \eqref{g2}. However, it is not permissible to bound each term (or its derivative) in the sum individually. In order to capture the right order, it is necessary to see cancellation between each $\scp{\q_\alpha}{G\q_\alpha}$. 
For this, we note by \eqref{g1} that high-probability control on $\sum_{\alpha \in I } \scp{\q_\alpha}{G\q_\alpha}$ is equivalent to high probability control on $\sum_{\alpha \in I} | \scp{\q_\alpha}{\u_i}|^2$, which is provided by \Cref{t:main2}. Given this control, the standard four-moment matching argument can be applied to finish the proof of \Cref{t:main1}.

\paragraph{} Finally, we discuss the proof of \Cref{t:main2}, which also proceeds by the dynamical approach to random matrix theory.
Since the theorem was proved for matrices with Gaussian noise in \cite{benigni2021fermionic},
it suffices to complete the comparison step.
The argument is rather technical, so we mention only a few key points here and defer a detailed outline to the beginning of \Cref{s:quereg}.
There are three main ingredients. 
The first is the observation that if we define, for $k\neq \ell$,
\begin{equation}
p_{kk}  = \sum_{\alpha \in I } \scp{\q_\alpha}{\u_k}^2 - \frac{| I | }{N}, \qquad 
p_{k\ell} = \sum_{\alpha \in I} \scp{\q_\alpha}{\u_k}\scp{\q_\alpha}{\u_\ell},
\end{equation}
then
\begin{equation}\label{key2}
 \sum_{i,j }^N \frac{ \eta_k\eta_\ell   p^2_{ij} }{ \left( ( \lambda_i - E_1)^2 + \eta^2_k\right) \left( ( \lambda_j - E_2)^2 + \eta^2_\ell  \right) }
\end{equation}
may be written as a sum of products of resolvent entries; the exact expression is complicated and given as \eqref{key} below.\footnote{This expression is analogous to an observable used in \cite{cipolloni2021normal}; see \cite{cipolloni2021normal}*{(1.17)}.} This provides a single simultaneous regularization of all the quantities in \Cref{t:main2}, and by integrating over small intervals centered at $\lambda_k$ and $\lambda_\ell$, one can recover the individual $p_{k\ell}^2$ for any $k, \ell \in \unn{1}{N}$ .

The second main idea is bootstrapping on the scale of $I$. As with the observable \eqref{g2}, control over the $p_{k\ell}$ is necessary to control \eqref{key2} and its derivatives. But there is an apparent circularity here, since \Cref{t:main2} is what we intend to use to bound the $p_{k\ell}$, yet we need high-probability control over $p_{k\ell}$ in its proof. The way out is to induct on the size of $I$, increasing the range of $I$ over which we have control by a factor of $N^{\sigma}$ at each step for a small, fixed $\sigma>0$. By $\ell^\infty$ eigenvector delocalization bounds, if the $p_{k\ell}$ are bounded by $N^{\omega}$ for sets $I$ satisfying $| I | \le N^{\kappa}$ and all $k, \ell \in \unn{1}{N}$, then we have estimates of order $N^{\omega + \sigma}$ for sets $| I | \le N^{\kappa + \sigma}$. If $\sigma$ is small enough, the $N^{\kappa + \sigma}$ estimate and regularization \eqref{key2} are sufficient to complete the comparison and prove the improved $N^{\omega}$ bound holds for $| I | \le N^{\kappa + \sigma}$, justifying the induction step. For the base case, we take $I$ to be constant size, where the bounds in \Cref{t:main2} hold by eigenvector delocalization.

The third main idea is that, since we only need to bound the $p^2_{k\ell}$ from above, we do not need to restrict the comparison to the probability $1 - N^{-c}$ set where the regularization \eqref{key2} holds exactly, where eigenvalues abnormally close to $\lambda_k$ and $\lambda_\ell$ are excluded. Due to the positivity of the terms in \eqref{key2}, after integrating over the appropriate intervals it dominates $p^2_{k\ell}$ deterministically (compare \eqref{g2}). This is what permits the stronger $N^{-D}$ error in \Cref{t:main2}. The analogous observation for the regularization of a single eigenvector entry was used in \cite{benigni2020optimal} to prove optimal $\ell^\infty$ delocalization bounds.

We remark that it would not be possible to use the observable \eqref{key2} for the proof of \Cref{t:main1}, since it only identifies the absolute value of $p_{kk}$, and not its sign.

\subsection{Outline of the paper} \Cref{s:prem} states preliminary estimates from previous works. In \Cref{sec:dynamics}, we prove short-time relaxation of the eigenvector moment flow for perfect matching observables, assuming a key eigenvector decorrelation estimate. In \Cref{s:decorrelation}, we prove this estimate. 
\Cref{s:quereg} constructs a regularized observable that is necessary for the comparison step of our argument.
\Cref{s:t2} completes the proof of \Cref{t:main2}, and \Cref{s:t1} defines another regularized observable and completes the proof of \Cref{t:main1}.
\Cref{s:a1} provides bounds on the observables of the previous sections. 
\Cref{a:genmoments} provides a more general viewpoint on the symmetrized moment observables.
\paragraph{Acknowledgments.} We thank Paul Bourgade for helpful comments. We also thank Xiaoyu Xie for notifying us of several misprints in an earlier version of this work.

\section{Preliminaries}\label{s:prem}

For brevity, we write all statements and proofs for the real symmetric case of \Cref{d:wigner} only. The complex Hermitian case differs only in notation. The origin of the different $\beta$ in the Hermitian version of \Cref{t:main1} is the Hermitian analogue of \Cref{t:flow} quoted below, which can be found at the same reference given there.

%\ber
%For ease of notation, we develop the proof for $(\q_\alpha)_\alpha=(\mathbf{e}_\alpha)_\alpha$ to be the standard basis but we note that the proof stays the same as the isotropic local law $\scp{\q_\alpha}{G\q_\beta}$ and the entrywise local law $G_{\alpha\beta}$ give the same error bound for $\q_\alpha \perp \q_\beta$. [Say more here: we should cover changes to both dynamics and comparison in detail.]\eer

\subsection{Notation}

We write $X \ll Y$ if there exists a constant $c > 0 $ such that $N^c |X| \le Y$. %Constants in this paper may depend on the constant $c>0$ implicit in the claim $X \ll Y$, but we suppress this in the notation. 
We also write $X = \O{Y}$ if there exists a constant $C> 0$ such that $|X| \le C Y$. The constant $C$ will not depend on other parameters, unless this is explicitly mentioned. Set $\mathbb{H} = \{ z \in \C : \Im z > 0 \}$ and $\S^{N-1}=\{ \x \in \R^N : \|\x\|_2 = 1 \}$. Although the constants in all results below may depend on the constants in \Cref{d:wigner}, we will not notate this dependence explicitly.

Let $\matn$ be the set of $N\times N$ real symmetric matrices. We index the eigenvalues of matrices $M\in \matn$ in increasing order, so that $\lambda_1 \le \lambda_2 \le \dots \le \lambda_N$. 
For $z\in \mathbb{H}$, the resolvent of $M\in \matn$ is given by $G(z) = ( M - z \Id)^{-1}$, and the
Stieltjes transform of $M$ is
\begin{align}
m_N (z) = \Tr G(z) =  \frac{1}{N} \sum_i \frac{1}{ \lambda_i - z}.
\end{align}
The resolvent has the spectral decomposition
\beq\label{eq:defsresolv}
G (z) =
\sum_{i=1}^N
\frac{ \u_i  \u^*_i }{\lambda_i-z},
\eeq
where $\u_i$ is the eigenvector corresponding to the eigenvalue $\lambda_i$ of $M$ such that $\| \u_i \|_2 = 1$; we fix the sign of $\u_i$ arbitrarily by demanding that $\u_i (1) \ge 0$.

The semicircle distribution and its Stieltjes transform are
\begin{align}
\scrho (E) = \frac{\sqrt{ (4 - E^2)_+ }}{2 \pi }\, \d E, \qquad \msc (z) = \int_{\R} \frac{\scrho (x)\, \d x}{ x - z },
\end{align}
for $E \in \R$ and $z \in \mathbb{H}$.
For $i\in\unn{1}{N}$, we denote the $i$th $N$-quantile of the semicircle distribution by $\gamma_i$ and define it implicitly by
\beq\label{e:classical}
\frac{i}{N} = \int_{-2}^{\gamma_i} \scrho (x) \, \d x.
\eeq

 %We first define the resolvent for $z\in\mathbb{C}_+$,
%\beq\label{eq:defsresolv}
%G^s(z)
%=
%\sum_{k=1}^N
%\frac{\vert\u_k^s\rangle \langle\u_k^s\vert}{\lambda_k(s)-z}.
%\eeq
%Note that $G^s$ has the same distribution as the resolvent of $H_s$ defined by $(H_s- z \Id)^{-1}$.

\subsection{Dyson Brownian motion}\label{s:dbm}

The $N\times N$ real symmetric Dyson Brownian motion is the stochastic process $(H_s)_{0\le s \le 1}$  defined by 
\beq\label{e:dysondyn}
 \d H_s =  \frac{1}{\sqrt{N}} \d B(s) - \frac{1}{2} H_s \d s,
\eeq
 where $B(s) \in \matn$ is a symmetric matrix such that $B_{ij}(s)$ and $B_{ii}(s)/\sqrt{2}$ are mutually independent standard Brownian motions for $1 \le i < j\le N$.

The process $(H_s)_{0\le s \le 1}$ has the same distribution as $\left( \u^s \bm \lambda(s) (\u^s)^* \right)_{0 \le s \le 1}$, where $\bm \lambda(s) = ( \lambda_1(s), \dots, \lambda_N(s))\in \R^N$ and $\u^s = \left( \u^s_1, \dots, \u^s_N \right) \in \R^{N \times N}$ solve
\begin{align}
\label{eq:dysonval}\d \lambda_k(s) &= \frac{\d \wt{B}_{kk}(s)}{\sqrt{N}}
+\left(
	\frac{1}{N}
	\sum_{\ell\neq k}
	\frac{1}{\lambda_k(s)-\lambda_\ell(s)}
	 -\frac{\lambda_k(s)}{2}
\right)\d s, \\
\label{eq:dysonvect}\d \u^s_k &= \frac{1}{\sqrt{N}}\sum_{\ell\neq k}\frac{\d \wt{B}_{k\ell}(s)}{\lambda_k(s)-\lambda_\ell(s)}\u_l^s
-\frac{1}{2N}
\sum_{\ell\neq k}
\frac{\d t}{(\lambda_k(s)-\lambda_\ell(s))^2}\u_k^s,
\end{align}
with initial data $H_0 =\u^0\bm  \lambda(0) (\u^0)^*$, and $\wt{B}$ has the same distribution as $B$ \cite{bourgade2013eigenvector}*{Theorem 2.3}. 
%Using these SDEs, we define the stochastic processes $\bm{\lambda} = ( \bm{\lambda} (s) )_{0 \le s \le 1}$ and $\u = ( \u (s) )_{0 \le s \le 1}$.
We let $m_N^s(z)$ be the Stieltjes transform of $H_s$ and $G^s(z)$ be its resolvent.

\subsection{Initial estimates}

%Let $\mathbb{S}^{N-1} \subset \mathbb R^N$ be the set of vectors $\q \in \R^N$ such that $\| \q \|_2 =1$. 
For any $\omega >0$, we define
\beq \D_\omega  =  \{ z = E + \I\eta \in \mathbb C : |E| < \omega^{-1}, N^{-1+ \omega} \le \eta \le \omega^{-1}  \}. 
\eeq
Given $k \neq \ell \in \unn{1}{N}$, $I \subset \unn{1}{N}$, and a family of orthogonal vectors ${\q}=(\q_\alpha)_{\alpha\in I}$ with $\q_\alpha \in \S^{N-1}$ for all $\alpha \in I$, we set
\begin{equation}\label{pdef}
p_{kk} (s) = \sum_{\alpha \in I } \scp{\q_\alpha}{\u^s_k}^2 - \frac{| I | }{N}, \qquad 
p_{k\ell}(s) = \sum_{\alpha \in I} \scp{\q_\alpha}{\u^s_k}\scp{\q_\alpha}{\u^s_\ell}.
\end{equation}
We use the shorthand $p_{k\ell} = p_{k\ell}(0)$.
We also define the error parameter 
\begin{equation}\label{psis}
\Psi(s) = \frac{|I|}{N^{3/2} s^2}  + \sqrt{\frac{|I|}{N^2 s^3}}.
\end{equation}

\bel[\cite{bourgade2013eigenvector}*{Lemma 4.2}, 
\cite{benigni2021fermionic}*{Proposition 3.5}]\label{l:goodset}
Fix  $D, \omega >  0$, $I \subset \unn{1}{N}$, and a family of orthogonal vectors ${\q}=(\q_\alpha)_{\alpha\in I}$ in $\S^{N-1}$.
Let $H$ be a $N\times N$ generalized Wigner matrix, and define $H_s$, $G^s$, $\u^s_k$,  $m^s_N(z)$, $p_{kk}(s)$, and $p_{k\ell}(s)$ as above. Let $\mu$ be the measure on the space of joint eigenvalue and eigenvector trajectories  $(\bm \lambda(s), \u(s))_{0 \le s \le 1}$ induced by the Dyson Brownian Motion $(H_s)_{0\le s \le 1}$.
%, and denote the corresponding probability measure by $\P_{\mu}$.
%the event $\mathcal{A}(\q,\omega,\tau)$ defined below occurs with overwhelming probability uniformly in $\mathbf{q}$, in other words, 
Then there exists $N_0(D)> 0$ such that for $N \ge N_0$ we have
\beq\label{e:exphigh}
%\inf_{\q \in\mathbb{S}^{N-1}}
\P_\mu \left(
	\mathcal{A}(\omega,\q, I)
\right)
\geqslant 1- N^{-D},
\eeq
where $\mathcal{A}(\omega,\q, I)$ is the set of trajectories $(\bm \lambda(s), \u(s))_{0 \le s \le 1}$  where all of the following statements hold.

\begin{enumerate}
\item For all $z=E+\I \eta\in \D_\omega$,
\beq\label{e:sclaw}
\begin{gathered}
\sup_{s\in [0,1]}
\left\vert
	m^s_N(z)-\msc(z)
\right\vert
\leqslant
\frac{N^\omega}{N\eta}, \qquad
\sup_{i,j \in \unn{1}{N}} \sup_{s\in[0,1]} \left| G^s_{ij}(z) - \one_{i = j} \msc(z) \right| \le N^{\omega}\left(
	\sqrt{\frac{\Im \msc(z)}{N\eta}}
	+
	\frac{1}{N\eta}
\right),\\
\sup_{\alpha,\beta \in I}
\sup_{s\in[0,1]}\left\vert \scp{\q_\alpha}{G^s(z)\q_\beta}-\one_{\alpha=\beta}\msc(z)\right\vert\le N^{\omega}\left(
\sqrt{\frac{\Im \msc(z)}{N\eta}}
+
\frac{1}{N\eta}
\right).
\end{gathered}
\eeq
\item For all $k\in\unn{1}{N}$,
\beq\label{e:rigidity}
\begin{gathered}
\sup_{s\in[0,1]}
\left\vert
	\lambda_k(s)-\gamma_k
\right\vert
\leqslant
N^{-2/3+\omega}\left[\min(k, N - k +1) \right]^{-1/3}, \qquad
\sup_{s\in[0,1]} \sup_{j \in \unn{1}{N}}
 \u_k^s(j) ^2
\leqslant
N^{-1+\omega},\\
\sup_{s\in[0,1]}\sup_{\alpha\in I}\,\scp{\q_\alpha}{\u_k^s}^2\leqslant N^{-1+\omega}.
\end{gathered}
\eeq

\item
For any $s \in [N^{-1/3 + \omega}, 1]$, 
\begin{equation} \label{e:que}
\sup_{k,l \in \unn{1}{N}} | p_{k\ell}(s) | \le N^\omega \Psi(s).
\end{equation}
\end{enumerate}
\eel

\section{Relaxation by Dyson Brownian motion}\label{sec:dynamics}

\subsection{Eigenvector moment flow}\label{subsec:emf}

Let $H$ be a generalized Wigner matrix, and let $H_s$ be Dyson Brownian motion \eqref{eq:dysonval} with initial condition $H_0 = H$. 
It was proved in \cite{bourgade2018random} that certain products of the $p_{k \ell}(s)$ observables for $H_s$ obey parabolic equations, which may be described in terms of particle configurations on $\unn{1}{N}$. We begin by reviewing these dynamics.

Let $\{ (i_1, j_1), \dots , (i_m, j_m)\}$ be an index set with pairwise distinct $i_k \in \llbracket 1, N \rrbracket$ and positive $j_k \in \N$. We associate with such a set the vector $\bm \xi = (\xi_1, \xi_2, \dots, \xi_N) \in \mathbb N^N$ with $\xi_{i_k} = j_k$ for $1 \le k \le m$ and $\xi_p =0$ for $p \notin \{ i_1, \dots, i_m\}$. We view $\bm \xi$ as a particle configuration on $\unn{1}{N}$, with $j_k$ particles at site $i_k$ for all $k$ and zero particles on the sites not in $\{ i_1, \dots, i_m\}$. 
The configuration $\bm \xi^{ij}$ is defined by moving one particle in $\bm \xi$ from $i$ to $j$, if this is possible. More precisely, if $i \ne j$ and $\xi_i >0$, then $\xi^{ij}_k$ equals $\xi_k + 1$ if $k=j$, $\xi_k - 1$ if $k=i$, and  $\xi_k$ if $k \notin \{ i, j \}$. When $\xi_i =0$, set $\bm \xi^{ij} = \bm \xi$.

For any configuration $\bm{\xi}$, we define the vertex set
\begin{equation}\label{mathcalV}
\mathcal{V}_{\bm \xi} = \left\{ (i,a) : 1 \le i \le n , 1 \le a \le 2 \xi_i \right\}.
\end{equation}
We let $\mathcal G_{\bm \xi}$ be the set of perfect matchings of the complete graph on $\mathcal V_{\bm \xi}$. For any edge 
$e  = \{ (i_1, a_1), (i_2, a_2) \}$ of a graph $G \in \mathcal{G}_{\bm \xi}$, we set $p(e) = p_{i_1 i_2}(s)$ and $P_s(G) = \prod_{e \in \mathcal E (G) } p(e)$, where $\mathcal E (G)$ denotes the edge set of $G$.

\begin{figure}[!ht]
	\centering
	\begin{subfigure}[t]{.5\textwidth}
		\centering
		\includegraphics[width=.8\linewidth]{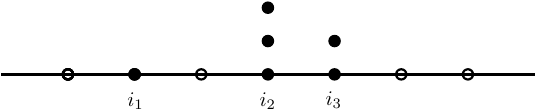}
		\caption{A configuration $\bm{\xi}$ with $\mathcal{N}(\bm{\xi})=6$ particles.}
	\end{subfigure}%
	\hspace{-0em}\begin{subfigure}[t]{.5\linewidth}
		\centering
		\includegraphics[width=.8\linewidth]{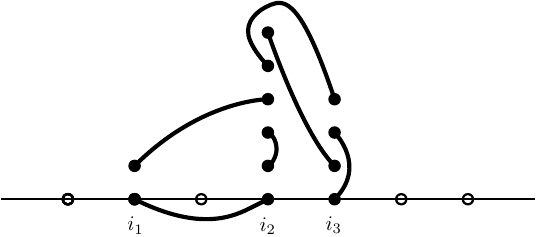}
		\captionof{figure}{\raggedright An example of a perfect matching \hbox{$G\in\mathcal{G}_{\bm{\xi}}$} with \hbox{$P(G)=p_{i_1i_2}^2p_{i_2i_2}p_{i_2i_3}^2p_{i_3i_3}$}.}
	\end{subfigure}
\end{figure}

We now define the perfect matching observables $f_s\colon \N^N \rightarrow \R$.
Given $I \subset \unn{1}{N}$ and a path $\bm \lambda = (\bm \lambda(s))_{0\le s \le1}$ of eigenvalues of $H_s$, we set
\beq
\label{eq:deffs}
f_{s}(\bm{\xi})  = f_{\bm \lambda,s}(\bm{\xi}) 
=
\frac{1}{\M(\bm{\xi})}
\E\left[
	\sum_{G \in \mathcal{G}_{\bm \xi}} P_s(G)
	\;\middle\vert\;
	\bm{\lambda}
\right],
\quad\text{where}\quad
\M(\bm{\xi})
=
\prod_{k=1}^N (2\xi_k -1 )!!.
\eeq
The eigenvectors in the definition of $P_s(G)$ are defined using the eigenvectors $\u^s$ of $H_s$.
%The normalization factor $\mathcal M(\bm \xi)$  is chosen because it equals the cardinality of $\mathcal{G}_{\bm \xi}$.

The time evolution of the observable $f_s(\bm{\xi})$ is given by the parabolic equation 
% known as the eigenvector moment flow, which is 
in the following theorem.
\bet[\cite{bourgade2018random}*{Theorem 2.6}]\label{t:flow}
For all $s\in (0,1)$, the perfect matching observable $f_s$ defined in \eqref{eq:deffs} satisfies the equation
\beq\label{eq:emf}
\partial_s f_s(\bm{\xi})
=
\sum_{k\neq \ell}
2\xi_k(1+2\xi_\ell)
\frac{f_s(\bm{\xi}^{k,\ell})-f_s(\bm{\xi})}{N(\lambda_k(s)-\lambda_\ell(s))^2}.
\eeq
\eet

Finally, we define $\mathcal A_1 (\omega,\q, I)$ to be the set of paths $\bm \lambda = (\bm \lambda(s))_{0 \le s \le 1}$ such that the statements  \eqref{e:sclaw}, \eqref{e:rigidity}, and \eqref{e:que} hold with probability at least 
$ 1 - N^{-D}$
with respect to the marginal distribution on paths $\bm \lambda$ induced by the measure $\P_\mu$ defined in \Cref{l:goodset}. 
By \Cref{l:goodset} and Fubini's theorem,
\beq\label{e:Croothigh}
\P_\mu\left( \mathcal A_1 (\omega,\q, I) \right) \ge 1 - N^{-D}.
\eeq

\subsection{Gaussian divisible ensembles}
We now use the flow \eqref{eq:emf} to obtain a bound on the observable $f_s(\bm{\xi})$ for arbitrarily large configurations $\bm \xi$ and times $s$ close to $1$. To analyze the dynamics, we require the following decorrelation estimate, which is proved in \Cref{s:decorrelation}.
\bel\label{l:decorrelation}
Fix $D, \eps, \omega > 0$ and $n \in \N$, and let $I \subset \unn{1}{N}$ satisfy $N^\eps \le | I | \le N^{1 - \eps}$. Let ${\q}=(\q_\alpha)_{\alpha\in I}$ be orthogonal vectors with $\q_\alpha \in \S^{N-1}$ for all $\alpha \in I$. Let 
\beq
Q_s = \prod_{i=1}^{n} p_{a_i b_i}
\eeq
be a product of $n$ observables defined by \eqref{pdef} for some $a_i, b_i \in \unn{1}{N}$. 
Then there exist $C(D, n, \eps, \omega)>0$ and an event $\mathcal{A}_2(\eps,\omega,\q, I)\subset \mathcal{A}_1(\omega,\q,I)$ such that the following holds. We have $\P(\mathcal{A}_2(\eps, \omega,\q,I))\geqslant 1-N^{-D}$, and
for every $\bm \lambda \in \mathcal A_2$, $s \in [ N^{-1/3 + \eps}, 1]$, and $j,k\in \unn{1}{N}$,
\begin{multline}\label{e:keyestimate}
\left|\sum_{\alpha\neq\beta\in I} \E \left[ \scp{\q_\alpha}{\u^s_j} \scp{\q_\beta}{\u^s_k} \Im\scp{\q_\alpha}{G^s(z)\q_\beta} Q_s | \bm \lambda\right] \right|\\ \le C N^{n\omega}\Psi^{n}(z)\left(\frac{\vert I\vert^2}{N^2}\left(\frac{1}{\eta}+\frac{\sqrt{s}}{\eta^2}+\frac{N^{2\omega}}{s^{3/4}\sqrt{\eta}}\right)+\frac{\vert I\vert^{3/2} N^{2\omega}}{N^{3/2}\sqrt{\eta}}\right).
\end{multline}
\eel

\bep\label{p:highmom}
Fix $\eps >0$ and $n \in \N$. There exist $\theta (n, \eps ) >0$ and $N_0(n, \eps) >0$ for which the following holds.
For all $N \ge N_0$,
\beq\label{e:dynamicsclaim}
\sup_{I} \sup_{\bm \xi: | \bm \xi| = n} \sup_{t \in [ N^{-\theta}, 1]}   \left| \left(\frac{N}{\sqrt{|I|}}\right)^{n}\E\left[  f_{\bm \lambda,s}(\bm{\xi})  \right]   - 2^n\E\left[g^n\right]  \right| \le N^{-\theta} ,
\eeq
where the supremum on $I$ is taken over coordinate sets $I \subset \unn{1}{N}$ with $N^\eps \le | I | \le N^{1 - \eps}$. Here $g$ is a real Gaussian random variable with mean $0$ and variance $1$.
\eep

\begin{proof}

For $m \in \N$, we define $\ximax = \ximax^{(m)}(s)$ to be the maximizer among all configurations with $m$ particles for the perfect matching observable \eqref{eq:deffs}:
\beq
f_s(\ximax) 
=
\sup_{\bm{\xi}:\,\vert\bm{\xi}\vert=m}
f_s(\bm{\xi}).
\eeq
If the supremum is attained by multiple configurations $\bm \xi$, we select one arbitrarily, subject to the requirement that $\widetilde{\bm \xi} (s)$ remains piecewise constant in $s$.  We denote the number of particles in $\widetilde {\bm \xi}$ at site $i$ by $\widetilde \xi_i$. 
We let $(k_1,\dots,k_p)$ be the sites where $\ximax$ has at least one particle;
these are the indices $k_i \in \unn{1}{N}$ such that $\widetilde \xi_{k_i}>0$. We let the implicit constants in each occurrence of the $\mathcal O$ notation in this proof depend on $m$; however, they will be independent of all other parameters in the statement of the proposition.

We begin by deriving a differential inequality for $f_s(\ximax^{(m)}(s))$
for any number of particles $m \le n$. 
Let $\omega>0$ be a small parameter, which will be fixed later. 
Fix some path $\bm \lambda \in \mathcal A_2 (\eps, \omega,\q, I)$, which was defined in \Cref{l:decorrelation}. It is straightforward to show that there exists a countable subset $\mathcal C = \mathcal C( H_0, \bm \lambda) \subset [0,1]$ such that the continuous function $f_s(\ximax)$ is differentiable on $[0,1] \setminus \mathcal C$ (see the proof of \cite{aggarwal2020eigenvector}*{Lemma 6.10}). Using the flow \eqref{eq:emf}, we see that for any $\eta>0$ and $s \in [0,1]\setminus \mathcal C$, 
\begin{align}\label{eq:bounddyn0}
\partial_s f_s(\ximax)
&=
\sum_{i=1}^p \sum_{\ell\neq k_i}
2\widetilde\xi_{k_i}(1+2\widetilde\xi_{\ell})
\frac{f_s(\ximax^{k_i,\ell})-f_s(\ximax)}{N(\lambda_{k_i}(s)-\lambda_\ell(s))^2}\\
&\leqslant
\frac{2}{N\eta}
\sum_{i=1}^p \sum_{\ell\neq k_i}
\left(f_s(\ximax^{k_i,\ell})-f_s(\ximax)\right)\frac{\eta}{(\lambda_{k_i}(s)-\lambda_\ell(s))^2+\eta^2}\\
&=\frac{2}{\eta} \sum_{i=1}^p
\Im
\sum_{\ell\neq k_i}
\frac{f_s(\ximax^{k_i,\ell})}{N(\lambda_\ell(s)-z_{k_i})} 
 - 
 \frac{2}{\eta}\sum_{i=1}^p
\Im
\sum_{\ell\neq k_i}
\frac{f_s(\ximax) }{N(\lambda_\ell(s)-z_{k_i})}.\label{e:secondtermdynamics}
\end{align}
In the last line, we set $z_{k_i}=\lambda_{k_i}+\I\eta$. In the inequality, we used the fact that $f_s(\ximax^{k_i,\ell})\leqslant f_s(\ximax)$ by the definition of $\ximax$. We also used that $\widetilde\xi_\ell \ge 0$ for all $l \in \llbracket 1 , N \rrbracket$, and $\widetilde\xi_{k_i}>0$.

We bound the second term in \eqref{e:secondtermdynamics} by
\begin{align}
\sum_{i=1}^p
\Im
\sum_{\ell\neq k_i}
\frac{f_s(\ximax)}{N(\lambda_\ell(s) - z_{k_i})}
&=
f_s(\ximax)
\left(
\sum_{i=1}^p
\Im m_N^s(z_{k_i})
-\frac{1}{N\eta}\right)
\\
&= 
f_s(\ximax)\left(
\sum_{i=1}^p
\Im \msc(z_{k_i})
+
\O{\frac{N^\omega}{N\eta}}\right),\label{eq:bounddyn1}
\\
&=\label{e:314}
f_s(\ximax)\sum_{i=1}^p
\Im \msc(z_{k_i})
+ \O{ \frac{N^{(m+1)\omega} \Psi^m(s)}{N\eta}}.
\end{align}
We used \eqref{e:sclaw} in \eqref{eq:bounddyn1}, and \eqref{e:que} in the last equality.

For the first term in \eqref{e:secondtermdynamics}, we use \eqref{e:que} to see that
\beq\label{e:316}
\sum_{i=1}^p
\Im
\sum_{\ell\neq k_i}
\frac{f_s(\ximax^{k_i,\ell})}{N(\lambda_\ell(s)-z_{k_i})}
=
\sum_{i=1}^p
\Im
\sum_{\ell \neq k_1,\dots,k_p}
\frac{f_s(\ximax^{k_i,\ell})}{N(\lambda_\ell(s)-z_{k_i})}
+
\O{
\frac{ N^{n \omega} \Psi^m(s)}{N\eta}
}.
\eeq

For the first term in \eqref{e:316}, fix an index $i\in\unn{1}{p}$ and some $\ell \notin \{k_1,\dots,k_p\}$; note that $\ximax^{k_i,\ell}$ has a unique particle on site $\ell$. We divide the sum in the definition of $f_s(\ximax^{k_i,\ell})$ into three parts depending on the structure of the graph $\mathcal{G}_{\ximax^{k_i,\ell}}$: there is an edge between the two vertices corresponding to $\ell$, there are two edges from the vertices on $\ell$ connected to vertices on the same site $k_j$, and there are two edges from the vertices on $\ell$ going to different sites $k_{j_1}$ and $k_{j_2}$. Let $\mathcal{G}^{(1)}_{\ximax^{k_i,\ell}} \subset \mathcal{G}_{\ximax^{k_i,\ell}}$ be the subset of graphs corresponding to the first case, and define $\mathcal{G}^{(2)}_{\ximax^{k_i,\ell}} $ and $\mathcal{G}^{(3)}_{\ximax^{k_i,\ell}}$ similarly.

\begin{figure}[!ht]
	\centering
	\begin{subfigure}[t]{.3\textwidth}
		\centering
		\includegraphics[width=.8\linewidth]{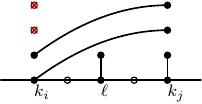}
		\caption{$G\in \mathcal{G}^{(1)}_{\ximax^{k_i,\ell}}$ }
	\end{subfigure}%
	\hspace{-0em}\begin{subfigure}[t]{.3\linewidth}
		\centering
		\includegraphics[width=.8\linewidth]{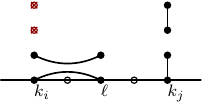}
		\captionof{figure}{$G\in \mathcal{G}^{(2)}_{\ximax^{k_i,\ell}}$}
	\end{subfigure}
	\hspace{-0em}\begin{subfigure}[t]{.3\linewidth}
		\centering
		\includegraphics[width=.8\linewidth]{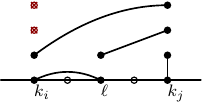}
		\captionof{figure}{$G\in \mathcal{G}^{(3)}_{\ximax^{k_i,\ell}}$}
\end{subfigure}
\caption{Examples of the three types of graphs.}
\end{figure}

We begin by considering the observable 
\beq
f^{(1)}_s(\ximax^{k_i,\ell}) = 
\frac{1}{ \mathcal M(\ximax^{k_i,\ell}) }
\sum_{ G \in\mathcal{G}^{(1)}_{\ximax^{k_i,\ell} }}  \E [ P_s(G) | \bm \lambda]
=  \frac{1}{ \mathcal{M}(\ximax \setminus \{ k_i \}) } 
\sum_{ G \in\mathcal{G}_{\ximax \setminus \{ k_i \}} }\E [ p_{\ell \ell} P_s(G) | \bm \lambda ],
\eeq
where $\ximax\setminus \{ k_i \}$ denotes the configuration $\ximax$ with a particle removed from site $k_i$. Note that we have $ \mathcal M(\ximax^{k_i,\ell})= \mathcal{M}(\ximax \setminus \{ k_i \})$ from the condition that $\ell\not\in\{k_1,\dots,k_p\}$ so that $(2\ximax^{k_i,\ell}(\ell)-1)!!=1$.  We have
\beq
\Im
\sum_{\ell \neq k_1,\dots,k_p}
\frac{f^{(1)}_s(\ximax^{k_i,\ell})}{N(\lambda_\ell(s)-z_{k_i})}
= \frac{1}{ \mathcal{M}(\ximax \setminus \{ k_i \})}
\sum_{ G \in\mathcal{G}_{\ximax \setminus \{ k_i \}} }
\E \left[
\sum_{\ell \neq k_1,\dots,k_p}
\frac{p_{\ell \ell}}{N ( \lambda_\ell - z_{k_i} )} P_s(G)
\;\middle\vert\; \bm \lambda \right],
\eeq
and using the definition of $p_{\ell \ell}$ we deduce
\begin{align}
\sum_{\ell \neq k_1,\dots,k_p}
\frac{p_{\ell \ell}}{N ( \lambda_\ell - z_{k_i} )}
&=
\sum_{\ell=1}^N
\frac{p_{\ell\ell}}{ N ( \lambda_\ell - z_{k_1})}
+ \O{\frac{  N^\omega \Psi(s)}{N\eta}}
\\
&= \frac{|I|}{N} \Im \left( \frac{1}{|I|} \sum_{\alpha \in I}  \langle \q_\alpha, G(z_{k_i} ) \q_\alpha \rangle - \msc(z_{k_i}) \right) + 
\O{\frac{  N^\omega \Psi(s)}{N\eta}}\\
&=
\O{ \frac{|I| N^\omega}{N \sqrt{N\eta} } }
+
\O{\frac{ N^\omega \Psi(s)}{N\eta}},
\end{align}
where the last equality follows from the local law \eqref{e:sclaw}. We conclude using \eqref{e:que} that 
\beq\label{f1}
\Im
\sum_{\ell \neq k_1,\dots,k_p}
\frac{f^{(1)}_s(\ximax^{k_i,\ell})}{N(\lambda_\ell(s)-z_{k_i})}
=
\O{
\frac{|I| N^{m\omega} \Psi^{m-1}(s)}{N \sqrt{N\eta} } + 
\frac{ N^{m\omega} \Psi^m(s)}{N\eta} }.
\eeq
We next consider
\begin{align}
f^{(2)}_s(\ximax^{k_i,\ell}) &= 
\frac{1}{ \mathcal M(\ximax^{k_i,\ell}) }
\sum_{ G \in\mathcal{G}^{(2)}_{\ximax^{k_i,\ell} }}  \E [ P_s(G) | \bm \lambda]\\
&= \label{combo1}
\sum_{j=1}^p \frac{2\ximax^{k_i,\ell}(k_j) (2\ximax^{k_i,\ell}(k_j) -1  ) }{ \mathcal{M} (\ximax^{k_i,\ell}) } 
\E \left[ p^2_{k_j\ell}
\sum_{ G \in\mathcal{G}_{\ximax \setminus \{ k_i, k_j \}} }
 P_s(G)  \;\middle\vert\;  \bm \lambda \right]
\\
&=\label{combo3} \sum_{j=1}^p  2\ximax^{k_i,\ell}(k_j) 
\E \left[
p^2_{k_j, \ell} \frac{1}{\mathcal{M} ( \ximax \setminus \{k_i,k_j \} )}
\sum_{G \in\mathcal{G}_{\ximax \setminus \{ k_i, k_j \}}}
P_s(G)
\;\middle\vert\; 
\bm \lambda
\right].
\end{align}
The combinatorial factor in \eqref{combo1} comes from counting the number of ways two edges from site $\ell$ can connect to vertices on site $k_j$. Motivated by \eqref{e:316}, we want to understand
\begin{equation}
\Im \sum_{\ell \neq k_1, \dots , k_p} 
\frac{p^2_{k_j \ell} }{N ( \lambda_\ell - z_{k_i}) }
=
\Im \sum_{\ell=1}^N
\frac{p^2_{k_j \ell} }{N ( \lambda_\ell - z_{k_i}) }
+
\O{ \frac{  N^{2\omega} \Psi^2(s) }{N\eta}}.
\end{equation}
This can be written in terms of the resolvent as
\begin{align}
\Im \sum_{\ell =1}^N \frac{p^2_{k_j\ell}}{N (\lambda_\ell - z_{k_1} )}
=& \Im \sum_{\ell=1}^N \sum_{\alpha, \beta \in I } \frac{ \scp{\q_\alpha}{\u_{k_j}} \scp{\q_\beta}{\u_{k_j}} \scp{\q_\alpha}{\u_{\ell}} \scp{\q_\beta}{\u_{\ell}} }{N(\lambda_\ell - z_{k_i})}
\\
=&\frac{1}{N}\sum_{\alpha, \beta \in I }  \scp{\q_\alpha}{\u_{k_j}} \scp{\q_\beta}{\u_{k_j}} \Im \scp{\q_\alpha}{G(z_{k_i})\q_\beta}
\\
=&
\frac{1}{N} \sum_{\alpha \in I} \scp{\q_\alpha}{\u_{k_j}}^2  \Im \scp{\q_\alpha}{G(z_{k_i})\q_\alpha}
+ \frac{1}{N} \sum_{\alpha \neq \beta} \scp{\q_\alpha}{\u_{k_j}} \scp{\q_\beta}{\u_{k_j}}
 \Im \scp{\q_\alpha}{G(z_{k_i})\q_\beta}\\
=&\label{blahr}
\frac{|I|}{N^2} \Im \msc(z_{k_i})
+ \frac{1}{N} \Im \msc(z_{k_i}) p_{k_j k_j}\\
+& \frac{1}{N} \sum_{\alpha \neq \beta \in I} \scp{\q_\alpha}{\u_{k_j}} \scp{\q_\beta}{\u_{k_j}}
 \Im \scp{\q_\alpha}{G(z_{k_i})\q_\beta}
+ \O{ \frac{N^{2\omega}|I|}{N^2 \sqrt{N\eta}}},
\end{align}
where in \eqref{blahr} we used \eqref{e:sclaw}. Combining this with \eqref{combo3} gives 
\begin{multline}\label{f2}
\Im
\sum_{\ell \neq k_1,\dots,k_p}
\frac{f_s^{(2)}(\ximax^{k_i,\ell})}{N(\lambda_\ell(s)-z_{k_i})}
= \frac{|I|}{N^2} \Im \msc(z_{k_i}) \sum_{j=1}^p 
2\left( \ximax \setminus \{ k_i \} \right)(k_j) f_s(\ximax \setminus \{k_i, k_j\} )
\\
+
 \O{\frac{ N^{m \omega} \vert I\vert^{3/2}\Psi^{m-2}(s)}{N^{5/2} \sqrt{\eta}}
+ \frac{N^{m\omega} \Psi^{m}(s)}{N\eta}
+\frac{N^{(m-1)\omega} \Psi^{m-1}(s)}{N}
+\frac{|I|  N^{m \omega} \Psi^{m-2}(s)}{N^2 \sqrt{N\eta}}
}.
\end{multline}
In this calculation, we used \eqref{e:keyestimate} to bound 
\begin{multline}
\sum_{j=1}^p \frac{2 \ximax^{k_i,\ell}(k_j) }{\mathcal{M}( \ximax \setminus \{k_i,k_j \}}
\sum_{G \in\mathcal{G}_{\ximax \setminus \{ k_i, k_j \}}}
\frac{1}{N}\sum_{\alpha\neq \beta\in I}
\E \left[ \scp{\q_\alpha}{\u_{k_j}} \scp{\q_\beta}{\u_{k_j}}  \Im \scp{\q_\alpha}{G(z_{k_i})\q_\beta} P_s(G) | \bm \lambda\right] 
\\= \O{\frac{ N^{m \omega} \vert I\vert^{3/2}\Psi^{m-2}(s)}{N^{5/2} \sqrt{\eta}} }.
\end{multline}
Note that we only used the last error from Lemma \ref{l:decorrelation} since this is the largest one with our choice of parameters.

Our final observable, $f^{(3)}_s(\ximax^{k_i,\ell})$, contains the graphs in $\mathcal{G}^{(3)}_{\ximax^{k_i,\ell} }$, where the two edges going from the site $\ell$ connect to two distinct sites. 
Since we perform perfect matchings after doubling the number of particles, we see that there exist $r \ge 3$ and $j_1, \dots , j_r \in \{ 1, \dots , p \}$ pairwise distinct such that there exist edges in the matching corresponding to
$p_{\ell k_{j_1}} p_{k_{j_1}k_{j_2}}  \dots p_{k_{j_{r-1}}k_{j_r}} p_{k_{j_r} \ell}$.
\begin{figure}[!ht]
	\centering
	\includegraphics[width=.3\textwidth]{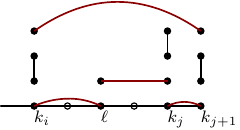}
	\caption{A configuration where a path $p_{\ell k_i}p_{k_ik_{j+1}}p_{k_{j+1}k_{j}}p_{k_j\ell}$ is illustrated in red.}
\end{figure}
At each site, choosing the two sites corresponding to this path gives a combinatorial term of 
\begin{equation}\label{e:combo2}
\prod_{p=1}^r 2 \left(\ximax \setminus \{k_i\}\right)(k_i) 
\left( 2\left(\ximax \setminus \{ k_i \} \right)(k_{j_q}) - 1 \right).
\end{equation}
Note that this quantity represents an upper bound on the cardinality of $\mathcal{G}^{(3)}_{\ximax^{k_i,\ell} }$; it overcounts
because the same configuration of particles can have two such paths $p_{\ell k_{j_1}} p_{k_{j_1}k_{j_2}}  \dots p_{k_{j_{r-1}}k_{j_r}} p_{k_{j_r} \ell}$. However, this crude bound suffices to show these graphs are negligible.

Using \eqref{e:combo2}, we thus want to consider
\begin{multline}
f^{(3)}_s(\ximax^{k_i,\ell}) = 
\frac{1}{ \mathcal M(\ximax^{k_i,\ell}) }
\sum_{ G \in\mathcal{G}^{(3)}_{\ximax^{k_i,\ell} }}  \E [ P_s(G) | \bm \lambda]\\
= 
\sum_{r=3}^p
\sum_{\substack{j_1\dots j_r\\ \text{pairwise }\neq}}^p
\O{
\frac{ \prod_{q=1}^r 2 \left( \ximax \setminus \{k_i\} \right) ( k_{j_q})}{ \mathcal{M} (\ximax \setminus \{k_i, k_{j_1}, \dots , k_{j_r}   \}) }
} 
\E \left[ p_{\ell k_{j_1}} p_{k_{j_1}k_{j_2}}  \dots p_{k_{j_{r-1}}k_{j_r}} p_{k_{j_r} \ell}
\sum_{ G \in\mathcal{G}_{\ximax \setminus \{ k_i, k_{j_1} \dots k_{j_r} \}} }
 P_s(G)  \;\middle\vert\;  \bm \lambda \right],
\end{multline}
where the combinatorial factor is derived using \eqref{e:combo2}.
As before, we consider
\begin{equation}
\Im \sum_{\ell \neq k_1, \dots , k_p} 
\frac{p_{\ell k_{j_1}} p_{k_{j_r} \ell} }{N ( \lambda_\ell - z_{k_i}) }
=
\Im \sum_{\ell=1}^N
\frac{p_{\ell k_{j_1}} p_{k_{j_r} \ell}}{N ( \lambda_\ell - z_{k_i}) }
+
\O{ \frac{ N^{2\omega} \Psi^2(s) }{N\eta}},
\end{equation}
and following  \eqref{blahr} we obtain
\begin{align}
\Im\sum_{\ell =1}^N\frac{p_{\ell k_{j_1}} p_{k_{j_r} \ell}}{N ( \lambda_\ell - z_{k_i}) }
&= \frac{1}{N}p_{k_{j_1}k_{j_r}}\Im \msc(z_{k_i})
\\ &+\frac{1}{N} \sum_{\alpha \neq \beta \in I} \scp{\q_\alpha}{\u_{k_{j_1}}} \scp{\q_\beta}{\u_{k_{j_r}}}
\Im G_{\alpha\beta}(z_{k_i}) + \O{ 
\frac{N^{2\omega}|I|}{N^2 \sqrt{N\eta}}
+ \frac{N^\omega \Psi(s)}{N}
},
\end{align}
where we used \eqref{e:que} to bound the term $p_{k_{j_1} k_{j_2}}$ appearing in this calculation.
Then \eqref{e:keyestimate} yields
\begin{multline}\label{f3}
\Im
\sum_{\ell \neq k_1,\dots,k_p}
\frac{f^{(3)}_s(\ximax^{k_i,\ell})}{N(\lambda_\ell(s)-z_{k_i})}\\
=
 \O{\frac{ N^{m \omega} \vert I\vert^{3/2}\Psi^{m-2}(s)}{N^{5/2} \sqrt{\eta}}
+ \frac{N^{m\omega} \Psi^{m}(s)}{N\eta}
+\frac{N^{(m-1)\omega} \Psi^{m-1}(s)}{N}
+\frac{N^{m\omega}|I|\Psi^{m-2}(s)}{N^2 \sqrt{N\eta}}
}.
\end{multline}
We now insert our estimates \eqref{e:314}, \eqref{e:316}, \eqref{f1}, \eqref{f2}, and \eqref{f3} into \eqref{e:secondtermdynamics}. This gives
\begin{align}
\partial_s f_s(\ximax)
&\le 
- \frac{2}{\eta} \sum_{i=1}^p
\Im \msc(z_{k_i}) \left(f_s(\ximax)
-
 \frac{|I|}{N^2}  \sum_{j=1}^p 
2\left( \ximax \setminus \{ k_i \} \right)(k_j) f_s(\ximax \setminus \{k_i, k_j\} )
\right)
\\
&+ \O{\frac{ N^{m \omega} \vert I\vert^{3/2}\Psi^{m-2}(s)}{N^{5/2} \sqrt{\eta}}
+ \frac{N^{m\omega} \Psi^{m}(s)}{N\eta}
+\frac{N^{(m-1)\omega} \Psi^{m-1}(s)}{N}
+\frac{|I|  N^{2 \omega} \Psi^{m-2}(s)}{N^2 \sqrt{N\eta}}
}\\
& + \O{
 \frac{|I| N^{m\omega} \Psi^{m-1}(s)}{N \sqrt{N\eta} } + 
\frac{ N^{m\omega} \Psi^m(s)}{N\eta}
+ \frac{N^{(m+1)\omega} \Phi^m(s)}{N\eta}
+ \frac{N^{(m+1)\omega} \Phi^m(s)}{N\eta}
}.
\end{align}

With $\theta$ as a parameter, we now suppose that $s \ge N^{-\theta}/2$ and set $\eta = N^{-2\theta -\omega}$ (so that $s \gg \sqrt{\eta}$, needed later). If $\theta$ is chosen small enough relative to $\eps$, the second term in the definition of $\Psi(s)$ dominates.
We now fix $\theta$ and $\omega$ small enough, in a way that depends only on $n$ and $\eps$, so that we have
\begin{align}\label{e:diffeq}
\partial_s f_s(\ximax)
&\le 
- \frac{2}{\eta} \sum_{i=1}^p
\Im \msc(z_{k_i}) \left(f_s(\ximax)
-
 \frac{|I|}{N^2}  \sum_{j=1}^p 
2\left( \ximax \setminus \{ k_i \} \right)(k_j) f_s(\ximax \setminus \{k_i, k_j\} )
\right)
\\&+  C N^{-\eps/2} \left( \frac{\sqrt{|I|}}{N} \right)^{m},
\end{align}
where $C(m)>0$ is a constant.

Using \eqref{e:diffeq}, we proceed by induction on the particle number $m$.
Consider the set of times $t_k = t  - \left(n - k \right) N^{\omega/2} \eta^{1/2}$ for $0 \le k \le  n$. 
Our induction hypothesis at step $m$ is that there exists $N_0(m, \eps)>0$ such that for all $\ell \le m$ and $N \ge N_0$,
\beq\label{e:inducthypo}
\sup_{I} \sup_{\bm \xi: | \bm \xi| = \ell} \sup_{s \in [ t_k , 1]}  \left| \E\left[  f_{\bm \lambda,s}(\bm{\xi})  \right]   -  \left(\frac{\sqrt{2|I|}}{N}\right)^{\ell}\E[g^\ell]  \right| \le N^{-\theta} \left(\frac{\sqrt{|I|}}{N}\right)^{\ell}.
\eeq
For the base case $m=0$, \eqref{e:inducthypo} is trivial, since $f_s(\bm \xi)=1$. 
Next, for the induction step, fix $m < n$ and suppose that \eqref{e:inducthypo} holds for $m-1$.
We now consider times $s \ge t_{m-1}$. Using the induction hypothesis to simplify the second term in \eqref{e:diffeq}, we have 
\begin{align}\label{e:diffeq2}
\partial_s f_s(\ximax)
&\le 
- \frac{2}{\eta} \sum_{i=1}^p
\Im \msc(z_{k_i}) \left(f_s(\ximax)
-
 \left( \frac{\sqrt{2|I|}}{N} \right)^{m}(m-1)\E[g^{m-2}]
\right)
+  C N^{-\theta} \left( \frac{\sqrt{|I|}}{N} \right)^{m}.
\end{align}
Note that we get the correct recursion since $(m-1)\E[g^{m-2}]=\E[g^{m}]$. Recall that $\Im \msc(z) \ge c_{\mathrm sc} \sqrt{\eta}$ for $|z| \le 10$.
%Then at least one of the following claims hold: either 
%\beq
%f_s(\ximax) \le \left( \frac{\sqrt{2|I|}}{N} \right)^{m}\E[g^{m}] + c^{-1}_{\mathrm sc}C %\eta^{1/2} N^{-\theta} \left( \frac{\sqrt{|I|}}{N} \right)^{m}, 
%\eeq

By \eqref{eq:bounddyn0}, $f_s(\ximax)$ is decreasing. So, if there exists a time $r \ge t_{m-1}$ such that
\begin{equation}\label{decreasingupper}
f_r(\ximax) \le \left(  \frac{\sqrt{2|I|} }{N} \right)^m\E[g^{m}] + c^{-1}_{\mathrm sc}C \eta^{1/2} N^{-\theta} \left( \frac{\sqrt{|I|}}{N} \right)^{m},
\end{equation}
then this inequality also holds for all $s > r$. We therefore assume that \eqref{decreasingupper} is false for all $s \in [t_{m-1}, t_m]$. Otherwise, the induction hypothesis holds for $m$, and the induction step is complete. 

When \eqref{decreasingupper} is false, the quantity in parentheses on the right side of \eqref{e:diffeq2} is positive. Then using the inequality $\Im \msc(z) \ge c_{\mathrm sc} \sqrt{\eta}$ in \eqref{e:diffeq2} and integrating the resulting differential inequality, we have
\begin{equation}\label{logit}
\partial_s \log \left(f_s(\ximax)
-
 \left(  \frac{\sqrt{2|I|} }{N} \right)^m\E[g^{m}]
 - C \eta^{1/2} N^{-\theta} \left( \frac{\sqrt{|I|}}{N} \right)^{m}
\right) \le - \frac{c}{\sqrt{\eta}}
\end{equation}
for some constant $c>0$, where we adjusted $C$ upward to absorb $c^{-1}_{\mathrm sc}$.

Integrating \eqref{logit} on the interval $[t_{m-1}, t_m]$ gives, after exponentiation, 
\begin{equation}
f_{t_m}(\ximax)  \le \left(  \frac{\sqrt{2|I|} }{N} \right)^m\E[g^{m}] + C \eta^{1/2} N^{-\theta} \left( \frac{\sqrt{|I|}}{N} \right)^{m} +  f_{t_{n-1}}(\ximax) \exp\left( - \frac{c}{\sqrt{\eta} }( t_m - t_{m-1} ) \right).
\end{equation}
Using \eqref{e:que} to bound
$f_{t_{n-1}}(\ximax)$, we deduce that there exists $N_0(m, \eps)>0$ such that, for $N \ge N_0$,  
\begin{equation}
f_{t_m}(\ximax)  \le \left(  \frac{\sqrt{2|I|} }{N} \right)^m\E[g^{m}] + N^{-\theta} \left( \frac{\sqrt{|I|}}{N} \right)^{m} .
\end{equation}
Recalling the logic following \eqref{e:diffeq}, this implies
\begin{equation}
\sup_{[t_n , 1]} f_{s}(\ximax)  \le \left(  \frac{\sqrt{2|I|} }{N} \right)^m\E[g^{m}] + N^{-\theta} \left( \frac{\sqrt{|I|}}{N} \right)^{m} .
\end{equation}
Applying the same argument to the minimizing configuration $\ximax_{\mathrm{min}}$, we obtain the corresponding lower bound.
 This proves \eqref{e:inducthypo} for $\ell = m$ and completes the induction step.

Integrating over $\bm \lambda \in \mathcal A_2$, we have shown that 
\beq\label{e:dynamicsgood}
\sup_{\bm \xi: | \bm \xi| = n} \sup_{t \in [ N^{-\theta}, 1]}  \left| \E\left[ \one_{\mathcal A_2} f_{\bm \lambda,s}(\bm{\xi})  \right]   - \left(\frac{\sqrt{2|I|}}{N}\right)^{n} \E[g^{m}] \right| \le N^{-\theta} \left(\frac{\sqrt{|I|}}{N}\right)^{n} ,
\eeq
The conclusion follows, since $\E\left[ \one_{\mathcal A_2^c} f_{\bm \lambda,s}(\bm{\xi})  \right]$ is negligible due to \eqref{e:que} and \eqref{e:Croothigh}.
\end{proof}

\bec\label{c:quemom}
Fix $\eps >0$ and $n \in \N$. There exist $\theta (\eps, n ) >0$ and $N_0(\eps, n) >0$ for which the following holds.
For all $N \ge N_0$,
\beq\label{e:quemom}
\sup_{I} \sup_{k \in \unn{1}{N} } \sup_{t \in [ N^{-\theta}, 1]}  \left| \left(\frac{N^2}{2|I|}\right)^n\E\left[p^{2n}_{kk}(t)\right]  - (2n-1)!!  \right| \le N^{-\theta},
\eeq
where the supremum on $I$ is taken over coordinate sets $I \subset \unn{1}{N}$ with $N^\eps \le | I | \le N^{1 - \eps}$, and $p_{kk}(t)$ is defined with respect to $I$.
\eec
\begin{proof}
The claim is immediate from \eqref{e:dynamicsclaim} by considering the particle configurations $\bm\xi$ with $n$ particles at site $k$. Then $f_s(\bm \xi)$ is exactly the moment in \eqref{e:quemom}, up to rescaling.
\end{proof}

\section{Decorrelation estimate}\label{s:decorrelation}

Before giving the proof of Lemma \ref{l:decorrelation}, we need to introduce a few important tools. First, we state the dynamics of the resolvent entries given in \cite{benigni2021fermionic}.

\bel[\cite{benigni2021fermionic}*{Lemma 3.11}]
	Recall that $(\bm{\lambda}(s), \mathbf{u}^s)$ was defined as the solution to the system \eqref{eq:dysonval} and \eqref{eq:dysonvect}. For any $z$ such that $\Im z\neq 0$, if one defines
	\beq\label{timechange}
	\widetilde{G}^s=\mathrm{e}^{-s/2}G^s,
	\eeq
	then we have for any $\alpha,\beta\in\unn{1}{N}$ that
	\begin{equation}\label{eq:dynareso}
		\d \scp{\q_\alpha}{\widetilde{G}^s(z)\q_\beta}
		=
		\left(
		s(z)+\frac{z}{2}
		\right)
		\partial_z \scp{\q_\alpha}{\widetilde{G}^s(z)\q_\beta}\d s
		-
		\frac{1}{\sqrt{N}}
		\sum_{k,\ell=1}^N
		\frac{\scp{\q_\alpha}{\u_k^s}\scp{\q_\beta}{\u_\ell^s}}{(\lambda_k-z)(\lambda_\ell-z)}\d B_{k\ell}.
	\end{equation}
\eel

We make the time change in \eqref{timechange} to connect the dynamics of $\widetilde{G}^s$ with the following stochastic advection equation, which first appeared in \cite{bourgade2018extreme}. The proof is a straightforward computation and therefore omitted.

\bel
Let the characteristic $z_s$ be defined as
\beq\label{characteristics}
z_s=\frac{1}{2}\left(
\mathrm{e}^{s/2}(z+\sqrt{z^2-4})
+
\mathrm{e}^{-s/2}(z-\sqrt{z^2-4})
\right).
\eeq
Then for any function $h_0 \in C^\infty(\mathbb{H})$, we have
\beq \label{advection}
\partial_s h_s(z)
=
\left(
\msc(z)+\frac{z}{2}
\right)
\partial_z h_s(z)
\quad\text{with}\quad
h_s(z) = h_0(z_s)
\eeq
for all $z \in \mathbb{H}$.
\eel

We will use this equation to relate the resolvent at time $s$ with itself at time 0, in order to decorrelate the resolvent from the eigenvectors in \eqref{e:keyestimate}. This reduces the problem to one of eigenvector decorrelation for eigenvector entries with distinct indices. For this, we begin by introducing a new family of eigenvector moment observables.

\bed\label{d:momobs1}
Let $\alpha_1,\alpha_2,\beta_1,\beta_2\in\unn{1}{N}$ be a set of pairwise distinct indices and let $j,k\in\unn{1}{N}$ such that $j\neq k$. Write $\u_k  = \u_k^s$. We define averaged moment observables by
\beq
g_s(k,k)=\frac{N^2}{3}\mathbb{E}\left[\mom{k}{k}{k}{k}\middle\vert H_0,\bm{\lambda}\right]
\eeq
and 
\beq
g_s(j,k) = \frac{N^2}{6}\mathbb{E}\left[Z\middle\vert H_0,\bm{\lambda} \right],
\eeq
where $Z$ is the symmetrized four-point function given by
\begin{align}
Z &= \mom{j}{j}{k}{k}+\mom{j}{k}{j}{k}\\
  &+ \mom{j}{k}{k}{j}+\mom{k}{j}{j}{k}\\
  &+\mom{k}{j}{k}{j}+\mom{k}{k}{j}{j}.
\end{align}
We omit the dependence on $\q_{\alpha_1},\q_{\alpha_2},\q_{\beta_1},\q_{\beta_2}$ in the notation of $g_s(j,k)$.
\eed
\begin{rmk}
	These moment observables generalize the original ones introduced in \cite{bourgade2013eigenvector}, which treated the case $\alpha_1=\alpha_2=\beta_1=\beta_2$. Though they now involve different entries of different eigenvectors, due to the symmetrization they still satisfy a parabolic equation. While we only need this four moment observable for our main results, an extension to higher moments is possible. The details are provided in Appendix \ref{a:genmoments}.
\end{rmk}

\bel\label{l:emfnew1}
Suppose that $(\bm{\lambda}(s), \u^s)$ is the solution to the Dyson vector flow \eqref{eq:dysonvect}. 
Then $g_s$ given by \Cref{d:momobs1} satisfies the equation 
\beq\label{eq:emfnew1}
\partial_s g_s(j,k)=\sum_{\ell\neq k} \zeta_k^{(j,k)}(1+2\zeta_\ell^{(j,k)})\frac{g_s(j,\ell)-g_s(j,k)}{N(\lambda_\ell-\lambda_k)^2}
+
\sum_{\ell\neq j} \zeta_j^{(j,k)}(1+2\zeta_\ell^{(j,k)})\frac{g_s(\ell,k)-g_s(j,k)}{N(\lambda_\ell-\lambda_j)^2},
\eeq
where we define, for $j\neq k$,
\beq
\zeta^{(j,k)}_\ell=0\quad\text{if }\ell\neq j,k,\quad \zeta^{(j,k)}_j=\zeta^{(j,k)}_k=1, \quad \zeta^{(j,j)}_\ell=0\quad\text{if }\ell\neq j,\quad\text{and}\quad \zeta^{(j,j)}_j=1.
\eeq
\eel
We will later drop the superscript $(j,k)$ for ease of notation.
\begin{proof}
For any $k \in \unn{1}{N}$, we have by It\^{o}'s lemma and \eqref{eq:dysonvect} that
\begin{align}\label{e:productflow}
\d \left( \scp{\q_\alpha}{\u_k} \scp{\q_\beta}{\u_k} \right) &= \sum_{\ell \neq k} \frac{ \scp{\q_\alpha}{\u_\ell}\scp{\q_\beta}{\u_\ell} - \scp{\q_\alpha}{\u_k} \scp{\q_\beta}{\u_k} }{ N (\lambda_k - \lambda_\ell)^2}  \\
&+ \frac{1}{\sqrt{N}} \sum_{\ell \neq k} \frac{ \d B_{kl}}{\lambda_k - \lambda_\ell}
\left( \scp{\q_\alpha}{\u_k}\scp{\q_\beta}{\u_\ell} + \scp{\q_\beta}{\u_k} \scp{\q_\alpha}{\u_\ell} \right).
\end{align}
It is then straightforward to derive \eqref{eq:emfnew1} from \eqref{e:productflow} and It\^{o}'s lemma by writing the product of four eigenvector entries as the product of two pairs of entries.
\end{proof}

Because this observable follows a parabolic equation, we can obtain a quantitative bound on it using the maximum principle.
The following proposition shows decorrelation of different entries of eigenvectors up to symmetrization. In particular, if $j=k$ we obtain a bound on the joint moment of distinct entries of the same eigenvector.

\bep\label{p:emf1}
Fix $\omega, \eps >0$, $\mathfrak{a}\in(0,1/3)$, and an initial condition $H_0$. 
Let $I \subset \unn{1}{N}$ be such that $N^{\eps} \le | I | \le N^{1- \eps}$ and set $s=N^{-1/3+\mathfrak{a}}$.
There exists an event $\mathcal{A}_2(\eps, \omega,\q, I)\subset \mathcal{A}_1(\omega,I)$ with $\P\left(\mathcal{A}_2(\eps, \omega, \q, I)\right)\geqslant 1-N^{-D}$ such that the following holds for all eigenvalue paths $\bm{\lambda}\in \mathcal{A}_2(\eps, \omega, \q, I )$. 
There exists $C(\fa, \delta, \eps, \omega) >0$ such that
\beq
\sup_{j,k}\vert g_s(j,k)\vert  \le \frac{C N^{5 \omega }}{N s^{3/2} }.
\eeq
\eep

\begin{proof}
	%The proof is based on a maximum principle since the equation \eqref{eq:emfnew1} can be understood as a parabolic equation. 
Let $(j_m,k_m)$ be a pair of indices such that
	\beq
	g_s(j_m,k_m)=\max_{j,k\in\unn{1}{N}} g_s(j,k)
	\eeq
	Note that this maximum is not necessarily unique; in this case we pick one arbitrarily under the constraint that $(j_m(s),k_m(s))$ remains piecewise constant in $s$. From Lemma \ref{l:emfnew1}, we have
	\beq\label{e:450}
	\partial_s g_s(j_m,k_m)=\sum_{\ell\neq k_m}\zeta_{k_m}(1+2\zeta_\ell)\frac{g_s(j_m,\ell)-g_s(j_m,k_m)}{N(\lambda_\ell-\lambda_{k_m})^2}
	+
	\sum_{\ell\neq j_m} \zeta_{j_m}(1+2\zeta_\ell)\frac{g_s(\ell,k_m)-g_s(j_m,k_m)}{N(\lambda_\ell-\lambda_{j_m})^2}.
	\eeq
	The two sums in \eqref{e:450} are bounded in the same way, so we give details for the first sum only. Note that by definition of $(j_m,k_m)$ we have $g_s(j_m,\ell)-g_s(j_m,k_m)<0$. Let $\eta>0$,  so that	\beq\label{415}
	\sum_{\ell\neq k_m} \zeta_{k_m}(1+2\zeta_\ell)\frac{g_s(j_m,\ell)-g_s(j_m,k_m)}{N(\lambda_\ell-\lambda_{k_m})^2}
	\leqslant -\frac{1}{\eta}\left(
		\sum_{\ell\neq k_m}\frac{\eta g_s(j_m,k_m)}{N((\lambda_\ell-\lambda_{k_m})^2+\eta^2)}
		-
		\sum_{\ell\neq k_m}\frac{\eta g_s(j_m, \ell)}{N((\lambda_{\ell}-\lambda_{k_m})^2+\eta^2)}
	\right).
	\eeq
	If we denote $z_{k_m}=\lambda_{k_m}+\I\eta$, the first sum becomes
	\begin{align}\label{416}
	\sum_{\ell\neq k_m}\frac{g_s(j_m,k_m) \eta }{N((\lambda_k-\lambda_\ell)^2+\eta^2)}&= \frac{g_s(j_m,k_m)}{N}\Im\sum_{\ell =1}^N\frac{1}{\lambda_{\ell}-z_{k_m}}+\O{\frac{g_s(j_m,k_m)}{N\eta}}
	\\&=\Im \msc(z_{k_m}) g_s(j_m,k_m)+\O{\frac{N^\omega g_s(j_m,k_m)}{N\eta}},
	\end{align}
	where we used the local law \eqref{e:sclaw}.
	
	For the second sum in \eqref{415}, %we consider only one term in the definition of $g_s(j_m,\ell)$, first see that we have
	we first note that
	\beq
	\sum_{\ell\neq k_m}\frac{\eta g_s(j_m, \ell)}{N((\lambda_{\ell}-\lambda_{k_m})^2+\eta^2)}
	=
	\Im \sum_{\ell\neq j_m,k_m}\frac{g_s(j_m, \ell)}{N(\lambda_\ell-z_{k_m})}+\O{\frac{g_s(j_m,\ell)}{N\eta}}.
	\eeq
	Considering the term $\mathbb{E}\left[N^2\mom{j_m}{j_m}{\ell}{\ell}\middle\vert H_0,\bm{\lambda}\right]$ in the definition of $g_s(j_m,\ell)$, we bound, using the conditioning on $\bm \lambda$,
	\begin{multline}\label{e:451}
	\mathbb{E}\left[
		N\scp{\q_{\alpha_1}}{\u_{j_m}}\scp{\q_{\beta_1}}{\u_{j_m}}\Im \sum_{\ell \neq j_m,k_m}\frac{\scp{\q_{\alpha_2}}{\u_\ell}\scp{\q_{\beta_2}}{\u_\ell}}{\lambda_\ell-z_{k_m}}
		\middle\vert
		H_0,\bm{\lambda}
	\right]\\
	=\mathbb{E}\left[
	N\scp{\q_{\alpha_1}}{\u_{j_m}}\scp{\q_{\beta_1}}{\u_{j_m}}\Im\scp{\q_{\alpha_2}}{G^s(z_{k_m})\q_{\beta_2}}
	\middle\vert
	H_0,\bm{\lambda}
	\right]
	+\O{\frac{g_s(j_m,j_m)+g_s(j_m,k_m)}{N\eta}}.
	\end{multline}
	In order to bound the last expectation, we use %, similarly to \cite{benigni2021fermionic}*{Equation (3,14)}, 
the dynamics of the resolvent given in \eqref{eq:dynareso}. By Duhamel's formula we have
	\beq
	\widetilde{G}^s_{\alpha\beta}(z)-\widetilde{G}^0_{\alpha\beta}(z_s)
	=
	-\int_0^s \d \widetilde{G}^{s-\tau}_{\alpha\beta}(z_\tau)
	=
	\frac{1}{\sqrt{N}}\sum_{p,q=1}^N\int_0^s
	\frac{\scp{\q_{\alpha}}{\u_{p}^\tau}\scp{\q_{\beta}}{\u_{q}^\tau}\d B_{pq}(\tau)}{(\lambda_k(\tau)-z_{s-\tau})(\lambda_\ell(\tau)-z_{s-\tau})}
	+\O{\frac{N^\omega}{N\eta}}
	\eeq
	from the equation \eqref{eq:dynareso} with $\widetilde{G}^s_{\alpha\beta}(z)=\e^{-s/2}G^s_{\alpha\beta}(z)$, where we used 
\eqref{advection} to bound the drift term.
	
	 Now we observe that in \eqref{e:451}, we condition over the trajectory $\bm{\lambda}$, which consists in conditioning over the Brownian paths $((B_{kk}(\tau))_{\tau\in[0,1]})_{k\in\unn{1}{N}}$. However, the eigenvector dynamics is only driven by the off-diagonal $(B_{k\ell})_{k\neq \ell}$, so that when introducing the sum over the stochastic integrals above, only the diagonal terms have nonzero expectation. Thus
	\begin{align}\label{eq:452}
	\mathbb{E}&\left[
	N\scp{\q_{\alpha_1}}{\u_{j_m}}\scp{\q_{\beta_1}}{\u_{j_m}}\Im\scp{\q_{\alpha_2}}{G^s(z_{k_m})\q_{\beta_2}}
	\middle\vert
	H_0,\bm{\lambda}
	\right]
	\\&=
	\e^{s/2}\mathbb{E}\left[
	N\scp{\q_{\alpha_1}}{\u_{j_m}}\scp{\q_{\beta_1}}{\u_{j_m}}\Im\scp{\q_{\alpha_2}}{G^0((z_{k_m})_s)\q_{\beta_2}}
	\middle\vert
	H_0,\bm{\lambda}
	\right]
	\\
	&+\e^{s/2}\mathbb{E}\left[
	N\scp{\q_{\alpha_1}}{\u_{j_m}}\scp{\q_{\beta_1}}{\u_{j_m}}
	\frac{1}{\sqrt{N}}\sum_{p=1}^N\int_0^s\frac{\scp{\q_{\alpha_2}}{\u_{p}^\tau}\scp{\q_{\beta_2}}{\u_{p}^\tau}}{(\lambda_p(\tau)-(z_{k_m})_{s-\tau})^2}\d B_{pp}(\tau)
	\middle\vert
	H_0,\bm{\lambda}
	\right]
	+\O{\frac{N^{2\omega}}{N\eta}}.
	\end{align}
	For the second term in the equation above, we use the Burkholder--Gundy--Davis inequality to see that the quadratic variation is bounded by
	\beq
	\sum_{p=1}^N\int_0^s\frac{\scp{\q_{\alpha_2}}{\u_{p}^\tau}\scp{\q_{\beta_2}}{\u_{p}^\tau}}{(\lambda_p(\tau)-(z_{k_m})_{s-\tau})^2}\d B_{pp}(\tau)
	=
	\O{\sqrt{\sum_{p=1}^N\int_0^s\frac{\scp{\q_{\alpha_2}}{\u_{p}^\tau}^2\scp{\q_{\beta_2}}{\u_{p}^\tau}^2}{(\lambda_p(\tau)-(z_{k_m})_{s-\tau})^4}\d\tau}}
	=
	\O{\frac{\sqrt{s}N^{2\omega}}{\sqrt{N}\eta^2}},
	\eeq
	where we also used the eigenvector delocalization $\eqref{e:rigidity}$ and the fact that the imaginary part of a characteristic \eqref{characteristics} is increasing in $s$ so that $\Im (z_{k_m})_s\geqslant \eta$. Using this bound, we are able to control the martingale term since we have that for any $D>0$,
	\beq
	\P\left(\sup_{s\in[0,1]}\left\vert\sum_{p=1}^N\int_0^s\frac{\scp{\q_{\alpha_2}}{\u_{p}^\tau}\scp{\q_{\beta_2}}{\u_{p}^\tau}}{(\lambda_p(\tau)-(z_{k_m})_{s-\tau})^2}\d B_{pp}(\tau)\right\vert
	\leqslant \frac{\sqrt{s}N^{3\omega}}{\sqrt{N}\eta^2}\right)\geqslant 1-N^{-D}.
	\eeq
	We now define, $\A_2(\eps, \omega,\q, I)$ to be the intersection of $\A_1(\omega,I)$ and the set of paths such that the inequality above holds for all choices of $\alpha_2, \beta_2 \in I$. 
	Using this bound in \eqref{e:451}, we have
	\beq
	\e^{s/2}\mathbb{E}\left[
	\scp{\q_{\alpha_1}}{\u_{j_m}}\scp{\q_{\beta_1}}{\u_{j_m}}\frac{1}{\sqrt{N}}\sum_{p=1}^N\int_0^s\frac{\scp{\q_{\alpha_2}}{\u_{p}^\tau}\scp{\q_{\beta_2}}{\u_{p}^\tau}}{(\lambda_p(\tau)-(z_{k_m})_{s-\tau})^2}\d B_{pp}(\tau)
	\middle\vert
	H_0,\bm{\lambda}
	\right]
	=
	\O{\frac{\sqrt{s}N^{3\omega}}{N\eta^2}}.
	\eeq
	For the first term in \eqref{eq:452}, $\Im G^0_{\alpha_2\beta_2}(z_{k_m})$ is measurable with respect to $H_0$ and $\bm{\lambda}$, so that 
	\begin{multline}
	\e^{s/2}\mathbb{E}\left[
	N\scp{\q_{\alpha_1}}{\u_{j_m}}\scp{\q_{\beta_1}}{\u_{j_m}}\Im\scp{\q_{\alpha_2}}{G^0((z_{k_m})_s)\q_{\beta_2}}
	\middle\vert
	H_0,\bm{\lambda}
	\right]
	\\=
	\e^{s/2}\mathbb{E}\left[
	N\scp{\q_{\alpha_1}}{\u_{j_m}}\scp{\q_{\beta_1}}{\u_{j_m}}
	\middle\vert
	H_0,\bm{\lambda}
	\right]
	\Im\scp{\q_{\alpha_2}}{G^0((z_{k_m})_s)\q_{\beta_2}}=\O{\frac{N^{2\omega}}{N\eta\sqrt{\Im \msc(z_{k_m})}}}.
	\end{multline}
Here we used the entrywise local law \eqref{e:sclaw} and the bound \cite{benigni2021fermionic}*{(3.12)}, which states that
\beq
\mathbb{E}\left[
N\scp{\q_{\alpha_1}}{\u_{j_m}}\scp{\q_{\beta_1}}{\u_{j_m}}
\middle\vert
H_0,\bm{\lambda}
\right]
=
\O{\frac{N^\omega}{\sqrt{N\eta\Im \msc(z_{k_m})}}}.
\eeq
Combining the previous estimates, we obtain a bound for $\mathbb{E}\left[N^2\mom{j_m}{j_m}{\ell}{\ell}\middle\vert H_0,\bm{\lambda}\right]$ in $g_s(j_m,\ell)$. The other terms in $g_s(j_m,\ell)$ are bounded the same way.
Now we use this bound on $g_s(j_m,\ell)$ in \eqref{415} together with the bound \eqref{416} 
to control the first term of \eqref{e:450}. The second term of  \eqref{e:450} may be bounded similarly, and we obtain
\begin{align}
\partial_s g_s(j_m,k_m)
\leqslant &-\frac{C}{\eta}\left(
(\Im \msc(z_{k_m})+\Im \msc(z_{j_m}))g_s(j_m,k_m)
%\frac{\sqrt{s}N^{3\omega}}{N\eta^2}
\right)\\
&+
\O{\frac{N^{3\omega}}{N\eta}+\frac{\sqrt{s}N^{3\omega}}{N\eta^2}+\frac{N^{2\omega}}{N\eta\sqrt{\Im \msc(z_{k_m})}}+\frac{N^{2\omega}}{N\eta\sqrt{\Im \msc(z_{j_m})}}}.
\end{align}
Here, we used eigenvector delocalization \eqref{e:rigidity} to control the $g_s$ observables in the error terms of the previously mentioned estimates.

Recall the bounds $c\sqrt{\eta}\leqslant \Im \msc(z_{k_m})\leqslant C$, and set $\eta=N^{-2\omega}s^2$, so that
\beq
\frac{(\Im \msc(z_{k_m})+\Im \msc(z_{j_m}))s}{\eta} \geqslant cN^\omega.
\eeq
We obtain
\beq
g_s(j_m,k_m)
=
\O{\frac{N^{3\omega}}{N}+\frac{\sqrt{s}N^{3\omega}}{N\eta}+\frac{N^{3\omega}}{N\eta^{1/4}}+\e^{-cN^\omega}}.
\eeq
The second error term is the largest term, and we obtain an upper bound on the maximizer $g_s(j_m, k_m)$. The same proof can be done with the minimizer instead of the maximizer to get a bound on the absolute value by reversing all inequalities. This completes the proof.
\end{proof}

The next step of our strategy is to relax the symmetrization assumption. Indeed, while we get quantitative decorrelation bounds on symmetrized eigenvector moments, they are not enough to bound one of the joint moments, say $\E\left[N^2\mom{j}{j}{k}{k}\right]$, by itself. We now introduce another set of moment observables which also follows a form of the eigenvector moment flow.

\bed\label{d:momobs2}
Let $\alpha_1,\beta_1,\alpha_2,\beta_2\in\unn{1}{N}$ be a family of pairwise distinct indices, let $j,k\in\unn{1}{N}$ be such that $j\neq k$, 
and write $\u_k = \u^s_k$.
We define the observables
\beq
h_s(j,k)=h_s(k,j)=\frac{N^2}{2}\mathbb{E}\left[Y\middle\vert H_0,\bm{\lambda}\right]-g_s(j,k)
\eeq 
with 
\begin{equation}
Y= \mom{j}{j}{k}{k}+\mom{k}{k}{j}{j},
\end{equation}
and set
\beq
h_s(j,j)=0.
\eeq
\eed
%This a new form of moments observables, closer to the Fermionic observables from \cite{benigni2021fermionic} since it is now a signed observables. 
The dynamics for $h_s$ is given by the following lemma, which resembles the Fermionic eigenvector moment flow of \cite{benigni2021fermionic}.
Its proof is similar to that of \Cref{l:emfnew1}, so we omit it.

\bel
Suppose that $\u^s$ is the solution to the Dyson vector flow \eqref{eq:dysonvect} and $h_s$ is given by \ref{d:momobs2} then it satisfies the equation
\beq
\partial_s h_s(j,k) = \sum_{\ell \neq j,k} \frac{h_s(\ell,k)-h_s(j,k)}{N(\lambda_j-\lambda_\ell)^2}+\sum_{\ell \neq j,k} \frac{h_s(j,\ell)-h_s(j,k)}{N(\lambda_k-\lambda_\ell)^2}.
\eeq
\eel
Since we obtain a parabolic equation, we are able again to use a maximum principle to get quantitative bounds on $h_s(j,k)$ and thus to transfer our bounds on the average over six symmetrized moments to the average over only two moments. This quantitive estimate is given in the following proposition. We omit the proof, since it is similar to the proof of \Cref{p:emf1}.
\bep\label{p:emf2}
Fix $\omega, \eps >0$, $\mathfrak{a}\in (0, 1/3)$, and an initial condition $H_0$. Let $I \subset \unn{1}{N}$ be such that 
$N^{\eps} \le | I | \le N^{1- \eps}$ and set $s=N^{-1/3+\mathfrak{a}}$.
Then the following holds for all eigenvalue paths $\bm{\lambda} \in \mathcal{A}_2(\omega,\eps, \q, I )$, where $\mathcal{A}_2$ is the set from \Cref{p:emf1}. There exists $C(\fa, \delta, \eps) >0$ such that 
\beq
\sup_{j,k}\vert h_s(j,k)\vert \le \frac{C N^{5 \omega}}{N s^{3/2} }.
\eeq
\eep

We now have all the tools to prove Lemma \ref{l:decorrelation}.

\begin{proof}[Proof of Lemma \ref{l:decorrelation}]
	We start by using the resolvent dynamics to relate the resolvent at time $s$ to the resolvent at time $0$ as in the proof of Proposition \ref{p:emf1}. Indeed, $Q_s$ is an observable depending only on the eigenvector dynamics, and the same reasoning holds, so we obtain
	\begin{align}\label{437}
	\E \left[ \sum_{\alpha\neq\beta\in I}\scp{\q_\alpha}{\u_k} \scp{\q_\beta}{\u_j} \Im\scp{\q_\alpha}{G^s(z)\q_\beta}Q_s\middle\vert \bm \lambda\right]
	&=
	\e^{s/2}\mathbb{E}\left[ \sum_{\alpha\neq\beta}\scp{\q_\alpha}{\u_k} \scp{\q_\beta}{\u_j} \Im\scp{\q_\alpha}{G^0(z_s)\q_\beta}Q_s \middle\vert \bm{\lambda}\right]
	\\
	&+ \O{\frac{N^{(n+2)\omega}\vert I\vert^2\Psi^{n}(z)}{N^2\eta}+ \frac{\vert I\vert^2\sqrt{s}N^{(n+2)\omega}\Psi^{n}(z)}{N^2\eta^2}}.
	\end{align}
The estimate is similar to \eqref{eq:452}; the only difference comes from bounding $Q_s =\O{N^{(n-2)\omega}\Psi^n(z)}$ by \eqref{e:que}. To bound the term on the right side of \eqref{437}, we use the conditional Cauchy--Schwarz inequality to get
\begin{multline}
\left\vert \mathbb{E}\left[\sum_{\alpha\neq\beta}\scp{\q_\alpha}{\u_k} \scp{\q_\beta}{\u_j} \Im\scp{\q_\alpha}{G^0(z_s)\q_\beta}P_s(G)\middle\vert \bm{\lambda}\right]\right\vert
\\\leqslant 
\sqrt{\mathbb{E}\left[\sum_{\substack{\alpha_1\neq \beta_1\\ \alpha_2\neq\beta_2}\in I}\mom{k}{j}{k}{j}\Im\scp{\q_{\alpha_1}}{G^0(z_s)\q_{\beta_1}}\Im\scp{\q_{\alpha_2}}{G^0(z_s)\q_{\beta_2}}\vert \bm{\lambda}\right]\E\left[P_s(G)^2\middle\vert \bm{\lambda}\right]}.
\end{multline}
The second expectation in the square root can be bounded by $Q_s^2=\O{N^{2n\omega}\Psi^{2n}(z)}$ using \eqref{e:que}. For the first expectation we  begin by noting that if two indices are equal, say $\alpha_1 = \alpha_2$, then we can bound directly using delocalization and entrywise local law. Using that the imaginary part of $z_s$ is increasing in $s$ we obtain
\beq
\mathbb{E}\left[\sum_{\substack{\alpha \neq \beta_1,\beta_2}\in I}\scp{\q_\alpha}{\u_k}\scp{\q_{\beta_1}}{\u_j}\scp{\q_\alpha}{\u_k}\scp{\q_{\beta_2}}{\u_j}\Im\scp{\q_{\alpha}}{G^0(z_s)\q_{\beta_1}}\Im\scp{\q_{\alpha}}{G^0(z_s)\q_{\beta_2}}\vert \bm{\lambda}\right]=\O{\frac{\vert I\vert^3N^{4\omega}}{N^3\eta}}.
\eeq
It is clear that we get an even better bound if more than 2 indices are equal. Thus, it remains to bound the sum over 4 pairwise distinct indices; we denote by $\sum^\star$ the sum over such indices. We now use that our resolvent depends only on the initial data $H_0$ to write
\begin{multline}
\mathbb{E}\left[\sum_{\alpha_1,\beta_1,\alpha_2,\beta_2\in I}^\star\mom{k}{j}{k}{j}\Im\scp{\q_{\alpha_1}}{G^0(z_s)\q_{\beta_1}}\Im\scp{\q_{\alpha_2}}{G^0(z_s)\q_{\beta_2}}\;\middle\vert\;  \bm{\lambda}\right]
\\=
\mathbb{E}\left[\sum_{\alpha_1,\beta_1,\alpha_2,\beta_2\in I}^\star\E\left[\mom{k}{j}{k}{j}\middle\vert H_0,\bm{\lambda}\right]\Im\scp{\q_{\alpha_1}}{G^0(z_s)\q_{\beta_1}}\Im\scp{\q_{\alpha_2}}{G^0(z_s)\q_{\beta_2}}
 \;\middle\vert\;  \bm{\lambda}\right].
\end{multline}
Since the resolvent matrix is symmetric, $G_{\alpha_1\beta_1}=G_{\beta_1\alpha_1}$, we can symmetrize the sum and obtain
 \begin{multline}
 	\mathbb{E}\left[\sum_{\alpha_1,\beta_1,\alpha_2,\beta_2\in I}^\star\E\left[\mom{k}{j}{k}{j}\vert H_0,\bm{\lambda}\right]\Im\scp{\q_{\alpha_1}}{G^0(z_s)\q_{\beta_1}}\Im\scp{\q_{\alpha_2}}{G^0(z_s)\q_{\beta_2}}
 \;\middle\vert\;  \bm{\lambda}\right]
 	\\=
 	\mathbb{E}\left[\frac{1}{2}\sum_{\substack{\alpha_1,\beta_1,\alpha_2,\beta_2\in I}}^\star\frac{2}{N^2}(h_s(j,k) + g_s(j,k))\Im\scp{\q_{\alpha_1}}{G^0(z_s)\q_{\beta_1}}\Im\scp{\q_{\alpha_2}}{G^0(z_s)\q_{\beta_2}}\middle \vert \bm{\lambda}\right]=\O{\frac{\vert I\vert^4N^{\omega}}{N^4s^{3/2}\eta}},
 \end{multline}
where we used Propositions \ref{p:emf1} and \ref{p:emf2} and the entrywise local law \eqref{e:sclaw}. Putting all these estimates together, we obtain
\begin{multline}
\E \left[ \sum_{\alpha\neq\beta\in I}\scp{\q_\alpha}{\u_k} \scp{\q_\beta}{\u_j} \Im\scp{\q_\alpha}{G^s(z)\q_\beta}Q_s \;\middle\vert\; \bm \lambda\right]\\
=
\O{N^{n\omega}\Psi^{n}(z)\left(\frac{\vert I\vert^2}{N^2}\left(\frac{1}{\eta}+\frac{\sqrt{s}}{\eta^2}+\frac{N^{2\omega}}{s^{3/4}\sqrt{\eta}}\right)+\frac{\vert I\vert^{3/2} N^{2\omega}}{N^{3/2}\sqrt{\eta}}\right)}.
\end{multline}
This proves the lemma. % after taking $\omega$ small.
\end{proof}

\section{Regularized observable}\label{s:quereg}

This section and the next are devoted to the proof of  \Cref{t:main2}. Before entering into details, we give an outline of the argument.

We proceed by induction using the following induction hypothesis, where we recall that the quantities $p_{k\ell}$ were defined in \eqref{pdef}.  For the purpose of this outline, we state a slightly simplified version. The precise definition is given below as \Cref{d:bootstrap}. 

\begin{itemize}%[label=$\star$]
        \item Let $H$ be a generalized Wigner matrix. Fix $\mathfrak{a} \ge 0$ and a family of orthogonal vectors $\q=(\q_\alpha)_{\alpha\in \unn{1}{N}}$  in $\S^{N-1}$. The induction hypothesis holds with parameter $\mathfrak{r}$ if for every $D >0$, there exists $C(\mathfrak{r}, \mathfrak{a}, D) > 0$ such that
\beq\label{unionme2}
 \sup_{k,\ell \in \unn{1}{N}} \sup_{\substack{I_N \subset \unn{1}{N} \\ |I_N| \le N^{\mathfrak{r}}} } \P \left( 
 \left| p_{kl}(I_N) \right| \ge N^{\mathfrak{a}} \frac{\sqrt{|I_N|}}{N} \right) \le C N^{-D}.
\eeq
\end{itemize}

Let $\mathfrak{a} > 0$ and $\q$ be fixed, and suppose the induction hypothesis holds with parameter $\mathfrak{r}$. We will show as a consequence that there exists $\mathfrak{s}(\mathfrak{a}) >0$ such that the induction hypothesis holds for 
the parameter $\mathfrak{r} + \mathfrak{s}$. In short, this means that a joint QUE--QWM bound for sets of size $| I_N | \le N^{\mathfrak{r}}$ implies the same bound for set of size $| I_N | \le N^{\mathfrak{r} + \mathfrak{s}}$. Since the base case $\mathfrak{r} = 0$ holds by the delocalization estimate \eqref{e:rigidity}, it suffices to complete this induction step to show that this estimate applies to all $|I_N| \le N$, proving \Cref{t:main2}.

We begin with a preliminary technical obstacle noted in the introduction: the observables in \Cref{t:main2} are not regular with respect to the matrix entries. We introduce in this section the regularized observable $v(k,\ell, I)$, defined below in \eqref{vdef}, which serves as a smoothed substitute for the overlap between the $k$th and $\ell$th eigenvector appearing in \Cref{t:main2} (or the centered self-overlap when $k = \ell$). It has the property that 
\begin{equation}\label{11}
\left( \frac{N}{\sqrt{I}} p_{k\ell}(I) \right)^2 \lesssim v(k, \ell, I)
\end{equation}
and 
\begin{equation}\label{22}
v(k, \ell, I) \lesssim \sum_{(i,j) \in \mathcal W} \left( \frac{N}{\sqrt{I}} p_{ij}(I) \right)^2
\end{equation}
for a set of index pairs $\mathcal W$ containing $(k,\ell)$, which satisfies $| \mathcal W | \le N^{\kappa}$ for arbitrary $\kappa >0$. Equation \eqref{11} shows that it is enough to prove \Cref{t:main2} for the regularizations $v(k, \ell, I)$, and \eqref{22}, together with \eqref{e:que}, shows that the desired ergodicity and mixing bounds for $v(k, \ell, I)$ hold for matrices with large Gaussian noise. Hence, it remains to remove the hypothesis of Gaussian noise and show these bounds on $v(k, \ell, I)$ hold for all generalized Wigner matrices.

We use the four-moment method to remove this hypothesis. As mentioned in the introduction, given a generalized Wigner matrix $H$ and a time $t \ll 1$, we can find a generalized Wigner matrix $X$ such that the first three moments of the entries of $X_t$ match those of $H$, with error $O(N^{-2} t)$ in the fourth moment. If the first five derivatives of $v(k, \ell, I)$ satisfy good bounds, then by exchanging the entries of $H$ and $X_t$ one at a time using Taylor expansion to fifth order, we can transfer the aforementioned bound on $v(k,\ell, I)$ for $X_t$ to $H$ by comparing moments. Since $O(N^2)$ elements need to be exchanged, with each exchange accumulating error $O(N^{-2} t)$ multiplied by the derivative, it is necessary that the derivatives are $o(t)$ for this to work. 

However, we observe that to obtain the bound in \eqref{unionme2}, the time in \eqref{e:que} degenerates to one as $\mathfrak{a}$ tends to zero. Therefore, we require precise estimates on the derivatives to ensure they are $o(t)$. Obtaining these derivative estimates is the main idea behind the induction step. Assuming the induction hypothesis on sets $I_N$ with $| I_N | \le N^{\mathfrak{r}}$, we show that the derivatives of $v(k,\ell, I)$ for $| I_N | \le N^{\mathfrak{r} + \mathfrak{s}}$ satisfy $\left|\partial^j_{ab} v(k,\ell, I) \right| \lesssim N^{2(\mathfrak{a} + \mathfrak{s})}$ for $j\in\unn{1}{5}$, where $\partial_{ab}$ denotes the derivative with respect to the $(a,b)$ entry. We will see that if $\mathfrak{s}$ is chosen small enough relative to $\mathfrak{a}$, then the derivatives are $o(t)$, and therefore the four moment method implies the moments of $v(k,\ell, I)$ can be controlled when $| I_N | \le N^{\mathfrak{r} + \mathfrak{s}}$. As mentioned previously, this is enough to complete the induction step and hence prove \Cref{t:main2}.

In this section, we construct the regularized observables $v(k,\ell, I)$ and establish their basic properties. The estimates \eqref{11} and \eqref{22} are proved in \Cref{512} and \Cref{regcorollary}, and the derivatives of $v(k,\ell, I)$ are controlled in \Cref{l:Tderiv}. In the next section, \Cref{s:t2}, we give the four moment matching argument in \Cref{l:newcomparison} and prove \Cref{t:main2}.

\subsection{Regularized eigenvalues}
 We begin by recalling a result on the regularization of eigenvalues from \cite{benigni2020optimal}. 
 \bed
 For any $w \in [0,1]$, $M = \left( m_{ij} \right)_{1\le i,j \le N} \in \matn$, and indices $a,b \in \unn{1}{N}$, we define $\Theta^{(a,b)}_w M  \in \matn$ as follows. Let $\Theta^{(a,b)}_w M$ be the $N\times N$ matrix whose $(i,j)$ entry is equal to $m_{ij}$ if $(i,j) \notin \{ (a,b), (b,a) \}$. If $(i,j) \in \{ (a,b), (b,a) \}$, then set the $(i,j)$ entry equal to $w m_{a,b} = w m_{b,a}$.  We also set $\Theta^{(a,b)}_w  G  = (\Theta^{(a,b)}_w M - z)^{-1}$. 
\eed

For $k\in\unn{1}{N}$, we set $\hat{k}=\min(k,N+1-k)$.

\bep[{\cite{benigni2020optimal}*{Proposition 4.1}}]\label{p:regulareigval}
Fix $D,\delta, \eps >0$. For all $i \in \unn{1}{N}$, there exist functions $\widetilde{\lambda}_{i, \delta, \eps} \colon \matn \rightarrow \R$ such that the following holds, where we write $\tilde \lambda_i = \tilde \lambda_{i,\delta, \eps}$. For any generalized Wigner matrix $H$, there exists an event $\mathcal A = \mathcal A (\delta, \eps)$ such that for all  $j\in \unn{1}{5}$ and $a,b,c,d \in \unn{1}{N}$,
\begin{equation}\label{e:regev}
	\one_{\mathcal A} \left\vert	
		\widetilde{\lambda}_i(H)-\lambda_i(H)
	\right\vert	
	\le  \frac{N^\eps}{N^{2/3+\delta}\hat{i}^{1/3}}, \qquad 
	\sup_{0\leqslant w \leqslant 1}	\one_{\mathcal A}
	\left\vert
		\partial_{ab}^j \widetilde{\lambda}_i(\Theta^{(c,d)}_w H)
	\right\vert
	\le
	\frac{N^{j (\eps+\delta)}}{N^{2/3} \hat i^{1/3}},
\end{equation}
and 
\beq\label{e:regevpbound}
\P( \mathcal A^c) \le C_1 N^{-D},
\eeq
for some constant $C_1(D,\delta, \eps)>0$.
Further, there exists $C >0$ such that
\beq\label{e:detbounds}
\sup_{0\leqslant w \leqslant 1}
\left\vert
	\partial^j_{ab} \widetilde{\lambda}_i(\Theta^{(c,d)}_wH)
\right\vert
\leqslant C N^{Cj}.
\eeq
\eep
\begin{rmk}
We note that \cite{benigni2020optimal}*{Proposition 4.1} had stronger error estimates and was stated for random matrices whose entry distributions have subexponential decay. It is routine to see that the same proof gives the previous proposition under the finite moment assumption \eqref{finitemoments}. The same remark holds for \Cref{l:good1} below.
\end{rmk}

\subsection{Construction of regularized observable} \label{s:constructre}

Given \Cref{p:regulareigval}, we can now state the definition of our regularized observable.
For any $i \in \unn{1}{N}$, $x\in \R$ and $\delta_2 >0$, we define intervals
\beq\label{eq:definte1}
I_{\delta_2} (x) =\left[
	x- \frac{ N^{-\delta_2}}{N^{2/3} \hat{i}^{1/3}},x + \frac{ N^{-\delta_2}}{N^{2/3} \hat{i}^{1/3}}\right], \qquad
\hat I_{\delta_2} (x) =\left[
	x - \frac{ N^{-\delta_2}}{2N^{2/3} \hat{i}^{1/3}}, x+ \frac{ N^{-\delta_2}}{2N^{2/3} \hat{i}^{1/3}}\right].
\eeq
When considering intervals $I_{\delta_2} (\lambda_m)$ or $\hat I_{\delta_2} (\lambda_m)$ centered at some eigenvalue $\lambda_m$, we always set $i=m$ in their definitions. 
For any interval $I\subset \R$, we denote the function that counts the eigenvalues $\bm{\lambda}$ of a given matrix in $I$ by
\beq\label{e:lambdaI1}
\mathcal{N}_{\bm{\lambda}}(I)=
\left\vert
	\left\{
		i \in\unn{1}{N}\mid \lambda_i \in I
	\right\}
\right\vert.
\begin{comment}
\quad\text{and}\quad
\mathcal{N}_{\tilde{\bm{\lambda}}}(I)=
\left\vert
	\left\{
		i \in\unn{1}{N}\mid \tilde \lambda_i \in I
	\right\}
\right\vert.
\end{comment}
\eeq

\bed\label{d:regvect}
Suppose $M \in \matn$ and fix $\delta_1, \eps_1, \delta_2, \eps_2 >0$. Let $I_N \subset \unn{1}{N}$ be a deterministic subset and $\q=(\q_\alpha)_{\alpha\in I}$ be an orthogonal family of vectors in $\S^{N-1}$. Let $\tilde \lambda_i = \tilde \lambda_{i, \delta_1, \eps_1}(M)$ denote the regularized eigenvalues of \Cref{p:regulareigval}. Let $G = (M - z)^{-1}$ and set
\begin{align}
Z(z_1, z_2) = Z(M, \q, z_1, z_2) &= \frac{1 }{ | I_N|}\sum_{ \alpha, \beta \in I_N} \Im\scp{\q_\alpha}{G(z_1)\q_\beta}  \Im\scp{\q_\alpha}{G(z_2)\q_\beta}
\\
& - 2N^{-1} \sum_{\alpha \in I_N} \scp{\q_\alpha}{\Im G  (z_1)  \Im G  (z_2)\q_{\alpha}}\\
&+ N^{-2} |I_N| \tr  \Im  G  (z_1)  \Im G  (z_2),
\end{align}
where $z_1, z_2 \in \C \setminus \R$.

Given $k, \ell \in \unn{1}{N}$ and $E_1, E_2 \in \R$, set
\begin{equation}
\eta_k= \frac{N^{-\varepsilon_2}}{N^{2/3} \hat k^{1/3}}, \qquad 
\eta_\ell = \frac{N^{-\varepsilon_2}}{N^{2/3} \hat \ell^{1/3}},
\end{equation}
and define the regularized observable
\begin{equation}\label{vdef}
v(k, \ell) =  
\frac{N^2}{\pi^2} \int_{\hat I_{\delta_2}(\tilde{\lambda}_k)}\int_{\hat I_{\delta_2}(\tilde{\lambda}_\ell)}
Z(M, E_1 + \I \eta_k, E_2 + \I \eta_\ell)
\, \d E_1 \, \d E_2.
\end{equation}
\eed
The function $Z(z_1,z_2)$ is chosen because 
\begin{equation}\label{key}
Z(z_1, z_2) = \frac{1}{| I _N | } \sum_{i,j }^N \frac{ \eta_k\eta_\ell   p^2_{ij} }{ \left( ( \lambda_i - E_1)^2 + \eta^2_k\right) \left( ( \lambda_j - E_2)^2 + \eta^2_\ell  \right) }.
\end{equation}

Following \cite{benigni2020optimal}*{Definition 4.4}, we also define an event where various useful estimates hold simultaneously.

\bed\label{d:goodset}
Suppose $M \in \matn$. Given $\omega,\delta_1, \varepsilon_1,\delta_2, \eps_2, \delta_3 > 0$, and a family of orthogonal vectors $\q=(\q_\alpha)_{\alpha\in I}$ in $\S^{N-1}$, we define the events
\beq
\B(\omega,\delta_1,\varepsilon_1,\delta_2, \eps_2,\q)
=
\B_1(\omega,\q)\cap \B_1(\eps_2/8,\q) \cap \B_2(\delta_1,\varepsilon_1),
\eeq
\beq
\tilde \B(\omega,\delta_1,\varepsilon_1,\delta_2,\eps_2, \delta_3 , \q, k, \ell)  = \B(\omega,\delta_1,\varepsilon_1,\delta_2, \eps_2,\q) \cap \B_3(\delta_2, \delta_3, k , \ell),
\eeq 
where
\begin{itemize}
\item $\B_1(\omega,\q)$ is the event where the local semicircle laws \eqref{e:sclaw}, and rigidity and isotropic delocalization \eqref{e:rigidity} hold for all $\Theta^{(a,b)}_w M$ uniformly in $a,b \in \unn{1}{N}$ and $w \in [0,1]$,\footnote{More precisely, we demand that these equations hold with $\Theta^{(a,b)}_w G$ replacing $G^s$, $\lambda_k( \Theta^{(a,b)}_w M)$ replacing $\lambda_k$, the eigenvectors $\Theta^{(a,b)}_w \u_k$ of $\Theta^{(a,b)}_w M$ replacing $\u^s_k$, and $\sup_{a,b \in \unn{1}{N}}\, \sup_{w \in [0,1]}$ replacing $\sup_{s \in [ 0,1]}$.}

\item $\B_1(\eps_2/8,\q)$ is defined analogously to $\B_1(\omega,\q)$ with $\eps_2/8$ replacing $\omega$,

\item $\B_2(\delta_1,\varepsilon_1)$ is the event where 
\beq\label{e:regcloseB}
  \left\vert	
		\widetilde{\lambda}_i(M)-\lambda_i(M)
	\right\vert	
	\le  \frac{N^{\eps_1}}{N^{2/3+{\delta_1}}\hat{i}^{1/3}}
	\eeq
	and
\beq
\sup_{i \in \unn{1}{N}}\sup_{a,b,c,d \in \unn{1}{N} } \sup_{0\leqslant w \leqslant 1} 
	\left\vert
		\partial_{ab}^j \widetilde{\lambda}_i(\Theta^{(c,d)}_w M)
	\right\vert
	\le
	\frac{N^{j (\delta_1 + \eps_1)}}{N^{2/3} \hat i^{1/3}}
\eeq
hold for all $j\in \unn{1}{5}$, where the $\tilde \lambda_i$ are the regularized eigenvalues given by \Cref{p:regulareigval},
\item $\B_3(\delta_2, \delta_3, k, \ell)$ is the event on which both
$\mathcal{N}_{\bm{\lambda}}(I_{\delta_2}(\lambda_k )) \le N^{\delta_3} $ and $\mathcal{N}_{\bm{\lambda}}(I_{\delta_2}(\lambda_\ell )) \le N^{\delta_3}$ hold,
where $I_{\delta_2}(x )$ is defined in \eqref{eq:definte1} and $\bm \lambda$ is the vector of eigenvalues of $M$.
\end{itemize}
\eed 

We have the following probability bounds.

\bel[{\cite{benigni2020optimal}*{Lemma 4.6}}]\label{l:good1}
Let $H$ be a generalized Wigner matrix. For every choice of $D, \omega, \delta_1, \epsilon_1, \delta_2,\eps_2 >0$, and family of orthogonal vectors $\q=(\q_\alpha)_{\alpha\in \unn{1}{N}}$  in $\S^{N-1}$, there exists a constant $C (D,\omega, \delta_1, \eps_1, \delta_2, \eps_2, \q)$ such that
\beq\label{e:unionthis1}
\P \left(\B^c \left(\omega,\delta_1,\varepsilon_1,\delta_2,\eps_2,\q \right) \right)  \le C N^{-D}.
\eeq
\eel
\bel\label{l:good2}
Let $H$ be a generalized Wigner matrix and let $k,\ell \in \unn{1}{N}$.
For any $D,\omega, \delta_1, \epsilon_1, \delta_2, \eps_2, \delta_3 >0$ with $1/100 > \delta_2 > \eps_1 > 0$, and family of orthogonal vectors $\q=(\q_\alpha)_{\alpha\in \unn{1}{N}}$  in $\S^{N-1}$,
there exists $C(D, \omega, \delta_1, \epsilon_1, \delta_2, \eps_2, \delta_3, \q)>0$ such that
\beq
\P \left(\tilde \B^c \left(\omega,\delta_1,\varepsilon_1,\delta_2, \eps_2, \delta_3, \q, k, \ell \right) \right)\le CN^{-D}.
\eeq
\eel
\begin{proof}
This follows immediately from \Cref{l:good1} and the rigidity statement in \Cref{l:goodset}.
\end{proof}
\begin{rmk}
The bound in the previous lemma is uniform in the choice of $k,\ell \in \unn{1}{N}$. Therefore, a union bound gives a high-probability estimate for all choices of $k,\ell$ simultaneously:
\begin{equation}\label{simultaneous}
\P\left( \bigcap_{k,\ell \in \unn{1}{N}} \tilde \B \left(\omega,\delta_1,\varepsilon_1,\delta_2, \eps_2, \delta_3, \q, k, \ell \right)\right) \ge 1 - C N^{-D}.
\end{equation}
\end{rmk}

Given $M \in \matn$ and a set of orthogonal vectors $\q = (\q_\alpha)_{\alpha \in \unn{1}{N}}$, we define the rescaled ergodicity and mixing observables for any $k, \ell \in \unn{1}{N}$ and $I \subset \unn{1}{N}$ by 
\beq\label{hats}
\hat p_{k\ell}(I) = \hat p_{k\ell}(M,I) = \frac{N}{\sqrt{|I|}} p_{k\ell}(I).
\eeq

\bed \label{d:bootstrap}
Let $H$ be a generalized Wigner matrix. Fix $\mathfrak{r}, \mathfrak{a} \ge 0$ and a family of orthogonal vectors $\q=(\q_\alpha)_{\alpha\in \unn{1}{N}}$  in $\S^{N-1}$. We say that the bootstrap hypothesis $\Q(\mathfrak{r}, \mathfrak{a},\q)$ holds for $H$ if for every $D >0$, there exists $C(\mathfrak{r}, \mathfrak{a}, D) > 0$ such that
\beq\label{unionme}
 \sup_{\substack{I_N \subset \unn{1}{N} \\ |I_N| \le N^{\mathfrak{r}}} } \P \big(\mathcal{E}( \mathfrak{a},\q, I_N )^c \big) \le C N^{-D},
\eeq
where 
\beq
\mathcal{E}( \mathfrak{a},\q, I_N) = 
\bigcap_{a,b \in \unn{1}{N} } \bigcap_{0\leqslant w \leqslant 1}\bigcap_{k,\ell \in \unn{1}{N}} 
 \left\{ \left| \hat p_{k \ell} (\Theta^{(a,b)}_wH,I_N)\right| \le  N^{ \mathfrak{a}} \right \} .
\eeq
\eed

\bel\label{l:bootstrap}
Fix $\mathfrak{r}, \mathfrak{s}, \mathfrak{a}\ge 0$,  a family of orthogonal vectors $\q=(\q_\alpha)_{\alpha\in \unn{1}{N}}$ in $\S^{N-1}$, and suppose the bootstrap hypothesis $\Q(\mathfrak{r}, \mathfrak{a},\q)$ holds. Then $\Q(\mathfrak{r} + \mathfrak{s}, \mathfrak{a} + \mathfrak{s},\q)$ holds.
\eel
\bp
We first consider the $\hat p_{k\ell}$ with $k = \ell$. 
We may assume that $ N^{\mathfrak{r}} \le | I _ N | \le N^{\mathfrak{r} + \mathfrak{s}}$, since otherwise there is nothing to prove. 
Since $| I _ N | \le N^{\mathfrak{r} + \mathfrak{s}}$ we can partition the index set $I_N$ into at most $N^\mathfrak{s}$ disjoint sets of size at most $N^\mathfrak{r}$, which we call $J_\beta$:
\beq
I_N = \bigcup_{\beta \in S} J_\beta, \quad | J_\beta | \le N^{\mathfrak{r}}, \quad |S| \le N^\mathfrak{s}, \quad S = \unn{1}{|S|}.
\eeq
Then using the definition of $\Q(\mathfrak{r}, \mathfrak{a},\q)$, for any $\ell \in \unn{1}{N}$ we have on the set
\begin{equation}
\mathcal E_0 = \bigcap_{\beta \in S} \mathcal E ( \mathfrak{a}, \q, J_\beta)
\end{equation}
that
\begin{align}
\frac{1}{\sqrt{| I_N |}} \sum_{\alpha \in I_N} N\scp{\q_\alpha}{\u_\ell}^2 &=\frac{1}{\sqrt{| I_N |}} \sum_{\beta\in S} \sum_{\alpha \in J_\beta} N\scp{\q_\alpha}{\u_\ell}^2 \\
& \le  \frac{1}{\sqrt{| I_N |}} \sum_{\beta\in S} |J_\beta| +   N^{\mathfrak{a} + \mathfrak{s}} \\
&\le  \sqrt{| I _N |} +  N^{\mathfrak{a} + \mathfrak{s}}.
\end{align}
%The constant in the $\mathcal O$ notation can be taken to be $1$. 
%We conclude using $| I_N | \ge N^{\mathfrak{r}}$ that
A similar argument gives a lower bound of $\sqrt{| I _N |} -  N^{\mathfrak{a} + \mathfrak{s}}$, and we conclude that
\begin{equation}\label{e:newq}
\left | \hat p_{\ell \ell} (I_N) \right | \le N^{\mathfrak{a} + \mathfrak{s}} 
\end{equation}
for all $\ell \in \unn{1}{N}$. The argument for $p_{k \ell}$ with $k \neq \ell$ is similar. Finally, we observe that for any $D>0$, there exists $C(\mathfrak{r} + \mathfrak{s}, ,\mathfrak{a} + \mathfrak{s}, D)>0$ such that $\P( \mathcal E_0^c ) \le C N^{-D}$ by applying a union bound to \eqref{unionme}. In particular, this constant is independent of the choice of $I_N$ above.
\ep

\subsection{Estimates on the regularized observable}\label{s:estimatesre}

We first state a preliminary lemma.

\bel[{\cite{benigni2020optimal}*{Lemma 4.9}}]\label{l:eigsqsum}
Let $H$ be a generalized Wigner matrix, and fix  $\omega,\delta_1,\varepsilon_1,\delta_2, \eps_2> 0$. Then there exists $C = C(\omega) >0$ such that, for all $\ell\in\unn{1}{N}$, we have\footnote{In fact, the proof of \cite{benigni2020optimal}*{Lemma 4.9} yields the stronger upper bound $C N^{2/3 - 2\omega - \delta_2} {\hat \ell}^{1/3}$, but we do not need this improvement.}
\beq\label{e:eigsqsum}
\one_{\B} \sum_{p: \vert p-\ell\vert>N^{2\omega}}
\int_{\hat I_{\delta_2}(\lambda_\ell)} \frac{\d E}{(\lambda_p- E )^2}
\leqslant
C N^{2/3 + 2\omega - \delta_2} {\hat \ell}^{1/3},
\eeq
where $\B = \B(\omega,\delta_1,\varepsilon_1,\delta_2,\eps_2)$. 
\eel

The following lemma relates the regularized observable $v(k, \ell)$ to the $\hat p_{ij}$. We define
\begin{equation}
\mathcal J = \mathcal J(k,\ell) = \{ (i,j) :  i  \in  I_{\delta_2}({\lambda}_k) \text{ and } j  \in  I_{\delta_2}({\lambda}_\ell) \}.
\end{equation}

\bel\label{l:boundregeig}
Let $H$ be a generalized Wigner matrix. Fix $\delta_1,\varepsilon_1,\delta_2, \eps_2, \omega>0, \mathfrak{r},\mathfrak{s},\mathfrak{a}\geqslant 0$ with $\delta_1 > \eps_1$ and $\eps_2 > \delta_2 + 3\mathfrak{a}$, and a family of orthogonal vectors $\q=(\q_\alpha)_{\alpha\in \unn{1}{N}}$ in $\S^{N-1}$.
% and suppose the bootstrap hypothesis $\Q(\kappa, \omega,\q)$ holds. 
Suppose $I_N \subset \unn{1}{N}$ satisfies $| I_N | \le N^{\mathfrak{r} +\mathfrak{s}}$.
Then  
for all $k,\ell\in\unn{1}{N}$, we have
\begin{align}\label{e:boundedreclaim}
\one_{ \B \cap \mathcal E} v(k,\ell) 
&=
\one_{ \B \cap \mathcal E}
\frac{1}{\pi^2} \int_{\hat I_{\delta_2}({\lambda}_k)}\int_{\hat I_{\delta_2}({\lambda}_\ell)}  \sum_{(i,j)\in \mathcal J } \frac{ \eta_k\eta_\ell   \hat p^2_{ij} }{ \left( ( \lambda_i - E_1)^2 + \eta^2_k\right) \left( ( \lambda_j - E_2)^2 + \eta^2_\ell  \right) } \, \d E_1 \, \d E_2\\
&+ \mathcal{O}_{\mathfrak a}\left( N^{2\mathfrak{s} + 6\mathfrak{a} + \delta_2 - \eps_2} \left(  \mathcal N_{\bm \lambda } ( I_{\delta_2}(\lambda_k)) +  \mathcal N_{\bm \lambda } ( I_{\delta_2}(\lambda_\ell))   \right)\right)
+ \O{\left| \mathcal J  \right|N^{ \eps_1 + 2\eps_2 - \delta_1 - \delta_2}},
\end{align}
where $ \B =  \B(\omega,\delta_1,\varepsilon_1,\delta_2, \eps_2,\q)$ and $\mathcal E = \mathcal E (\mathfrak{a} + \mathfrak{s},\q, I)$. 
\eel
\begin{proof}

We work exclusively on the event $ \B \cap \mathcal{E}$ and drop this from our notation. 

Define
\begin{equation}
\tilde v(k, \ell) =  
\frac{N^2}{\pi^2} \int_{\hat I_{\delta_2}({\lambda}_k)}\int_{\hat I_{\delta_2}({\lambda}_\ell)}
Z(M, E_1 + \I \eta_k, E_2 + \I \eta_\ell)
\, \d E_1 \, \d E_2,
\end{equation}
where we have changed the intervals of integration compared to $v(k,\ell)$. Using \eqref{key} and \eqref{hats}, we write
\begin{align}
\tilde v(k, \ell)
&=
\frac{1}{\pi^2} \int_{\hat I_{\delta_2}({\lambda}_k)}\int_{\hat I_{\delta_2}({\lambda}_\ell)}  \sum_{(i,j)\in \mathcal I } \frac{ \eta_k\eta_\ell   \hat p^2_{ij} }{ \left( ( \lambda_i - E_1)^2 + \eta^2_k\right) \left( ( \lambda_j - E_2)^2 + \eta^2_\ell  \right) } \, \d E_1 \, \d E_2\\
&+\frac{1}{\pi^2} \int_{\hat I_{\delta_2}({\lambda}_k)}\int_{\hat I_{\delta_2}({\lambda}_\ell)}
\sum_{(i,j) \in \mathcal I^c } \frac{ \eta_k\eta_\ell   \hat p^2_{ij} }{ \left( ( \lambda_i - E_1)^2 + \eta^2_k\right) \left( ( \lambda_j - E_2)^2 + \eta^2_\ell  \right) }
\, \d E_1 \, \d E_2
\label{e:regvect1} ,
\end{align}
where we define
\begin{equation}
\mathcal I = \mathcal I(k,\ell) = \{ (i,j) :  | i - k| \le N^{2\mathfrak{a}} \text{ and }  | j -  \ell | \le N^{2\mathfrak{a}} \}.
\end{equation}

We now show that \eqref{e:regvect1} is a negligible error term. Using \eqref{e:newq} and \eqref{e:eigsqsum}, we  estimate
\begin{align}
&\frac{1}{\pi^2} \int_{\hat I_{\delta_2}({\lambda}_k)}\int_{\hat I_{\delta_2}({\lambda}_\ell)}
\sum_{\substack{ | i - k | > N^{2\mathfrak{a}} \\ 1 \le j \le N} } \frac{ \eta_k\eta_\ell   \hat p^2_{ij} }{ \left( ( \lambda_i - E_1)^2 + \eta^2_k\right) \left( ( \lambda_j - E_2)^2 + \eta^2_\ell  \right) } 
\, \d E_1 \, \d E_2 \\ 
&\le 
C \eta_k  N^{2\mathfrak{s} + 2\mathfrak{a}}
 \int_{\hat I_{\delta_2}({\lambda}_k)}\int_{\hat I_{\delta_2}({\lambda}_\ell)}
\sum_{\substack{ | i - k | > N^{2\mathfrak{a}} \\ 1 \le j \le N} }
\frac{ \eta_\ell  }{ \left(  \lambda_i - E_1\right)^2 \left( ( \lambda_j - E_2)^2 + \eta^2_\ell  \right) } 
\, \d E_1 \, \d E_2 \\ 
&\le 
C \mathcal N_{\bm \lambda } ( I_{\delta_2}(\lambda_\ell))\eta_k  N^{2\mathfrak{s} + 2\mathfrak{a}}
\int_{\hat I_{\delta_2}({\lambda}_k)}
\sum_{| i - k | > N^{2\mathfrak{a}}  }
\frac{ 1}{ \left(  \lambda_i - E_1\right)^2  } 
\, \d E_1\label{sumN}  \\ 
&\le C \mathcal N_{\bm \lambda } ( I_{\delta_2}(\lambda_\ell))  \eta_k N^{2\mathfrak{s} + 2\mathfrak{a}} N^{2/3 + 2\mathfrak{a} - \delta_2} \hat{k}^{1/3}\\
&\le C  \mathcal N_{\bm \lambda } ( I_{\delta_2}(\lambda_\ell)) N^{2\mathfrak{s} + 4\mathfrak{a} - \delta_2 - \eps_2}.
\end{align}
where $C=C(\mathfrak{a})>0$ is a constant that depends on $\mathfrak{a}$. In line \eqref{sumN}, we used the bound
\begin{align}
\frac{1}{\pi} \int_{\hat I_{\delta_2}({\lambda}_\ell)} \sum_{j=1}^N \frac{\eta_\ell}{(\lambda_j - E_2)^2 + \eta_\ell^2} \, \d E_2  &\le  \mathcal N_{\bm \lambda } ( I_{\delta_2}(\lambda_\ell)) + N^{2\mathfrak{a}} \eta_\ell N^{\eps_2 - \delta_2} \frac{\eta_\ell}{\left( \frac{1}{2} \eta_\ell N^{\eps_2 - \delta_2} \right)^2} + C N^{ 2\mathfrak{a} - \delta_2 - \eps_2} 
\\ &\le 2 \mathcal N_{\bm \lambda } ( I_{\delta_2}(\lambda_\ell)) ,
\end{align}
which follows from \eqref{e:eigsqsum}, $| \hat I_{\delta_2}(\lambda_k)  | =\eta_k N^{\eps_2 - \delta_2}$, and partitioning the sum into the subsets $\{j \in I_{\delta_2}(\lambda_\ell)\}$, $\{j \in I_{\delta_2}(\lambda_\ell)\}^c \cap \{ |j - \ell | \le N^{2\mathfrak{a}} \}$, and $\{ |j - \ell | \ge N^{2\mathfrak{a}} \}$. In the second subset, we bounded the integrand by using $| \lambda_j - E_2| > \frac{1}{2} \eta_\ell N^{\eps_2 - \delta_2}$ in the denominator.

An identical estimate holds for the sum over $ 1 \le i \le N$ and $ | j - \ell | > N^{2\mathfrak{a}}$ in \eqref{e:regvect1}. We conclude
\begin{align}\label{eq:regvect3}
\tilde v(k, \ell) &= 
\frac{1}{\pi^2} \int_{\hat I_{\delta_2}({\lambda}_k)}\int_{\hat I_{\delta_2}({\lambda}_\ell)}  \sum_{(i,j)\in \mathcal I }^N \frac{ \eta_k\eta_\ell   \hat p^2_{ij} }{ \left( ( \lambda_i - E_1)^2 + \eta^2_k\right) \left( ( \lambda_j - E_2)^2 + \eta^2_\ell  \right) } \, \d E_1 \, \d E_2\\
&+ \mathcal{O}_\mathfrak{a} \left( N^{  2\mathfrak{s} + 4\mathfrak{a} -\delta_2 - \eps_2} \left(  \mathcal N_{\bm \lambda } ( I_{\delta_2}(\lambda_k)) +  \mathcal N_{\bm \lambda } ( I_{\delta_2}(\lambda_\ell))   \right)\right).
\end{align}

Next, we remove the terms in the sum in \eqref{eq:regvect3} corresponding to the eigenvalues that do not lie in the sub-microscopic region $\mathcal J$. We observe that 
\begin{align}
\{(i,j) \in \mathcal I \setminus \mathcal J \} &\subset \left\{(i,j) \in \mathcal I \cap \lambda_i \notin I_{\delta_2}(\lambda_k)\right\} \bigcup \left\{(i,j) \in \mathcal I \cap \lambda_j \notin I_{\delta_2}(\lambda_\ell)\right\}\\
& := \mathcal K_1 \cup \mathcal K_2.
\end{align}
We give details for the sum over $\mathcal K_1$, since the sum over $\mathcal K_2$ is analogous.
We use $| \hat I_{\delta_2}(\lambda_k)  | =\eta_k N^{\eps_2 - \delta_2}$, $|\lambda_i - E| \ge  | \hat I_{\delta_2}(\lambda_k)  |$, $\left| \mathcal K_1 \right| \le N^{4\mathfrak{a}}$, and \eqref{e:newq} to deduce
\begin{align}
&\frac{1}{\pi^2} \int_{\hat I_{\delta_2}({\lambda}_k)}\int_{\hat I_{\delta_2}({\lambda}_\ell)}
\sum_{(i,j) \in \mathcal K_1 } \frac{ \eta_k\eta_\ell   \hat p^2_{ij} }{ \left( ( \lambda_i - E_1)^2 + \eta^2_k\right) \left( ( \lambda_j - E_2)^2 + \eta^2_\ell  \right) } 
\, \d E_1 \, \d E_2 \\ 
&\le 
\frac{N^{2\mathfrak{s} + 2\mathfrak{a}}}{\pi^2} \eta_k N^{\eps_2 - \delta_2} \int_{\hat I_{\delta_2}({\lambda}_\ell)}
\sum_{(i,j) \in \mathcal K_1 } \frac{ \eta_k\eta_\ell    }{ \left( (\frac{1}{2} \eta_k N^{\eps_2 - \delta_2}  )^2 \right) \left( ( \lambda_j - E_2)^2 + \eta^2_\ell  \right) } 
 \, \d E_2 \\ 
 &\le 
\frac{N^{2\mathfrak{s} + 2\mathfrak{a}}}{\pi} \eta_k N^{\eps_2 - \delta_2} 
N^{4\mathfrak{a}} \frac{ \eta_k   }{ (\frac{1}{2} \eta_k N^{\eps_2 - \delta_2}  )^2   } 
 \\ 
 & = 4 N^{2 \mathfrak{s} + 6 \mathfrak{a} + \delta_2 - \eps_2}.\label{eq:regvect4}
\end{align}

We have shown \eqref{e:boundedreclaim} with $\tilde v(k,\ell)$ in place of $v(k,\ell)$. To finish, we bound the difference of $\tilde v$ and $v$. 
Observe that by \eqref{e:regcloseB},
\beq\label{e:symmdiff}
 | \hat I_{\delta_2} (\lambda_k) \triangle \hat I_{\delta_2}( \tilde \lambda_k) | \le  | \lambda_k -\tilde \lambda_k | 
\le  \frac{N^{\eps_1}}{N^{2/3 + \delta_1} {\hat k}^{1/3}},
\eeq
and similarly for $\ell$ in place of $k$. Then, using that the integrand is bounded by $N^{2\mathfrak{s} + 2\mathfrak{a}} (\eta_k \eta_\ell)^{-1}$, we obtain

\begin{align}
\frac{1}{\pi^2} \int_{\hat I_{\delta_2}({\lambda}_k)}\int_{\hat I_{\delta_2}({\lambda}_\ell)} &\frac{ \eta_k\eta_\ell   \hat p^2_{ij} }{ \left( ( \lambda_i - E_1)^2 + \eta^2_k\right) \left( ( \lambda_j - E_2)^2 + \eta^2_\ell  \right) } 
\, \d E_1 \, \d E_2 \\ 
&= \frac{1}{\pi^2} \int_{\hat I_{\delta_2}({\tilde\lambda}_k)}\int_{\hat I_{\delta_2}({\tilde\lambda}_\ell)} \frac{ \eta_k\eta_\ell   \hat p^2_{ij} }{ \left( ( \lambda_i - E_1)^2 + \eta^2_k\right) \left( ( \lambda_j - E_2)^2 + \eta^2_\ell  \right) } 
\, \d E_1 \, \d E_2 \\
&+ \O{N^{\eps_1 + 2\eps_2 - \delta_1- \delta_2}}\label{addtilde}
\end{align}
We conclude by combining the estimates \eqref{eq:regvect3}, \eqref{eq:regvect4}, and \eqref{addtilde} to obtain \eqref{e:boundedreclaim}.
\end{proof}
\bec\label{regcorollary}
Retain the notation and hypotheses of the previous lemma, and fix $\delta_3 > 0$. Then
\begin{align}\label{e:boundedreclaim2}
\one_{ \tilde \B \cap \mathcal E} v(k,\ell) 
&\le
\one_{ \tilde \B \cap \mathcal E}
\sum_{i = k - N^{\delta_3}  }^{k+ N^{\delta_3}}  \sum_{j= \ell - N^{\delta_3}}^{\ell+N^{\delta_3}} 
\hat p^2_{ij}+ \mathcal{O}_{\mathfrak a} \left( N^{2\mathfrak{s} + 6\mathfrak{a} + \delta_2 + \delta_3 - \eps_2}\right)
+ {\mathcal O} \left( N^{\eps_1 + 2\eps_2 - \delta_1 - \delta_2 + \delta_3}\right),
\end{align}
where $ \tilde \B =  \tilde \B(\omega,\delta_1,\varepsilon_1,\delta_2, \eps_2, \delta_3, \q,k,\ell)$ and $\mathcal E = \mathcal E (\mathfrak{a} + \mathfrak{s},\q, I)$. 
\eec
\begin{proof}
The claim \eqref{e:boundedreclaim2} is immediate from \eqref{e:boundedreclaim} and the definition of $\tilde \B$.
\end{proof}

\begin{lem}\label{512} Let $H$ be a generalized Wigner matrix. Fix $\delta_1,\varepsilon_1,\delta_2, \eps_2, \omega>0,$ $\mathfrak{r}, \mathfrak{s}, \mathfrak{a}\geqslant 0$ and a family of orthogonal vectors $\q=(\q_\alpha)_{\alpha\in \unn{1}{N}}$ in $\S^{N-1}$.
Suppose $I_N \subset \unn{1}{N}$ satisfies $| I_N | \le N^{\mathfrak{r} +\mathfrak{s}}$.
Then  
for all $k,\ell\in\unn{1}{N}$, we have
\begin{align}\label{e:boundedreclaim4}
\one_{\B \cap \mathcal E} \hat{p}_{k\ell}^2\left(\Theta^{(a,b)}_w H\right) \le v\left(\Theta^{(a,b)}_w H,k,\ell\right)  + \O{ N^{2\mathfrak{s} + 2\mathfrak{a} + \delta_2 -\eps_2 }} + \O{N^{\eps_1 + 2\eps_2 - \delta_1- \delta_2}}
%+ \O{ N^{2\sigma + 2\omega + \delta_2 -\eps_2 }}.
\end{align}
 where the error is uniform in $a,b \in \unn{1}{N}$ and $w \in [0,1]$, and $ \B =  \B(\omega,\delta_1,\varepsilon_1,\delta_2, \eps_2,\q)$, $\mathcal E = \mathcal E (\mathfrak{a} + \mathfrak{s},\q, I)$.
\end{lem}
\begin{proof}
Using $\int \frac{a \, \d x}{x^2 + a^2}  = \arctan\left( \frac{x}{a} \right)$,  we have on $\B$ that
\begin{align}
\hat p_{k\ell}^2&=
\frac{1}{\pi^2}\int_\R \int_\R \frac{ \eta_k\eta_\ell   \hat p^2_{k \ell} }{ \left( ( \lambda_k - E_1)^2 + \eta^2_k\right) \left( ( \lambda_\ell - E_2)^2 + \eta^2_\ell  \right) }\, \d E_1 \, \d E_2 \\
&=
\frac{1}{\pi^2}\int_{\hat I_{\delta_2}({\lambda}_k)}\int_{\hat I_{\delta_2}({\lambda}_\ell)} \frac{ \eta_k\eta_\ell   \hat p^2_{k\ell} }{ \left( ( \lambda_k - E_1)^2 + \eta^2_k\right) \left( ( \lambda_\ell - E_2)^2 + \eta^2_\ell  \right) } \, \d E_1 \, \d E_2\label{positivity}\\
& + \O{ N^{2\mathfrak{s} + 2\mathfrak{a} + \delta_2 -\eps_2 }}.\label{theerror}
\end{align}
%This shows \eqref{e:boundedreclaim3}, after using \eqref{e:boundedreclaim2} and \eqref{addtilde}.
Finally, \eqref{positivity}, \eqref{addtilde}, and \eqref{key} imply \eqref{e:boundedreclaim4}.
\end{proof}
\begin{rmk}
We observe that the previous lemma holds with no restriction on the size of $\mathfrak{a}$ or $\mathfrak{s}$. In particular, the error term may  grow with $N$ if $\mathfrak{a}$ is taken large; this case will be used in the next section.
\end{rmk}

 For $M \in \matn$, $I_N \subset \unn{1}{N}$, $k, \ell \in \unn{1}{N}$, and $n \in \N$, we define the shorthand
\beq\label{e:Tdef}
T(M) = v(M, k, \ell), \qquad T_n(M) = T(M)^n.
\eeq 
We now control moments of the observable $T(M)$ for the matrix process $H_s$. Recall that the bootstrap hypothesis is not needed for $H_s$ since we already have the bound \eqref{e:que}.  

\bel \label{l:dynamichighmom}
Fix $\delta , \mathfrak{b} >0$ and a family of orthogonal vectors $\q=(\q_\alpha)_{\alpha\in \unn{1}{N}}$ in $\S^{N-1}$. 
Set $s = N^{- \mathfrak{b}/4}$. 
For any generalized Wigner matrix $H$, and all $k,\ell\in\unn{1}{N}$,  $I_N \subset \unn{1}{N}$, and $n \in \mathbb N$, there exists a constant $C = C(\delta, \mathfrak{b}, n)>0$
such that
\beq
\E\left[
T_n(H_s)
\right]
\leqslant 
C N^{n(1 + 10^{-5})\mathfrak{b}},
\eeq
 where $T_n(H_s)$ is defined as in \eqref{e:Tdef} using $\delta_1 = \delta$, $\eps_2 = \delta_1/10^2$, $ \delta_2 = \delta_1/10^3$,  and $\eps_1 = \delta_1/10^4$.

\eel
\begin{proof}
Set $\omega =  \min( \eps_1, 10^{-5}\mathfrak{b} /3)$,   $\tilde \B= \tilde \B(\omega,\delta_1,\varepsilon_1,\delta_2,\eps_2,\omega,\q,k, \ell)\subset\tilde \B(\omega,\delta_1,\varepsilon_1,\delta_2,\eps_2,\mathfrak{b},\q,k, \ell)$, and $\mathcal E = \mathcal E (\omega, I , \q)$.
We write

\beq\label{e:hpdecomp}
\E\left[ T_n(H_s) \right]
=
\E\left[
	T_n(H_s)
	\left(
		\one_{\tilde \B \cap \mathcal E}
		+
		\one_{(\tilde \B \cap \mathcal E)^c}
	\right)
\right].
\eeq
By \eqref{e:boundedreclaim2} with $\mathfrak{a} = \omega$, $\mathfrak{s} =0$, and $\delta_3 = \omega$, we have
\begin{align}\label{e:boundedreclaim22}
\one_{ \tilde \B \cap \mathcal E} v(k,\ell) 
&\le
\one_{ \tilde \B \cap \mathcal E}
\sum_{i = k - N^{\omega}  }^{k+N^{\omega}}  \sum_{j= \ell - N^{\omega}}^{\ell+ N^{\omega}} 
\hat p^2_{ij}+ \mathcal{O}_{\omega} \left( N^{- \theta}\right)
\end{align} 
for some $\theta(\delta) > 0$. 
\begin{comment}
 Because  $v_\ell \ge 0$, \beq
\E\left[
 T_n(M) 
	\one_{\tilde \B\cap \mathcal E}
\right]
\le
\E\left[
	 v(k,\ell)^n
	\one_{\tilde \B \cap \mathcal E}
\right].
\eeq
\end{comment}
Note that \eqref{e:que} and \eqref{psis} give that $\hat p^2_{ij}\leqslant N^{\omega}\left(\frac{\vert I\vert}{Ns^4}+s^{-3}\right)\leqslant 2N^{\omega + \mathfrak{b}}$. Together with \eqref{e:boundedreclaim22}, this implies
\beq
\E\left[
	T_n(H_s) \one_{\tilde \B\cap  \mathcal E}
\right]
\leqslant 
\E\left[
	\left(
\sum_{i = k - N^{\omega} }^{k+N^{\omega}}  \sum_{j= \ell - N^{\omega}}^{\ell+ N^{\omega}}  
\hat p^2_{ij}+ \mathcal{O}_{\omega} \left( N^{ - \theta}\right)
	\right)^n
\right] \leqslant CN^{3n\omega + n\mathfrak{b}}
\le C N^{n(1 + 10^{-5})\mathfrak{b}}\label{e:rcp1}
\eeq
for some constant $C(\delta, \mathfrak{b}, n) > 0$.
%If we set $s=N^{-10^5 \delta_3}$ we obtain
%\beq\label{e:rcp1}
%\E\left[
%T_n(H_s) \one_{\tilde \B\cap  \mathcal E}
%\right]
% \le C{N^{(3\cdot 10^5 + 3)n\delta_3}}.
%\eeq
Note that for every $D >0$ there exists $C(D, \delta, \mathfrak{b})>0$ such that 
\begin{equation}\label{detB}
\P(\tilde \B^c)   + \P({\mathcal E}^c) \le C N^{-D},
\end{equation}
by \Cref{l:good2} and \eqref{e:que}. Using the deterministic bound $|G_{\alpha\beta}(z) | \le (\Im z)^{-1}$ in the definition of $T_n$ (in particular in \Cref{d:regvect}) we have the bound $T_n \le C N^{10 n}$, so we take $D=20n$. Together with \eqref{detB}, this shows
\beq\label{e:rcp2}
\E\left[
	T_n(M) \one_{(\tilde \B \cap \mathcal E)^c}
\right]
\le C
\eeq
for some constant $C(\delta, \mathfrak{b}, n) > 0$. 
Combining \eqref{e:rcp1} and \eqref{e:rcp2} completes the proof.
\end{proof}

We also require bounds on the derivatives of $T(H)$, which will be used in the next section. We begin with the following bounds on $Z(k,\ell)$, which are proved in \Cref{s:a1}.

\bel\label{l:obbounds}
Fix $\omega, \delta_1, \eps_1, \delta_2, \eps_2>0$, $\mathfrak{r}, \mathfrak{s}, \mathfrak{a}\geqslant 0$ such that $\delta_1 > \eps_1$, a family of orthogonal vectors $\q=(\q_\alpha)_{\alpha\in \unn{1}{N}}$ in $\S^{N-1}$, and $k, \ell \in \unn{1}{N}$.  
Let $H$ be a generalized Wigner matrix, and let $\B = \B(\omega,\delta_1,\varepsilon_1,\delta_2, \eps_2,\q)$ be the set from \Cref{d:goodset}. Let $\eta_k =\frac{N^{-\varepsilon_2}}{N^{2/3} \hat k^{1/3}}$ and $\eta_\ell =\frac{N^{-\varepsilon_2}}{N^{2/3} \hat \ell^{1/3}}$.
Set $z_1 = E_1+\I\eta_k$ and $z_2 = E_2+\I\eta_\ell$, and suppose $I_N \subset \unn{1}{N}$ satisfies $| I_N | \le N^{\mathfrak{r} +\mathfrak{s}}$.
Then there exists a constant $C= C(\eps_2)>0$ such that for all $E_1 \in I_{\delta_2}(\lambda_k) \cup I_{\delta_2}(\tilde \lambda_k)$ and $E_2 \in I_{\delta_2}(\lambda_\ell) \cup I_{\delta_2}(\tilde \lambda_\ell)$,
\begin{align}\label{e:Zbounds}
\sup_{a,b \in \unn{1}{N}} \sup_{w \in [0,1]} \one_{\B \cap \mathcal E} \left|  
Z \left(\Theta^{(a,b)}_w H\right)  \right| 
&\le C N^{ 2\mathfrak{s} + 2\mathfrak{a} + 8 \eps_2}  \left( \frac{\hat k}{N}\right)^{1/3} \left( \frac{\hat \ell}{N}\right)^{1/3},\\
\sup_{a,b \in \unn{1}{N}} \sup_{j \in \unn{1}{5}}  \sup_{w \in [0,1]} \one_{\B \cap \mathcal E} 
 \left| \partial^k_{ab} Z \left(\Theta^{(a,b)}_w H\right)
\right| 
&\le CN^{2\mathfrak{s} +2\mathfrak{a}+ 20k ( \omega  +  \eps_2)}   \left( \frac{\hat k}{N}\right)^{1/3} \left( \frac{\hat \ell}{N}\right)^{1/3}.
\end{align}
Here $\mathcal E = \mathcal E (\mathfrak{a} + \mathfrak{s}, \q, I)$.
We also have the bounds
\begin{align}\label{e:Ztrivial}
\sup_{a,b \in \unn{1}{N}} \sup_{w \in [0,1]}  \left|  
Z \left(\Theta^{(a,b)}_w H\right)  \right| 
&\le C N^{10},
\\
\sup_{a,b \in \unn{1}{N}} \sup_{j \in \unn{1}{5}}  \sup_{w \in [0,1]} 
 \left| \partial^k_{ab} Z \left(\Theta^{(a,b)}_w H\right)
\right| 
&\le CN^{20j} .
\end{align}
\eel
\begin{rmk}
In this lemma and the next, it is essential that the dependence on $\mathfrak{a}$ in the bounds on the high-probability set is $N^{2\mathfrak{a}}$, and not some higher power of $N^\mathfrak{a}$. Otherwise, the comparison argument in the next section would not go through.
\end{rmk}

Given \Cref{l:obbounds}, the proof of the following theorem is nearly identical to that of \cite{benigni2020optimal}*{Lemma 4.13}, so we omit it.
\bel\label{l:Tderiv}
Let $H$ be a generalized Wigner matrix. Fix $\delta_1, \eps_1, \delta_2, \eps_2, \omega\in (0,1)$, $\mathfrak{r}, \mathfrak{s}, \mathfrak{a} \in [0,1)$, a family of orthogonal vectors $\q=(\q_\alpha)_{\alpha\in \unn{1}{N}}$ in $\S^{N-1}$, and suppose $\delta_1 > \eps_1$.
Let $I \subset \unn{1}{N}$ satisfy $| I  | \le N^{\mathfrak{r} + \mathfrak{s}}$,
and denote $\B= \B(\omega,\delta_1,\varepsilon_1,\delta_2,\q)$, $\mathcal E = \mathcal E (\mathfrak{a} + \mathfrak{s}, \q, I)$.
Then there exists $C = C(\eps_2) >0$ such that for all $k,\ell \in \unn{1}{N}$ and $j \in \unn{1}{5}$, we have
\beq\label{e:Tderiv}
\sup_{a,b,c,d \in \unn{1}{N} } \sup_{0\leqslant w \leqslant 1}  \one_{\mathcal B\cap \mathcal E} \left| \partial_{ab}^j T(\Theta^{(c,d)}_w H) \right| 
\le C N^{ 2\mathfrak{a}+2\mathfrak{s} + 40 j  (\omega  + \delta_1 + \eps_1 + \delta_2 + \eps_2)},  \eeq
and
\beq\label{e:Tderiv2}
 \sup_{a,b,c,d \in \unn{1}{N} } \sup_{0\leqslant w \leqslant 1}   \left| \partial_{ab}^j T(\Theta^{(c,d)}_w H) \right| \le C N^{Cj},
\eeq
where we recall $T(M)$ was defined in \eqref{e:Tdef}.
\eel

\section{\texorpdfstring{Proof of \Cref{t:main2}}{Proof of Theorem 1.4}}\label{s:t2}

In this section and the next, we fix the choice of parameters made in the statement of \Cref{l:dynamichighmom}; additionally, we will fix $\omega$ and $\mathfrak{s}$. Given $\delta >0$, we set
\beq\label{e:paramchoice}
\delta_1 = \delta, \quad  \eps_2 = \delta_1/10^2, \quad \delta_2 = \delta_1/10^3, \quad \eps_1 = \delta_1/10^4, \quad  \omega = \mathfrak{s} =\delta_1/ 10^5%, \quad \mathfrak{b} = \nu/2.
\eeq
We remark that \eqref{e:paramchoice} implies $\delta_1 > \eps_2 > \delta_2 > \eps_1$, and therefore these choices satisfy the hypotheses of all results in the previous section, including \Cref{l:good2} and \Cref{l:boundregeig}. (For  \Cref{l:boundregeig}, we also need that $\mathfrak{a}$ is small relative to $\eps_2, \delta_2$).
Given the definitions in \eqref{e:paramchoice}, $T(H)$ is determined by the choice of eigenvalue indices $k, \ell \in \unn{1}{N}$, $I_N \subset\unn{1}{N}$, and the free parameter $\delta$. For the proof of the next lemma, we follow closely the induction scheme used to prove \cite{benigni2020optimal}*{Lemma 5.1}.

\bel \label{l:newcomparison}
%There exists $\delta >0$ such that the following holds.
Fix $\mathfrak{r}, \mathfrak{b} \geqslant 0$, a family of orthogonal vectors $\q=(\q_\alpha)_{\alpha\in \unn{1}{N}}$ in $\S^{N-1}$, and let $H$ be a generalized Wigner matrix.
Then for every  $ n \in \N$, there exists $\delta(\mathfrak{b})>0$ and $C(\mathfrak{b},n)>0$ such that if $H$ satisfies the bootstrap hypothesis $\mathcal Q (\mathfrak{r}, (21/40) \mathfrak{b} ,\q)$, then
for all $k, \ell \in \unn{1}{N}$ and all $I_N \subset \unn{1}{N}$ satisfying $|I_N| \le N^{\mathfrak{r} + \mathfrak{s}}$, we have
\beq\label{e:conclusion1}
\sup_{a,b \in \unn{1}{N} } \sup_{w \in [0,1]} \E \left[ T_n \left(\Theta^{(a,b)}_w H \right)\right]  \le C N^{ n( 1 + 10^{-5}) \mathfrak{b}}.
\eeq
Here $\mathfrak{s}$ and the parameters in the definition of $T_n$ are chosen as in \eqref{e:paramchoice} in terms of $\delta$.
\eel
\bp
Set $s_1 = N^{-\mathfrak{b}/4}$ as in Lemma \ref{l:dynamichighmom} and let $H_s$ denote the dynamics defined in \eqref{e:dysondyn}.
By \cite{erdos2017dynamical}*{Lemma 16.2},
 there exists a generalized Wigner matrix $H_0$ such that the matrix $Q= H_{s_1}$ satisfies $\E[h_{ij}^k] = \E[q_{ij}^k]$ for $k\in\unn{1}{3}$ and $\left| \E[h_{ij}^4 ]  -  \E[q_{ij}^4 ] \right| \le C N^{-2} s_1$, where $C>0$ is a constant depending only on the constants used to show that $H_0$ satisfies \Cref{d:wigner}. Fix a bijection
\beq
\psi \colon \{ ( i, j) : 1 \le i \le j \le N\} \rightarrow \unn{1}{\gamma_N},
\eeq
where $\gamma_N = N ( N + 1) /2$, and define the matrices $H^1, H^2, \dots , H^{\gamma_N}$ by 
\beq
h_{ij}^\gamma = 
\begin{cases}
h_{ij} & \text{if } \psi(i,j) \leq \gamma
\\
q_{ij} & \text{if } \psi(i,j) > \gamma
\end{cases}
\eeq
for $i \le j$.

Fix any $\gamma \in \unn{1}{\gamma_N}$ and consider the index pair $(i,j)$ such that $\psi(i,j) = \gamma$. For any $p \ge 1$, Taylor expanding $T_p \left(H^\gamma\right)$ in the $(i,j)$ entry gives

\begin{align}
\label{e:taylora11} T_p \left( H^\gamma \right) - T_p \left(  \Theta^{(i,j)}_0 H^\gamma \right) &= \partial T_p \left(  \Theta^{(i,j)}_0 H^\gamma \right) h_{ij} + \frac{1}{2!} \partial^2 T_p \left(  \Theta^{(i,j)}_0 H^\gamma \right) h_{ij}^2 
+ \frac{1}{3!}\partial^3 T_p \left(  \Theta^{(i,j)}_0 H^\gamma \right) h_{ij}^3 \\
& \label{e:taylora21} + \frac{1}{4!}\partial^4 T_p \left(  \Theta^{(i,j)}_0 H^\gamma \right) h_{ij}^4 + \frac{1}{5!}\partial^5 T_p \left(  \Theta^{(i,j)}_{w_1(\gamma)} H^\gamma \right) h_{ij}^5,
\end{align}
where we write $\partial = \partial_{ij}$ and $w_1(\gamma) \in [0,1]$ is a random variable depending on $h_{ij}$. Similarly, we  expand $T_p \left( H^{\gamma -1} \right)$ and obtain
\begin{align}
\label{e:taylorb11} T_p \left( H^{\gamma -1} \right) - T_p \left(  \Theta^{(i,j)}_0 H^\gamma \right) 
&= \partial T_p \left(  \Theta^{(i,j)}_0 H^\gamma \right) q_{ij} + \frac{1}{2!} \partial^2 T_p \left(  \Theta^{(i,j)}_0 H^\gamma \right) q_{ij}^2 
+ \frac{1}{3!}\partial^3 T_p \left(  \Theta^{(i,j)}_0 H^\gamma \right) q_{ij}^3 \\
& \label{e:taylorb21} + \frac{1}{4!}\partial^4 T_p \left(  \Theta^{(i,j)}_0 H^\gamma \right) q_{ij}^4 + \frac{1}{5!}\partial^5 T_p \left(  \Theta^{(i,j)}_{w_2(\gamma)} H^\gamma \right) q_{ij}^5,
\end{align}
where $w_2(\gamma) \in [0,1]$ is a random variable depending on $q_{ij}$. Subtracting \eqref{e:taylorb11} and \eqref{e:taylorb21} from \eqref{e:taylora11} and \eqref{e:taylora21} gives
\begin{align}
 \label{e:fourthorder1} \E \left[ T_p \left( H^{\gamma } \right)  \right]- \E \left[ T_p \left(   H^{\gamma -1} \right) \right] &= \frac{1}{4!} \E \left[ \partial^4 T_p \left(  \Theta^{(i,j)}_0 H^\gamma \right) h_{ij}^4 \right] - \frac{1}{4!}\E \left[ \partial^4 T_p \left(  \Theta^{(i,j)}_0 H^\gamma \right) q_{ij}^4 \right] \\
  \label{e:fifthorder1} &+ \frac{1}{5!} \E \left[ \partial^5 T_p \left(  \Theta^{(i,j)}_{w_1(\gamma)} H^\gamma \right) h_{ij}^5\right]
  - \frac{1}{5!} \E \left[ \partial^5 T_p \left(  \Theta^{(i,j)}_{w_2(\gamma)} H^\gamma \right) q_{ij}^5 \right].\end{align}
Here we took expectation and used that $\Theta^{(i,j)}_0 H^\gamma$ is independent from $h_{ij}$ and $q_{ij}$, and $\E [ h^k_{ij} ] = \E [ q^k_{ij} ]$ for $k \in \unn{1}{3}$.

By \Cref{l:dynamichighmom} and our choice of $s_1$, there exists $C_0(\delta, \mathfrak{b}, n) > 0$ such that for any $p \in \N$ with $p \le n$,
\beq\label{e:TRn}
\E \left[ T_{ p } (Q) \right] \le C_0N^{ p ( 1 + 10^{-5}) \mathfrak{b}}.
\eeq

We will now use \eqref{e:TRn}, \eqref{e:fourthorder1}, and \eqref{e:fifthorder1} to show that $\E \left[ T_p (H) \right] \le 2 C_0 N^{ p ( 1 + 10^{-5}) \mathfrak{b}}$ for all $p \le n$.
We use induction, with the induction hypothesis at step $p \in \N$ being that 
\beq\label{e:inductionhypo1}
 \E \left[ T_q \left(  \Theta^{(a,b)}_{w } H^\gamma \right)\right]  \le 3 C_0 N^{ q( 1 + 10^{-5}) \mathfrak{b}}.
\eeq
holds for all $0 \le q \le p$ and choices of  $w \in [ 0 ,1 ]$ and $(a,b) \in \unn{1}{N}^2$.

The base case $p = 0$ is trivial. Assuming the induction hypothesis holds for $p -1$, we will show it holds for $p$. Using the independence of $h_{ij}$ and $q_{ij}$ from $ \Theta^{(i,j)}_0 H^\gamma$, we write the first two terms on the right side of \eqref{e:fourthorder1} as 
\beq\label{613}
\E \left[ \partial^4 T_p \left(  \Theta^{(i,j)}_0 H^\gamma \right) h_{ij}^4 \right] - \E \left[ \partial^4 T_p \left(  \Theta^{(i,j)}_0 H^\gamma \right) q_{ij}^4 \right] =  \E \left[ \partial^4 T_p \left(  \Theta^{(i,j)}_0 H^\gamma \right)  \right]  \E \left[  h_{ij}^4 - q_{ij}^4  \right].
\eeq
For the second factor in \eqref{613}, we compute $\left|  \E \left[ h_{ij}^4\right] - \E \left[ q_{ij}^4 \right]  \right|  \le C N^{-2}s_1 =  C N^{-2 - \mathfrak{b}/4}$. For the first, we compute 
\begin{align}\label{e:4thcompute1}
\partial^4 T_p  = \partial^4 \left( T^p \right) &= p T_{p-1} T^{(4)}  + 3 p (p - 1 ) T_{p-2} (T^{(2)})^2 + p (p-1) (p-2) (p-3) T_{p-4} (T')^4 \\
& + 4 p (p-1) T_{p-2} T^{(1)} T^{(3)} + 6 p (p-1) (p-2) T_{p-3} (T')^2 T^{(2)}.
\end{align}
Set $\mathfrak{a} = (21/40) \mathfrak{b}$ and $\F= \B(\omega,\delta_1,\varepsilon_1,\delta_2,  \q)\cap \mathcal E (\mathfrak{a} + \mathfrak{s}, \q, I)$. Using  $T_q \ge 0$ and the induction hypothesis \eqref{e:inductionhypo1} for $q \le p -1$, and \eqref{e:Tderiv}, we find that
\beq\label{e:4analogy1}
\left|  \E \left[ \one_{\mathcal F} \partial^4 T_p \left(  \Theta^{(i,j)}_0 H^\gamma \right)  \right]  \right| \le K C_0  N^{p ( 1 + 10^{-5})\mathfrak{b} + \mathfrak{b}/5  + 8\mathfrak{s}+160(\omega + \delta_1 + \eps_1 + \delta_2 + \eps_2)}.
\eeq
for some $K(\delta, \mathfrak{b}, n) > 0$. To see this, we note the worst term in \eqref{e:4thcompute1} is the one containing $T_{p-4} (T')^4$, which we bound as 
\begin{align}
\E\left[  \one_{\mathcal F} T_{p-4} (T')^4\right]
&\le K N^{ 8\mathfrak{a} + 8\mathfrak{s} + 160  (\omega  + \delta_1 + \eps_1 + \delta_2 + \eps_2)} \E\left[  \one_{\mathcal F} T_{p-4} \right]\\ &\le 
KC_0  N^{ (p -4) ( 1 + 10^{-5})\mathfrak{b} + 8\mathfrak{a} + 8\mathfrak{s} + 160  (\omega  + \delta_1 + \eps_1 + \delta_2 + \eps_2)}\\
&\le 
KC_0  N^{ p ( 1 + 10^{-5})\mathfrak{b} + \mathfrak{b}/5 + 8\mathfrak{s} + 160  (\omega  + \delta_1 + \eps_1 + \delta_2 + \eps_2)}
\end{align}
using the hypothesis that $8 \mathfrak{a} \le (4 + 1/5)\mathfrak{b}$. 

Next, by \eqref{e:Tderiv2}, \Cref{l:good1}, \Cref{l:bootstrap}, and the bootstrap assumption $\mathcal Q(\mathfrak{r},\mathfrak{a},\q)$, we find 
\beq\label{e:4analogytrivial1}
\left|\E \left[  \one_{\mathcal F^c}  \partial^4 T_p \left(  \Theta^{(i,j)}_0 H^\gamma \right)  \right]  \right| \le K N^{-2}
\eeq
where we increased $K(\delta, \mathfrak{b}, n)$ if needed.\footnote{The constants in the probability bound given by \Cref{l:good1} do not depend on the choice of $\gamma$, since the $H^\gamma$ satisfy \Cref{d:wigner} simultaneously for a single choice of constants. Therefore, the constant in \eqref{e:4analogytrivial1} is uniform in $\gamma$.}

 It follows from \eqref{613}, \eqref{e:4analogy1} and \eqref{e:4analogytrivial1} that if $\delta = \delta_1$ is chosen small enough in a way that depends only on $\mathfrak{b}$, so that (recall \eqref{e:paramchoice})
\beq\label{e:setparams}
10\mathfrak{s}+200(\omega  + \delta_1 + \eps_1 + \delta_2 + \eps_2)<  \mathfrak{b}/10^3,
\eeq
  then 
\beq\label{e:4final1}
\left| \E \left[ \partial^4 T_p \left(  \Theta^{(i,j)}_0 H^\gamma \right) h_{ij}^4 \right] - \E \left[ \partial^4 T_p \left(  \Theta^{(i,j)}_0 H^\gamma \right) q_{ij}^4 \right] \right| \le  K C_0  N^{-2 - \mathfrak{b}/21 + p( 1 + 10^{-5})\mathfrak{b}}.
\eeq
holds; we now fix $\delta(\mathfrak{b})$ so that \eqref{e:setparams} holds. (The bound \eqref{e:setparams} is stronger than necessary for the fourth order term, but will be needed for the fifth order remainder term.) Therefore, if $ N \ge N_0 (\delta, \mathfrak{b}, n)$, then
\beq\label{e:4final2}
\left| \E \left[ \partial^4 T_p \left(  \Theta^{(i,j)}_0 H^\gamma \right) h_{ij}^4 \right] - \E \left[ \partial^4 T_p \left(  \Theta^{(i,j)}_0 H^\gamma \right) q_{ij}^4 \right] \right| \le  \frac{1}{2} C_0 N^{-2 + p( 1 + 10^{-5})\mathfrak{b}}.
\eeq

Let $\mathcal D$ be the event where $\sup_{i,j} | q_{ij} | + | h_{ij} |\le C N^{-1/2 + \delta_1}$ holds. Since the variables $q_{ij}$ and $h_{ij}$ have all moments finite, for every $D>0$ there exists $C(D)>0$ such that
\beq\label{e:Csubexp}
\P\left( \mathcal D^c \right) \le C N^{-D}.
\eeq
For the terms in \eqref{e:fifthorder1}, we compute
\begin{align}
\left| \E \left[ \partial^5 T_p \left(  \Theta^{(i,j)}_{w_1(\gamma)} H^\gamma \right) h_{ij}^5\right] \right| 
&\le \left| \E  \left[  \one_{\mathcal D} \partial^5 T_p \left(  \Theta^{(i,j)}_{w_1(\gamma)} H^\gamma \right) h_{ij}^5\right] \right| + \left| \E \left[ \one_{\mathcal D^c} \partial^5 T_p \left(  \Theta^{(i,j)}_{w_1(\gamma)} H^\gamma \right) h_{ij}^5\right] \right|\\
&\le C N^{-5/2 + 5 \delta_1} \left( \E \left[ \left| \partial^5 T_p \left(  \Theta^{(i,j)}_{w_1(\gamma)} H^\gamma \right) \right| \right]  + 1\right),
\end{align}
for some $C(\delta_1) > 0$. In the last line, we used \eqref{e:Csubexp} and the inequality of \eqref{e:Tderiv2}. An analogous bound holds for the second term in \eqref{e:fifthorder1}. 

Similarly to the proof of \eqref{e:4analogy1} and \eqref{e:4analogytrivial1},  using \eqref{e:setparams} we find that there exists $C(\delta)>0$ such that
\begin{align}\label{e:5finala}
\Bigg|
\E \left[ \partial^5 T_p \left(  \Theta^{(i,j)}_{w_1(\gamma)} H^\gamma \right) h_{ij}^5\right]
  -   \E &\left[ \partial^5 T_p \left(  \Theta^{(i,j)}_{w_2(\gamma)} H^\gamma \right) q_{ij}^5 \right]
  \Bigg|\\
&\le  K C_0  N^{-5/2 + 5 \delta_1 + p ( 1 + 10^{-5})\mathfrak{b} + 10\mathfrak{a} + 10\mathfrak{s} + 200(\omega  + \delta_1 + \eps_1 + \delta_2 + \eps_2)}\\
& \le K C_0  N^{-2 - 1/8 + p ( 1 + 10^{-5})\mathfrak{b}} ,
\end{align}
where we increased the value of $K$, if necessary.
Therefore, if $ N \ge N_0 (\delta, \mathfrak{b}, n)$, then
\beq\label{e:5final2a}
\left|
\E \left[ \partial^5 T_p \left(  \Theta^{(i,j)}_{w_1(\gamma)} H^\gamma \right) h_{ij}^5\right]
-  \frac{1}{5!} \E \left[ \partial^5 T_p \left(  \Theta^{(i,j)}_{w_2(\gamma)} H^\gamma \right) q_{ij}^5 \right]
\right|
\le \frac{1}{2} C_0 N^{-2 + p( 1 + 10^{-5})\mathfrak{b}}.
\eeq

Together \eqref{e:4final2} and \eqref{e:5final2a} yield
\beq\label{e:gammainc}
\left| \E \left[ T_p \left( H^{\gamma } \right)  \right]- \E \left[ T_p \left(   H^{\gamma -1} \right) \right] \right| \le C_0 N^{- 2 + p( 1 + 10^{-5})\mathfrak{b}},
\eeq
and summing \eqref{e:gammainc} over all $\gamma_N$ pairs $(i,j)$, we find
\beq\label{e:4momentconclude1}
\left| \E \left[ T_p \left( Q \right)  \right]- \E \left[ T_p \left( H^\gamma \right) \right] \right| \le C_0 N^{p( 1 + 10^{-5})\mathfrak{b}}
\eeq
for any $\gamma$. Using \eqref{e:TRn} and \eqref{e:4momentconclude1}, we obtain 
\beq
\E\left[ T_p\left( H^\gamma \right) \right] \le 2 C_0 N^{p( 1 + 10^{-5})\mathfrak{b}}.
\eeq
 This verifies the induction hypothesis \eqref{e:inductionhypo1} when $w = 1$. 

To address other values of $w$, we consider the following expansion:
\begin{align}
\label{e:taylorc1a} T_p \left( H^\gamma \right) - T_p \left(  \Theta^{(a,b)}_w H^\gamma \right) &= \partial T_p \left(  \Theta^{(a,b)}_0 H^\gamma \right) h_{ij} + \frac{1}{2!} \partial^2 T_p \left(  \Theta^{(a,b)}_0 H^\gamma \right) h_{ij}^2 
+ \frac{1}{3!}\partial^3 T_p \left(  \Theta^{(a,b)}_0 H^\gamma \right) h_{ij}^3 \\
& \label{e:taylorc2a} + \frac{1}{4!}\partial^4 T_p \left(  \Theta^{(a,b)}_0 H^\gamma \right) h_{ij}^4 + \frac{1}{5!}\partial^5 T_p \left(  \Theta^{(a,b)}_{\tau(w)} H^\gamma \right) h_{ij}^5,
\end{align}
Here $\tau(w) \in [ 0, 1]$ is a random variable. The same argument that gave  \eqref{e:5finala} shows that the right side of \eqref{e:taylorc1a} and \eqref{e:taylorc2a} may be bounded in absolute value by $ C_0 N^{p ( 1 + 10^{-5})\mathfrak{b}}$ for $N \ge N_0(\delta, \mathfrak{b}, n)$. %The second term of \eqref{e:taylorc2a} is also bounded in absolute value by $\frac{1}{2}C_0 N^{p( 1 + 10^{-5})\delta_3}$. 
(Note that we are only controlling a single difference, instead of the sum of $N^2$ differences as before.)
We  conclude
\beq
\sup_{w \in [0,1]} \sup_{a,b \in \unn{1}{n}} \E \left[ T_p \left(  \Theta^{(a,b)}_{w } H^\gamma \right)\right]  \le 3 C_0 N^{p( 1 + 10^{-5})\mathfrak{b}}.
\eeq
This finishes the induction step and completes the proof.
\ep

\bp[Proof of \Cref{t:main2}]
Set $\mathfrak{b} = \eps$, where $\eps$ was fixed in the hypotheses of the theorem. Fix $\delta(\mathfrak{b})>0$ to the value given by \Cref{l:newcomparison}. Observe that this also fixes $\mathfrak{s} >0$, via \eqref{e:paramchoice}.

Note that by eigenvector delocalization \eqref{e:rigidity}, $H$ satisfies $\mathcal Q(0, (21/40)\mathfrak{b},\q)$.
With this as the base case, we will use induction to complete the proof;
in particular, we will show that if $H$ satisfies $\mathcal Q(\mathfrak{r}, (21/40)\mathfrak{b},\q)$, then it satisfies $\mathcal Q(\mathfrak{r} + \mathfrak{s}, (21/40)\mathfrak{b}, \q)$. This induction terminates when it reaches index sets $I$ of size $|I| \ge N$, at which point the proof is complete by the definition of $\mathcal Q$ and $\mathfrak{b}$. Various constants will increase at each step, but since the induction terminates in a finite number of steps (a number which depends only on $\mathfrak{b}$), these increases are not problematic. 

Now suppose $H$ satisfies $\mathcal Q(\mathfrak{r}, (21/40) \mathfrak{b} ,\q)$. By \Cref{l:newcomparison},
for every $I_N \subset \unn{1}{N}$ with $|I_N| \le N^{\mathfrak{r}+\mathfrak{s}}$,  $k, \ell \in \unn{1}{N}$, there exists $C( \mathfrak{b},n)>0$ such that \eqref{e:conclusion1} holds. By applying Markov's inequality with a sufficiently high moment $n$, this shows that for every $ D>0$, there exists $C(D,\mathfrak{b})>0$ such that 
\begin{equation}\label{returnto}
\sup_{a,b \in \unn{1}{N} } \sup_{k,\ell} \P \left( v \left(\Theta^{(a,b)}_w H\right)  \ge N^{( 1 + 2\cdot 10^{-5}) \mathfrak{b}} \right) \le C N^{-D}
\end{equation}
for any fixed $w \in [ 0 ,1]$, where recall that $T(M) = v(M, k, \ell)$, and $v$ is defined using $I_N$.
A union bound gives
\beq\label{prev2}
\P\left( \bigcup_{a,b \in \unn{1}{N} }\bigcup_{k,\ell \in \unn{1}{N}} 
 \left\{  v \left(\Theta^{(a,b)}_w H\right)  \ge N^{( 1 + 2\cdot 10^{-5})\mathfrak{b}} \right \} \right) 
 \le C N^{-D}.
 \eeq
 We now claim that we can take the union over $w\in [ 0, 1]$ inside the parentheses in \eqref{prev2}. In particular, we claim that a standard stochastic continuity argument reduces this task to taking a union over $N^C$ discrete values $w \in [0,1]$, which shows the result after increasing $D$ appropriately. Indeed, $v \left(\Theta^{(a,b)}_wH,I_N \right)$ is $N^C$-Lipchitz continuous in $w$ by \eqref{e:Tderiv2} applied with $(a,b) = (c,d)$. We conclude that
 \beq\label{prev22}
\P\left( \bigcup_{w \in [ 0 ,1]} \bigcup_{a,b \in \unn{1}{N} }\bigcup_{k,\ell \in \unn{1}{N}} 
 \left\{  v \left(\Theta^{(a,b)}_w H\right)  \ge N^{( 1 + 2\cdot 10^{-5}) \mathfrak{b}} \right \} \right) 
 \le C N^{-D}.
 \eeq
 
 Then using \eqref{e:boundedreclaim4} with $\mathfrak{a} = (21/40)\mathfrak{b}$, \eqref{simultaneous},  the induction hypothesis $\mathcal Q(\mathfrak{r}, (21/40)\mathfrak{b} ,\q)$ to bound the  probability $\P\left( \mathcal E (\mathfrak{a},\q, I)^c\right)\le CN^{-D}$, and a union bound,
we have
\beq\label{prev}
\P\left( \bigcup_{w \in [ 0 ,1]}   \bigcup_{a,b \in \unn{1}{N} }\bigcup_{k,\ell \in \unn{1}{N}} 
 \left\{ \left| \hat p_{k \ell} (\Theta^{(a,b)}_wH,I_N)\right| \ge  N^{ (21/40)\mathfrak{b}} \right \} \right) \le C N^{-D},
 \eeq
 where $C(D)$ is independent of the choice of $I_N$. We used that the exponent of $N$ in the error term of \eqref{e:boundedreclaim4} is strictly less than $2 \mathfrak{a}$ with our choice of parameters \eqref{e:paramchoice} and that $1/2+10^{-5}<21/40$. This proves that $\Q(\mathfrak{r} + \mathfrak{s}, (21/40)\mathfrak{b}, \q)$ holds, completing the induction step (and the proof).
 \ep

\section{\texorpdfstring{Proof of \Cref{t:main1}}{Proof of Theorem 1.2}}\label{s:t1}

We begin by defining a new regularized observable.

\bed\label{d:regvect2}
Suppose $M \in \matn$, and fix $\delta_1, \eps_1, \delta_2, \eps_2 >0$ and a family of orthogonal vectors $\q=(\q_\alpha)_{\alpha\in \unn{1}{N}}$ in $\S^{N-1}$. Let $\tilde \lambda_i = \tilde \lambda_{i, \delta_1, \eps_1}(M)$ denote the regularized eigenvalues of \Cref{p:regulareigval}. 
For $\ell, \alpha \in\unn{1}{N}$, we define the regularized and centered eigenvector entries of $M$ by
\beq\label{e:regvect}
v_\ell(\alpha ) = v_\ell(M, \alpha, \delta_1, \eps_1, \delta_2,\varepsilon_2) = \frac{1}{\pi}
\int_{\hat I_{\delta_2}(\tilde{\lambda}_\ell)}\Im \left( \scp{\q_\alpha}{G(E+\I\eta_\ell)\q_\alpha} -  m_N(E +\I \eta_\ell) \right)\, \d E,
\quad
\eta_\ell = \frac{N^{-\varepsilon_2}}{N^{2/3} \hat \ell^{1/3}}.
\eeq 
We define the regularized ergodicity observables for $M$ by
\begin{equation}
q_{\ell \ell }(I_N)  =  q_{\ell \ell } (M, I_N,\q) = \frac{N}{ \sqrt{ | I_N |} }  \sum_{\alpha\in I_N} v_\ell(\alpha) .
\end{equation}

\eed
The proof of the following is similar to \Cref{l:boundregeig} and uses \Cref{t:main2}; we omit it.
\bel\label{l:boundregeig2}
Let $H$ be a generalized Wigner matrix. Fix $\omega,\delta_1,\varepsilon_1,\delta_2, \eps_2> 0$ with $\delta_1 > \eps_1$ and $\eps_2 > \delta_2$ and a family of orthogonal vectors $\q=(\q_\alpha)_{\alpha\in I}$ in $\S^{N-1}$. Suppose $I_N \subset \unn{1}{N}$.
Then  
for all $\ell\in\unn{1}{N}$, we have
\beq\label{e:quecompare}
\one_{\tilde \B \cap \mathcal E} q_{\ell \ell}(I_N)
=
\one_{\tilde \B \cap \mathcal  E} \hat p_{\ell \ell} (I_N)
+ \mathcal O_{\omega, \eps_2} \left(
N^{3\omega  + \delta_2 - \eps_2} 
+ N^{\omega + \eps_1+ 4 \epsilon_2 - \delta_1}
\right) 
\eeq
where $\tilde \B = \tilde \B(\omega,\delta_1,\varepsilon_1,\delta_2, \eps_2,0,\q, \ell)$ and $\mathcal E = \mathcal E ( \omega,\q, I_N)$.
\eel
\begin{rmk}
Since we take $\delta_3=0$ in the previous lemma, we can no longer appeal to \Cref{l:good2} to show that $\tilde B$ holds with high probability. Instead, we will use \Cref{l:lr2} below. The estimate on $\P(\tilde \B^c)$ it yields is weaker, but still sufficient for our purposes.
\end{rmk}

The following bounds are proved in \Cref{s:a1}.
\bel\label{l:qintegrands2}
Fix $ \omega, \delta_1, \eps_1, \delta_2, \eps_2>0$ such that $\delta_1 > \eps_1$, a family of orthogonal vectors $\q=(\q_\alpha)_{\alpha\in \unn{1}{N}}$ in $\S^{N-1}$, and $\ell \in \unn{1}{N}$.  
Let $H$ be a generalized Wigner matrix, and let $\B = \B(\omega,\delta_1,\varepsilon_1,\delta_2, \eps_2,\q)$ be the set from \Cref{d:goodset}. Let $\eta_\ell =\frac{N^{-\varepsilon_2}}{N^{2/3} \hat \ell^{1/3}}$ and suppose $I_N \subset \unn{1}{N}$.
Then there exists a constant $C= C(\eps_2)>0$ such that for all $E \in I_{\delta_2}(\lambda_\ell) \cup I_{\delta_2}(\tilde \lambda_\ell)$,

\beq\label{e:intque}
\sup_{a,b \in \unn{1}{N}} \sup_{w \in [0,1]} \one_{\B \cap \mathcal E} \frac{1}{ \sqrt{| I_N|}} \left| \sum_{ \alpha \in I_N} \Im \left( \Theta^{(a,b)}_w\scp{\q_\alpha}{G(E+\I\eta_\ell)\q_\alpha} -   \Theta^{(a,b)}_w m_N(E + \I\eta_\ell) \right)\right| 
\le C N^{  \omega + 4 \eps_2}   \left( \frac{\hat \ell}{N}\right)^{1/3},
\eeq

\begin{multline}
\sup_{a,b \in \unn{1}{N}} \sup_{j \in \unn{1}{5}}  \sup_{w \in [0,1]} \one_{\B \cap \mathcal E} \frac{1}{ \sqrt{| I_N|}} \left| \partial^k_{ab} \sum_{ \alpha \in I_N} \Im \left( \Theta^{(a,b)}_w\scp{\q_\alpha}{G(E+\I\eta_\ell)\q_\alpha} -   \Theta^{(a,b)}_w m_N(E + \I\eta_\ell) \right)\right| 
\\ \le CN^{10 k (  \omega +  \eps_2)}   \left( \frac{\hat \ell}{N}\right)^{1/3}.
\end{multline}
Here $\mathcal E = \mathcal E ( \omega, \q , I_N)$.
\eel

 For $M \in \matn$, $I_N \subset \unn{1}{N}$, $\ell \in \unn{1}{N}$, and $n \in \N$, we define 
\beq\label{e:Sdef}
S(M) = S(M,  \ell) = N   q_{\ell \ell}(M, I_N, \q), \qquad S_n(M) = S(M)^n.
\eeq 
Given \Cref{l:qintegrands2}, the proof of the following lemma is essentially identical to the proof of \Cref{l:obbounds}, so we omit it. %We do not need to carefully track the powers of $N^\omega$ here, as was done previously.
\bel\label{l:Sderiv}
Let $H$ be a generalized Wigner matrix. Fix $\eps_1, \delta_1, \eps_2, \delta_2, \omega \in (0,1)$,
suppose $\delta_1 > \eps_1$,
and denote $\B= \B(\omega,\delta_1,\varepsilon_1,\delta_2,\q)$, $\mathcal E = \mathcal E (\omega,\q, I_N)$.
Then there exists $C = C(\eps_2) >0$ such that for all $k,\ell \in \unn{1}{N}$ and $j \in \unn{1}{5}$, we have
\beq\label{e:Sderiv}
\sup_{a,b,c,d \in \unn{1}{N} } \sup_{0\leqslant w \leqslant 1}  \one_{\mathcal B\cap \mathcal E} \left| \partial_{ab}^j S (\Theta^{(c,d)}_w H) \right| 
\le C N^{ 40 j  ( \omega  + \delta_1 + \eps_1 + \delta_2 + \eps_2)},  \eeq
and
\beq\label{e:Sderiv2}
 \sup_{a,b,c,d \in \unn{1}{N} } \sup_{0\leqslant w \leqslant 1}   \left| \partial_{ab}^j S (\Theta^{(c,d)}_w H) \right| \le C N^{Cj}.
\eeq
\eel

We also require the following level repulsion estimate.

\bep[{{\cite{benigni2020optimal}*{Proposition 5.7}}}]\label{l:lr2}
Let $H$ be a generalized Wigner matrix. Then there exists $\delta_0 >0$ such that for all $\delta \in (0, \delta_0)$, there exist constants $C = C(\delta)$ and $c=c(\delta) > 0$ such that for any $i \in \unn{1}{N -1 }$, 
\beq\label{e:optlr}
\P\left( \lambda_{i+1} - \lambda_i < \frac{ N^{-\delta} }{N^{2/3} {\hat i}^{1/3} } \right) < C(\delta) N^{-\delta - c}.
\eeq
\eep

Given the previous bounds, the proof of \Cref{t:main1} is very similar to the proof of \Cref{t:main2}. For brevity, we provide only a sketch.

\bp[Sketch of the proof of \Cref{t:main1}]
Let $i_N$ and $\eps$ be as in the statement of the theorem.
We can find $c(\eps,n)>0$ such that the moment convergence for $H_{s_1}$ with $s_1 = N^{-c}$ is true by Corollary \ref{c:quemom}. Note that such a time $s_1$ exists due to our assumption that $\vert I\vert \le N^{1-\varepsilon}$. As in the proof of \Cref{l:newcomparison}, let $R = H_{s_1}$ be a matrix that matches $H$ to three moments exactly, with $N^{-2}s_1$ error in the fourth moment. Using  \Cref{l:Sderiv} and applying four moment matching, we find that for any $n\in \N$, there exists $\delta>0$ such that 
\beq \left| S_n(H, i_N)  - S_n(R, i_N) \right| \le  C N^{-c}\eeq
 for some constants $C(n,\eps), c(n,\eps)>0$, where we choose the parameters for $S$ as in \eqref{e:paramchoice}. Given that we know the desired moment convergence for $R$ by \Cref{c:quemom}, it suffices to show that $\left| S_n(H, i_n)  - \E\left[\hat p^n_{ii} \right]\right| \le  C N^{-c}$ and similarly for $R$. We give the proof for $H$ only, since it will apply to any generalized Wigner matrix.
 
 By \eqref{e:quecompare}, we have 
\beq\label{e:quecompare2}
\one_{\tilde \B \cap \mathcal E} q^n_{ii}(I_N)
=
\one_{\tilde \B\cap \mathcal E} \left( \hat p _{ii} (I_N)
+ \O{N^{-c}  }\right)^n
\eeq
for some $c(\delta) > 0$ and any $n \in \N$. Expanding the product and taking expectation, we  conclude that 
 \begin{equation}
 \E \left[ \one_{\tilde \B \cap \mathcal E} q^n_{ii}(I_N) \right] = \E \left[ \one_{\tilde \B \cap\mathcal E} p^n_{ii}(I_N)\right] + \O{N^{-c/2}}.
 \end{equation}
 after taking $\omega (\delta) > 0$ small enough in $\mathcal E(\omega, \q, I_N)$.

 Finally, we use \Cref{l:good1}, \Cref{l:lr2}, and \Cref{t:main2} to conclude that 
  \begin{equation}
  \E \left[ \one_{ \left(\tilde \B \cap\mathcal E\right)^c} p^n_{ii}(I_N)\right] = \O{N^{-c/2}}.
 \end{equation}
 This completes the proof.
 \ep

\appendix

\section{Resolvent estimates}
\label{s:a1}

The following lemma was proved in \cite{benigni2020optimal}.
\bel[\cite{benigni2020optimal}*{Lemma 4.8}]\label{l:belowG}
Fix $\omega, \delta_1, \eps_1, \delta_2, \eps_2>0$ such that $\delta_1 > \eps_1$ and $\ell \in \unn{1}{N}$.  
Let $H$ be a generalized Wigner matrix, and let $\B = \B(\omega,\delta_1,\varepsilon_1,\delta_2, \eps_2)$ be the set from \Cref{d:goodset}. Let $\eta_\ell =\frac{N^{-\varepsilon_2}}{N^{2/3} \hat \ell^{1/3}}$.
Then there exists a constant $C= C(\eps_2)>0$ such that for all $E \in I_{\delta_2}(\lambda_\ell) \cup I_{\delta_2}(\tilde \lambda_\ell)$,
\beq\label{e:belowG}
\sup_{a,b,c,d \in \unn{1}{N}} \sup_{w \in [0,1]} \one_{\B}\Im \Theta^{(a,b)}_w G_{cd}(E+\I\eta_\ell) \le C N^{4\epsilon_2} \left( \frac{\hat \ell}{N}\right)^{1/3}
\eeq
and
\beq\label{e:belowGreal}
\sup_{a,b,c,d \in \unn{1}{N}} \sup_{w \in [0,1]} \one_{\B}\Re \Theta^{(a,b)}_w G_{cd}(E+\I\eta_\ell) \le C N^{4\epsilon_2}.
\eeq
%Here $\tilde E = \tilde E (1, \omega)$.
\eel

In what follows we write $H$ for $\Theta^{(a,b)}_w H$, and similarly for $G$ and $m_N$, omitting the rank one perturbation from the notation.

\bp[Proof of \Cref{l:obbounds}]
We first observe by the spectral theorem that
\beq\label{e:mixtrace}
| I _N | N^{-2}  \tr  \Im  G  (z_1)  \Im  G  (z_2) = \frac{1}{N^2} \sum_{j=1}^N \frac{ | I _N | \eta_k\eta_\ell  }{ \left( ( \lambda_j - E_1)^2 + \eta^2_k\right) \left( ( \lambda_j - E_2)^2 + \eta^2_\ell  \right) } ,
\eeq
and
\beq\label{e:crossterm}
\frac{2}{N} \sum_{\alpha \in I_N}  \scp{\q_\alpha}{\Im G(z_1)\Im G(z_2)\q_\alpha} = \frac{2}{N} \sum_{j=1}^N \frac{ \eta_k\eta_\ell  \left( p_{jj}+\frac{\vert I\vert}{N}\right)}{ \left( ( \lambda_j - E_1)^2 + \eta^2_k\right) \left( ( \lambda_j - E_2)^2 + \eta^2_\ell  \right) }.
\eeq
%and similarly with $\alpha$ replaced by $\beta$ in \eqref{e:crossterm}.
\begin{comment}
We also see by the bootstrap hypothesis that
\begin{align}
\frac{1}{| I_N| } \sum_{j =1 }^N \frac{ \eta_k\eta_\ell N^2 \left( \sum_{\alpha \in I_N} u_j(\alpha)^2 \right)^2 }{ \left( ( \lambda_j - E_1)^2 + \eta^2_k\right) \left( ( \lambda_j - E_2)^2 + \eta^2_\ell  \right) } 
- \sum_{j=1}^N \frac{ \eta_k\eta_\ell | I_N|  }{ \left( ( \lambda_j - E_1)^2 + \eta^2_k\right) \left( ( \lambda_j - E_2)^2 + \eta^2_\ell  \right) }
\end{align}
\end{comment}

For \eqref{e:Zbounds}, we compute
\begin{align}\label{a37}
  \frac{1 }{ | I_N|}\sum_{ \alpha, \beta \in I_N} &\Im\scp{\q_\alpha}{G(z_1)\q_\beta}  \Im\scp{\q_\alpha}{G(z_2)\q_\beta}\\
&=  \frac{1}{|I_N|}  \sum_{\alpha, \beta \in I_N}
\left( \sum_{j=1}^N \frac{ \eta_k \scp{\q_\alpha}{\u_j} \scp{\q_\beta}{\u_j} }{ ( \lambda_j - E_1)^2 + \eta^2_k }\right)
\left( \sum_{m=1}^N \frac{\eta_\ell \scp{\q_\alpha}{\u_m} \scp{\q_\beta}{\u_m} }{( \lambda_m - E_2)^2 + \eta^2_\ell }\right)\\
&=
\frac{1}{N^2} \sum_{j\neq m}^N \frac{ \eta_k\eta_\ell  \hat p^2_{jm}(I_N) }{ \left( ( \lambda_j - E_1)^2 + \eta^2_k\right) \left( ( \lambda_m - E_2)^2 + \eta^2_\ell  \right) }\\
&+\frac{1}{|I_N|}  \sum_{j=1}^N \frac{ \eta_k\eta_\ell \left(  p_{jj} +  \frac{ | I_N | }{N} \right)^2  }{ \left( ( \lambda_j - E_1)^2 + \eta^2_k\right) \left( \lambda_j - E_2)^2 + \eta^2_\ell  \right) }.
\end{align}
The second term is canceled by \eqref{e:mixtrace} and \eqref{e:crossterm}, except for the $p_{jj}^2$ term, and
\begin{align}
\frac{1}{N^2} \sum_{j,m=1}^N \frac{ \eta_k\eta_\ell \hat p^2_{jm}(I_N) }{ \left( ( \lambda_j - E_1)^2 + \eta^2_k\right) \left( ( \lambda_m - E_2)^2 + \eta^2_\ell  \right) }
&\le N^{2\mathfrak{s} + 2\mathfrak{a}} \left(  \Im m_N(E_1 + \I \eta_k) \right) \left(  \Im m_N(E_2 + \I \eta_\ell) \right)
\label{e:useboot}
\\
& \le N^{2\mathfrak{s} + 2\mathfrak{a} + 8 \eps_2 }\left( \frac{\hat k}{N}\right)^{1/3}\left( \frac{\hat \ell}{N}\right)^{1/3}.
\end{align}
In \eqref{e:useboot}, we used the definition of the boostrap set $\mathcal E$. In the last line, we averaged \eqref{e:belowG} over all indices to bound the imaginary part of $m_N$. This proves the first claim.

The first derivative of \eqref{a37} is, where $(\eb_\alpha)_{\alpha \in\unn{1}{N}}$ is the standard basis, 
\begin{align}
\partial_{ab} &\left( \frac{1 }{ | I_N|}\sum_{ \alpha, \beta \in I_N}\Im\scp{\q_\alpha}{G(z_1)\q_\beta}  \Im\scp{\q_\alpha}{G(z_2)\q_\beta}\right) \\
& = -\frac{1}{|I_N|}  \sum_{\alpha, \beta \in I_N}
\Im \scp{\q_\alpha}{G(z_1)\eb_a} \scp{\eb_b}{G(z_1)\q_\beta}  \Im \scp{\q_\alpha}{G(z_2)\q_\beta}  \\
&- \frac{1}{|I_N|} \sum_{\alpha, \beta \in I_N} \Im\scp{\q_\alpha}{G(z_1)\eb_b}  \scp{\eb_a}{G(z_1)\q_\beta}  \Im \scp{\q_\alpha}{G(z_2)\q_\beta} \\
& - \frac{1}{|I_N|}  \sum_{\alpha, \beta \in I_N}
\scp{\q_\alpha}{G(z_1)\q_\beta} \Im \scp{\q_\alpha}{G(z_2)\eb_a}  \scp{\eb_b}{G(z_2)\q_\beta} \\
& - \frac{1}{|I_N|}\sum_{\alpha, \beta \in I_N}  \scp{\q_\alpha}{G(z_1)\q_\beta} \Im  \scp{\q_\alpha}{G(z_2)\eb_b}  \scp{\eb_a}{G(z_2)\q_\beta} .
\label{e:similar28}
\end{align}

We compute a representative term:
\begin{align}
\frac{1}{|I_N|}  &\sum_{\alpha, \beta \in I_N}
\Im \scp{\q_\alpha}{G(z_1)\eb_a} \scp{\eb_b}{G(z_1)\q_\beta}  \Im \scp{\q_\alpha}{G(z_2)\q_\beta}\\
&=  \frac{1}{|I_N|}  \sum_{\alpha, \beta \in I_N} \Im \left[\left(  \sum_{i=1}^N \frac{\scp{\q_\alpha}{\u_i} \u_i(a) }{\lambda_i - z_1 } \right) \left(  \sum_{j=1}^N \frac{\u_j(b) \scp{\q_\beta}{\u_j} }{\lambda_j - z_1  } \right) \right] \Im \left(  \sum_{m=1}^N \frac{\scp{\q_\alpha}{\u_m} \scp{\q_\beta}{\u_m} }{\lambda_m - z_2 }\right)\\
&=\frac{1}{|I_N|}  \sum_{\alpha, \beta \in I_N}\sum_{i,j=1}^N \sum_{m=1}^N \Im \left[   \frac{ \u_i(a) \u_j(b) }{(\lambda_i - z_1 )(\lambda_j - z_1 )}   \right] \Im \left(   \frac{1 }{\lambda_m - z_2 }\right)
\scp{\q_\alpha}{\u_i} \scp{\q_\alpha}{\u_m} \scp{\q_\beta}{\u_m}\scp{\q_\beta}{\u_j}.
\end{align}
When $i \neq m$ and $j \neq m$, this can be bounded similarly to the previous computation by applying the bootstrap hypothesis and eigenvector delocalization. We also observe the bound
\begin{align}\label{doublesum}
\sum_{i,j} \left| \Im  \frac{ 1 }{(\lambda_i - z_1 )(\lambda_j - z_1 )} \right| 
&= \sum_{i,j} \left| \Im  \frac{ 1 }{(\lambda_i - z_1 )} \Re \frac{1}{(\lambda_j - z_1 )} + \Re  \frac{ 1 }{(\lambda_i - z_1 )} \Im \frac{1}{(\lambda_j - z_1 )}  \right| \\
&\le  2 \sum_{i,j}  \Im  \frac{ 1 }{(\lambda_i - z_1 )}  \frac{1}{| \lambda_j - z_1 |}  \\
&\le 2 N \Im m_N(z_1) \left( \sum_{i} 1 + \frac{1}{| \lambda_i - z_1|^2 }\right)\\
&\le 2 N  \Im m_N(z_1) \left( N \big( 1  + \Im m_N(z_1) \big) \right),
\end{align}
\begin{comment}
\begin{align}
\sum_{i,j} \frac{1}{| \lambda_i - z| } \frac{1}{| \lambda_j - z| } &= \left( \sum_{i} \frac{1}{| \lambda_i - z| }\right)^2
 \le  \left( \sum_{i} 1 + \frac{1}{| \lambda_i - z|^2 }\right)^2
 \le  \bigg( N \big( 1  + \Im m_N(z) \big) \bigg)^2,
\end{align}
\end{comment}
and we can again use $\eqref{e:belowG}$ to bound the imaginary part of $m_N$.

We now consider the other cases. When $i=j=m$ we get 
\begin{align}
&\frac{1}{|I_N|}  \sum_{\alpha, \beta \in I_N} \sum_{i=1}^N \Im \left[   \frac{ \u_i(a) \u_i(b) }{(\lambda_i- z_1 )(\lambda_i - z_1 )}   \right] \Im \left(   \frac{1 }{\lambda_i - z_2 }\right)
\scp{\q_\alpha}{\u_i}^2  \scp{\q_\beta}{\u_i}^2\\
&=
\frac{1}{|I_N|} \sum_{i=1}^N \Im \left[   \frac{ \u_i(a) \u_i(b) }{(\lambda_i- z_1 )(\lambda_i - z_1 )}   \right] \Im \left(   \frac{1 }{\lambda_i - z_2 }\right)
\left( p_{ii}(I_N) + \frac{ | I_N |}{N}  \right)^2.
\end{align}
The $p_{ii}(I_N)^2$ term is negligible, and the $\left( \frac{ | I_N |}{N} \right)^2$ term in the expansion is canceled by the derivative of \eqref{e:mixtrace}. (Note that the cross-term $2 p_{ii}(I_N) | I_N | N^{-1}$ contributes to the $\left( \frac{ | I_N |}{N} \right)^2$ term by the definition of $p_{ii}$.)
We are left with 
\begin{equation}
\frac{2}{N} \sum_{i=1}^N \Im \left[   \frac{ \u_i(a) \u_i(b) }{(\lambda_i- z_1 )(\lambda_i - z_1 )}   \right] \Im \left(   \frac{1 }{\lambda_i - z_2 }\right) \left( p_{ii} + \frac{ | I_N |}{N} \right).
\end{equation}
%\ber What cancels the previous equation?\eer

With $i=m, j \neq m$ or $j =m, i \neq m$, we get cross-terms 
\begin{equation}
 \frac{1}{N} \sum_{i\neq j }^N  \Im \left[   \frac{ \u_i(a) \u_j(b) }{(\lambda_i - z_k )(\lambda_j - z_k )}   \right] \Im \left(   \frac{1 }{\lambda_i - z_\ell }\right)
 p_{ij},\quad
\frac{1}{N} \sum_{i\neq j }^N  \Im \left[   \frac{ \u_i(a) \u_j(b) }{(\lambda_i - z_k )(\lambda_j - z_k )}   \right] \Im \left(   \frac{1 }{\lambda_j - z_\ell }\right)
 p_{ij}.
\end{equation}
Together with the previous sum, these are canceled by the derivative of \eqref{e:crossterm}. The other three terms in \eqref{e:similar28} are similar.

The higher derivatives in \eqref{e:Zbounds} are controlled in the same way. One thing important to note is that differentiating only adds resolvent terms and, in particular, does not add any sums over eigenvector entries. Thus, there can only be a quadratic term for $p_{k\ell}$ which is bounded by $N^{2\mathfrak{a}+2\mathfrak{s}}$ regardless of the order of the derivatives. We omit the straightforward but tedious computations.

The remaining two almost sure bounds are trivial consequences of the inequality $\left| G_{ij}(z) \right| \le (\Im z)^{-1}$.
\ep

\bp[Proof of \Cref{l:qintegrands2}]
We denote $z=E+\I\eta_\ell$ and we use the spectral theorem to see that 
\begin{align}
\frac{1}{ \sqrt{| I_N|}}  \left| \sum_{ \alpha \in I_N} \Im \scp{\q_\alpha}{G(z)\q_\alpha} -  \Im m_N(z) \right| &= \frac{1}{ \sqrt{| I_N|}} \left| \sum_{k=1} \sum_{ \alpha \in I_N}\frac{ \eta_\ell \left(\scp{\q_\alpha}{\u_k}^2 - N^{-1} \right)}{ ( E - \lambda_k)^2 + \eta^2_\ell} \right| \\
&\le  \frac{1}{N} \sum_{k=1} \frac{ \eta_\ell  \left|\hat  p_{kk}(I_N)\right|}{ ( E - \lambda_k)^2 + \eta^2_\ell}\\
& \le N^{-1 + \sigma + \omega}  \sum_{k=1} \frac{ \eta_\ell }{ ( E - \lambda_k)^2 + \eta^2_\ell}\\
& = CN^{ \sigma + \omega}  \Im m_N(z)\label{e:ec1}\\
& \le CN^{\sigma + \omega + 4 \eps_2}   \left( \frac{\hat \ell}{N}\right)^{1/3}.
\end{align}
In the last line we used \eqref{e:belowG}. Here $C = C(\eps_2) >0$ is a constant depending on $\eps_2$.

For the first derivative, we have
\begin{multline}\label{e:firstd}
\partial_{ab} \left( \frac{1}{ \sqrt{| I_N|}} \sum_{ \alpha \in I_N} \Im \scp{\q_\alpha}{G(z)\q_\alpha} -  \Im m_N(z)\right)
\\=
\frac{1}{ \sqrt{| I_N|}} \sum_{ \alpha \in I_N} \left( \Im \scp{\q_\alpha}{G(z)\eb_a}\scp{\eb_b}{G(z)\q_\alpha}  -  \Im \frac{1}{N} \sum_{i=1}^N G_{i a }(z) G_{bi}(z) \right).
\end{multline}
Using the spectral theorem, we compute
\begin{align}
\scp{\q_\alpha}{G(z)\eb_a}\scp{\eb_b}{G(z)\q_\alpha} &= \left(\sum_{k=1}^N  \frac{\scp{\q_\alpha}{\u_k} \u_k(a) }{\lambda_k - z }\right) \left( \sum_{j=1}^N \frac{ \u_j(b) \scp{\q_\alpha}{\u_j} }{\lambda_j - z }\right)\\
&= \sum_{k=1}^N \frac{\scp{\q_\alpha}{\u_k}^2 \u_k(a) \u_k(b) }{(\lambda_k - z )^2} + \sum_{k \neq j} \frac{ \scp{\q_\alpha}{\u_k} \scp{\q_\alpha}{\u_j}\u_k(a) \u_j(b)}{(\lambda_k - z )(\lambda_j - z)}.\label{e:firstterm}
\end{align}
We also have
\begin{align}
\frac{1}{N} \sum_{i=1}^N G_{i a }(z) G_{bi}(z)  &= \frac{1}{N} \sum_{i=1}^N \left(\sum_{k=1}^N  \frac{\u_k(i) \u_k(a) }{\lambda_k - z }\right) \left( \sum_{j=1}^N \frac{ \u_j(b) \u_j(i) }{\lambda_j - z }\right)\\
&= \frac{1}{N} \sum_{k=1}^N \frac{\u_k(a) \u_k(b) }{(\lambda_k - z )^2}\label{e:center}.
\end{align}
We used that the terms in the product with $k\neq j$ vanish due to the sum over $i$ and the orthogonality of $\u_k$ and $\u_j$.

Subtracting \eqref{e:center} from the first term of \eqref{e:firstterm}, we have
\begin{align}
\Im  \frac{1} {\sqrt{|I_N | } }\sum_{k=1}^N \sum_{\alpha \in I_N}  \frac{ ( \scp{\q_\alpha}{\u_k}^2 - N^{-1}) \u_k(a) \u_k(b) }{(\lambda_k - z )^2} &= 
\frac{1}{N} \Im \sum_{k=1}^N \frac{  \hat p_{kk}(I_N) \u_k(a) \u_k(b) }{(\lambda_k - z )^2}\\
&\le N^{ - 2 + \sigma + 3 \omega}   \sum_{k=1}^N \frac{  1 }{|\lambda_k - z |^2}\\
&\le N^{ - 2 + \sigma + 3 \omega}  \frac{1}{\eta_\ell} \Im \sum_{k=1}^N \frac{  1 }{\lambda_k - z }\\
&\le N^{ \sigma + 3 \omega + 5 \eps_2 } \left( \frac{\hat \ell}{N}\right)^{1/3}. 
\end{align}
In the last inequality, we used \eqref{e:belowG}. It remains to control the second term in \eqref{e:firstterm}. Using \eqref{doublesum}, this is 
\begin{align}
\Im  \sum_{k \neq j}   \frac{1} {\sqrt{|I_N | } } \sum_{\alpha \in I_N} \frac{ \scp{\q_\alpha}{\u_k} \scp{\q_\alpha}{\u_j} \u_k(a) \u_j(b)}{(\lambda_k - z )(\lambda_j - z)} &=\frac{1}{N} \Im  \sum_{k \neq j}  \frac{  \hat p_{kj} \u_k(a) \u_j(b)}{(\lambda_k - z )(\lambda_j - z)} \\
& \le N^{-2 + \sigma +3 \omega}  \sum_{k \neq j} \Im \frac{1}{(\lambda_k - z )(\lambda_j - z)} \\
%& \le N^{\sigma + 3 \omega} \Im m_N(E + \I \eta_\ell)^2 \\
& \le N^{ \sigma +3 \omega + 8 \eps_2} \left( \frac{\hat \ell}{N}\right)^{1/3}.
\end{align}

For the second derivative, we compute
\begin{align}
\partial^2_{ab}& \left( \frac{1}{ \sqrt{| I_N|}} \sum_{ \alpha \in I_N} \Im \scp{\q_\alpha}{G(z)\q_\alpha} -  \Im m_N(z)\right)
\end{align}
by noting that 
\begin{align}
\partial^2_{ab} \scp{\q_\alpha}{G(z)\q_\alpha} &= 2 \scp{\q_\alpha}{G(z)\eb_a} G_{b a}(z) \scp{\eb_b}{G(z)\q_\alpha}\\
&+ \scp{\q_\alpha}{G(z)\eb_b} G_{a a}(z) \scp{\eb_b}{G(z)\q_\alpha} \\
&+ \scp{\q_\alpha}{G(z)\eb_a} G_{b a}(z) \scp{\eb_a}{G(z)\q_\alpha}. 
\end{align}
One representative term is
\begin{align}
&
 \frac{2}{ \sqrt{| I_N|}} \sum_{ \alpha \in I_N} \left( \Im \scp{\q_\alpha}{G(z)\eb_a} G_{b a}(z) \scp{\eb_b}{G(z)\q_\alpha}    -  \Im \frac{1}{N} \sum_{i=1}^N G_{i a }(z) G_{ba}(z) G_{bi}(z) \right)\\
&=  \frac{2}{ \sqrt{| I_N|}} \sum_{ \alpha \in I_N} \left( \Im \left[ G_{b a}(z) \left( \scp{\q_\alpha}{G(z)\eb_a}  \scp{\eb_b}{G(z)\q_\alpha}    -  \frac{1}{N} \sum_{i=1}^N G_{i a }(z) G_{bi}(z) \right) \right] \right)\\
&=  \frac{2}{ \sqrt{| I_N|}}\Im G_{b a}(z)  \sum_{ \alpha \in I_N}    \Re \left( \scp{\q_\alpha}{G(z)\eb_a}  \scp{\eb_b}{G(z)\q_\alpha}    -  \frac{1}{N} \sum_{i=1}^N G_{i a }(z) G_{bi}(z)  \right)\\
&+\frac{2}{ \sqrt{| I_N|}}\Re G_{b a}(z)  \sum_{ \alpha \in I_N}   \Im \left( \scp{\q_\alpha}{G(z)\eb_a}  \scp{\eb_b}{G(z)\q_\alpha}    -  \frac{1}{N} \sum_{i=1}^N G_{i a }(z) G_{bi}(z) \right)  \\
& \le C N^{ \sigma +3 \omega + 12 \eps_2}\left( \frac{\hat \ell}{N}\right)^{1/3}.
\end{align}
In the last line, we used \eqref{e:belowG}, \eqref{e:belowGreal}, and repeated the computations for \eqref{e:firstd}  to obtain
\beq
\frac{2}{ \sqrt{| I_N|}}  \sum_{ \alpha \in I_N}    \Re \left( \scp{\q_\alpha}{G(z)\eb_a}  \scp{\eb_b}{G(z)\q_\alpha} -   \frac{1}{N} \sum_{i=1}^N G_{i a }(z) G_{bi}(z) \right) \le C N^{\sigma +3 \omega + 8 \eps_2}.
\eeq

Differentiating further, we find terms such as

\begin{align}\label{forex}
%\partial^{k+1}_{ab}&  \left( \frac{1}{ \sqrt{| I_N|}} \sum_{ \alpha \in I_N} %\Im \scp{\q_\alpha}{G(z)\q_\alpha} -  \Im m_N(z)\right)\\
%&=
\frac{(k+1)!}{ \sqrt{| I_N|}} \sum_{ \alpha \in I_N}  \Im \left(  \scp{\q_\alpha}{G(z)\eb_a} G_{b a}(z)^k \scp{\eb_b}{G(z)\q_\alpha}   -   \frac{1}{N} \sum_{i=1}^N G_{i a }(z) G_{ba}(z)^k G_{bi}(z) \right)
\end{align}
These higher derivatives are bounded exactly as the second derivative was. For example, in  \eqref{forex}, we separate out the $G_{ba}^k$ term and bounding it using \eqref{e:belowG}.
\ep

\section{Symmetrized moment observables}\label{a:genmoments}
In this section, we define a generalization of the moment observables given in Definition \ref{d:momobs1}, which facilitated the proof of decorrelation estimates for entries of one or two eigenvectors. These generalized observables lead to decorrelation estimates for more than two eigenvectors, which may be of independent interest. For brevity, we only verify that the observables follow a parabolic equation. We omit the proofs of the corresponding decorrelation estimates, since they are straightforward  applications of the maximum principle, as in \Cref{s:decorrelation}.

 Consider a configuration of $n$ particles $\bm{\xi}$ as in Subsection \ref{subsec:emf}. We denote by $i_1,\dots,i_p$ the sites where there is at least one particle and define the set of vertices
\[
\mathscr{V}_{\bm{\xi}} = \{(i_q,a)\in\N^2,\, q\in\unn{1}{p},\, a\in\unn{1}{\xi_{i_q}} \}.
\]
Note that $\mathscr{V}_{\bm{\xi}}$ differs from $\mathcal{V}_{\bm{\xi}}$ (defined in \eqref{mathcalV}) since we do not double the number of particles here. We define now a set of matchings $\mathfrak{M}_{\bm{\xi}}$ as the set of functions $\sigma$ such that if $v\in\mathscr{V}_{\bm{\xi}}$, then $\sigma(v)=(\sigma_1(v),\sigma_2(v))\in\unn{1}{2n}^2$ with $\sigma_1(v)<\sigma_2(v)$ and $\bigcup_{v\in\mathscr{V}_{\bm{\xi}}}\{\sigma_1(v),\sigma_2(v)\}=\unn{1}{2n}$. This matches each particle to a pair of distinct indices in $\unn{1}{2n}$. Note that $\vert \mathfrak{M}_{\bm{\xi}}\vert =2^{-n}(2n)!$.

\begin{figure}[!ht]
	\centering
	\begin{subfigure}[t]{.5\textwidth}
		\centering
		\includegraphics[width=.8\linewidth]{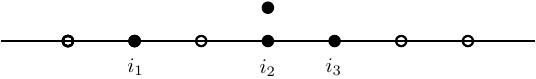}
		\caption{A configuration $\bm{\xi}$ with $\mathcal{N}(\bm{\xi})=4$ particles.}
	\end{subfigure}%
	\hspace{-0em}\begin{subfigure}[t]{.5\linewidth}
		\centering
		\includegraphics[width=.8\linewidth]{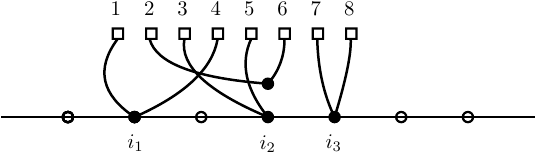}
		\captionof{figure}{An example of a matching $\sigma\in\mathfrak{M}_{\bm{\xi}}$}
	\end{subfigure}
\end{figure}
 We then define the generalized moment observable as 
 \[
 g_s(\bm{\xi})=\frac{2^n}{(2n)!\M(\bm{\xi})}\sum_{\sigma\in\mathfrak{M}_{\bm{\xi}}}\EL{\prod_{\substack{v\in\mathscr{V}_{\bm{\xi}}\\v=(k,a)}}\scp{\q_{\alpha_{\sigma_1(v)}}}{\u_k}\scp{\q_{\alpha_{\sigma_2(v)}}}{\u_k}}\quad\text{with}\quad\M(\bm{\xi})=\prod_{i=1}^N(2\xi_k-1)!!.
 \]
 
 \bep\label{p:flow2}
 For all $s\in (0,1)$, the perfect matching observable $f_s$ defined in \eqref{eq:deffs} satisfies the equation
 \beq\label{eq:emf2}
 \partial_s g_s(\bm{\xi})
 =
 \sum_{k\neq \ell}
 \xi_k(1+2\xi_\ell)
 \frac{g_s(\bm{\xi}^{k,\ell})-g_s(\bm{\xi})}{N(\lambda_k(s)-\lambda_\ell(s))^2},
 \eeq
 where $\bm{\xi}^{k,\ell}$ is the configuration where we move a particle from $k$ to $\ell$.
 \eep
 
 Before giving the proof of Proposition \ref{p:flow2}, we recall the generator for the dynamics \eqref{eq:dysonvect}.
 
\bel
 The generator acting on smooth functions of the eigenvector diffusion \eqref{eq:dysonvect} is
 \[
 \mathcal{L}_s = \frac{1}{2}\sum_{1\leqslant k<\ell\leqslant N}\frac{1}{N(\lambda_k-\lambda_\ell)^2}X_{k\ell}^2
 \]
 with
 \[
 X_{k\ell}=\sum_{\alpha =1}^N(\u_k(\alpha)\partial_{\u_\ell(\alpha)}-\u_\ell(\alpha)\partial_{\u_k(\alpha)}).
 \]
\eel

\begin{proof}[Proof of Proposition \ref{p:flow2}]
	To simplify the notation we define $\v_k(p)= \scp{\q_{\alpha_p}}{\u_k}$ so that 
	\[
	g_s(\bm{\xi})=\frac{2^n}{(2n)!\M(\bm{\xi})}\sum_{\sigma\in \mathfrak{M}_{\bm{\xi}}}\EL{\prod_{\substack{v\in \mathscr{V}_{\bm{\xi}}\\v=(k,a)}}\v_k(\sigma_1(v))\v_k(\sigma_2(v))}.
	\]
	Consider $k,\ell\in\unn{1}{N}$ such that there is at least one particle of $\bm{\xi}$ on sites $k$ and $\ell$. Then we can write
	\[
	\prod_{\substack{v\in \mathscr{V}_{\bm{\xi}}\\v=(p,a)}}\v_p(\sigma_1(v))\v_p(\sigma_2(v))
	=
	P(k,\ell)
	Q(k)Q(\ell)
	\]
	with \[Q(j)=\prod_{\substack{v\in \mathscr{V}_{\bm{\xi}}\\v=(j,a)}}\v_j(\sigma_1(v))\v_j(\sigma_2(v))\quad\text{and}\quad P(k,\ell)=\prod_{\substack{v\in \mathscr{V}_{\bm{\xi}}\\v=(p,a)\\p\neq k,\ell}}\v_p(\sigma_1(v))\v_p(\sigma_2(v)).\]
 Note that $P(k,\ell)$ does not involve any eigenvectors $\u_k$ or $\u_\ell$. We have
 \beq\label{e:sym-1}
 X_{k\ell}^2[P(k,\ell)
 Q(k)Q(\ell)]=P(k,\ell)X_{k\ell}\left[Q(\ell)X_{k\ell}Q(k)+Q(k)X_{k\ell}Q(\ell)\right].
 \eeq
 We now consider the first term $X_{k\ell}Q(\ell)X_{k\ell}Q(k)+Q(\ell)X_{k\ell}^2Q(k)$; the second term can be computed similarly. 
We have
 \beq\label{e:sym0}
 X_{k\ell}Q(k)=\sum_{v=(k,a)\in\mathscr{V}_{\bm{\xi}}}(-\v_\ell(\sigma_1(v))\v_k(\sigma_2(v))-\v_k(\sigma_1(v))\v_\ell(\sigma_2(v)))\prod_{\substack{v'=(k,b)\in\mathscr{V}_{\bm{\xi}}\\b\neq a}}\v_k(\sigma_1(v'))\v_k(\sigma_2(v')).
 \eeq
 After symmetrization, we obtain
 \beq\label{e:sym1}
 \frac{2^n}{(2n)!\M(\bm{\xi})}\sum_{\sigma\in\mathfrak{M}_{\bm{\xi}}}\EL{P(k,\ell)X_{k\ell} Q(k)X_{k\ell}Q(\ell)}
 	=
 	-\frac{2^n(4\xi_k\xi_\ell)}{(2n)!\M(\bm{\xi})}\sum_{\sigma\in\mathfrak{M}_{\bm{\xi}}}\EL{Q(k)Q(\ell)P(k,\ell)}=-4\xi_k\xi_\ell g_s(\bm{\xi}),
 \eeq
 where the 4 comes from unfolding the products and $\xi_k\xi_\ell$ comes from choosing the particles in the sum in \eqref{e:sym0}. There are two terms when computing $X_{k\ell}^2Q(k)$ coming from either differentiating the first term or the product in \eqref{e:sym0}. When differentiating the first term we obtain
 \beq\label{e:sym2}
 (I)=\sum_{v=(k,a)\in\mathscr{V}_{\bm{\xi}}}(-2\v_k(\sigma_1(v))\v_k(\sigma_1(v))+2\v_\ell(\sigma_1(v))\v_\ell(\sigma_2(v)))\prod_{\substack{v'=(k,b)\in\mathscr{V}_{\bm{\xi}}\\b\neq a}}\v_k(\sigma_1(v'))\v_k(\sigma_2(v')).
 \eeq
 When differentiating the second term, we obtain a term with a double sum
 \begin{multline}
(II) =\sum_{\substack{v=(k,a)\in\mathscr{V}_{\bm{\xi}}\\w=(k,b)\in\mathscr{V}_{\bm{\xi}}\\a\neq b}}
(-\v_\ell(\sigma_1(v))\v_k(\sigma_2(v))-\v_k(\sigma_1(v))\v_\ell(\sigma_2(v)))(-\v_\ell(\sigma_1(w))\v_k(\sigma_2(w))-\v_k(\sigma_1(w))\v_\ell(\sigma_2(w)))\\
\times \prod_{\substack{v'=(k,c)\in\mathscr{V}_{\bm{\xi}}\\c\neq a,b}}\v_k(\sigma_1(v'))\v_k(\sigma_2(v')).
 \end{multline}
After symmetrization, we have
\beq\label{e:sym3}
\frac{2^n}{(2n)!\M(\bm{\xi})}\sum_{\sigma\in\mathfrak{M}_{\bm{\xi}}}\EL{P(k,\ell)Q(\ell)(I)}
=
-2\xi_kg_s(\bm{\xi})+2\xi_k\frac{2\xi_\ell+1}{2\xi_k-1}g_s(\bm{\xi}^{k,\ell}),
\eeq
where the $\xi_k$ comes from the sum in \eqref{e:sym2} and the second term stems from the identity  $\M(\bm{\xi}^{k,\ell})=\frac{2\xi_\ell+1}{2\xi_k-1}\M(\bm{\xi})$. Similarly, the symmetrization gives
\beq\label{e:sym4}
\frac{2^n}{(2n)!\M(\bm{\xi})}\sum_{\sigma\in\mathfrak{M}_{\bm{\xi}}}\EL{P(k,\ell)Q(\ell)(II)}
=
4\xi_k(\xi_k-1)\frac{2\xi_\ell+1}{2\xi_k-1}g_s(\bm{\xi}^{k,\ell}).
\eeq
Finally, when summing \eqref{e:sym1}, \eqref{e:sym3}, \eqref{e:sym4}, and adding the second term coming from \eqref{e:sym-1}, we obtain
\[
X_{k\ell}^2g_s(\bm{\xi})
	=
	2\xi_k(1+2\xi_\ell)(g_s(\bm{\xi}^{k,\ell})-g_s(\bm{\xi}))
	+
	2\xi_\ell(1+2\xi_k)(g_s(\bm{\xi}^{\ell,k})-g_s(\bm{\xi})).
\]
If there is a particle on site $k$ and not on site $\ell$, the computation is similar but easier and if there are no particles in site $k$ and $\ell$ then $X_{k\ell}g_s(\bm{\xi})=0$. This completes the proof of Proposition \ref{p:flow2}.
\end{proof}
\bibliography{mass}

\begin{bibdiv}
\begin{biblist}

\bib{aggarwal2019bulk}{article}{
      author={Aggarwal, A.},
       title={Bulk universality for generalized {W}igner matrices with few
  moments},
        date={2019},
        ISSN={0178-8051},
     journal={Probab. Theory Related Fields},
      volume={173},
      number={1-2},
       pages={375\ndash 432},
}

\bib{aggarwal2020eigenvector}{article}{
      author={Aggarwal, A.},
      author={Lopatto, P.},
      author={Marcinek, J.},
       title={Eigenvector statistics of {L}{\'e}vy matrices},
        date={2021},
     journal={The Annals of Probability},
      volume={49},
      number={4},
       pages={1778\ndash 1846},
}

\bib{aggarwal2018goe}{article}{
      author={Aggarwal, A.},
      author={Lopatto, P.},
      author={Yau, H.-T.},
       title={Goe statistics for {L}{\'e}vy matrices},
        date={2021},
     journal={Journal of the European Mathematical Society},
}

\bib{anantharaman2008entropy}{article}{
      author={Anantharaman, N.},
       title={Entropy and the localization of eigenfunctions},
        date={2008},
     journal={Ann. of Math. (2)},
      volume={168},
       pages={435\ndash 475},
}

\bib{anantharaman2015quantum}{article}{
      author={Anantharaman, N.},
      author={Le~Masson, E.},
       title={Quantum ergodicity on large regular graphs},
        date={2015},
     journal={Duke Math. J.},
      volume={164},
      number={4},
       pages={723\ndash 765},
}

\bib{anantharaman2019quantum}{article}{
      author={Anantharaman, N.},
      author={Sabri, M.},
       title={Quantum ergodicity on graphs: from spectral to spatial
  delocalization},
        date={2019},
     journal={Ann. of Math. (2)},
      volume={189},
      number={3},
       pages={753\ndash 835},
}

\bib{bauerschmidt2017local}{article}{
      author={Bauerschmidt, R.},
      author={Knowles, A.},
      author={Yau, H.-T.},
       title={Local semicircle law for random regular graphs},
        date={2017},
     journal={Communications on Pure and Applied Mathematics},
      volume={70},
      number={10},
       pages={1898\ndash 1960},
}

\bib{benigni2020eigenvectors}{article}{
      author={Benigni, L.},
       title={Eigenvectors distribution and quantum unique ergodicity for
  deformed {W}igner matrices},
organization={Institut Henri Poincar{\'e}},
        date={2020},
     journal={Ann. Inst. H. Poincaré Probab. Statist.},
      volume={56},
      number={4},
       pages={2822\ndash 2867},
}

\bib{benigni2021fermionic}{article}{
      author={Benigni, L.},
       title={Fermionic eigenvector moment flow},
        date={2021},
        ISSN={0178-8051},
     journal={Probab. Theory Related Fields},
      volume={179},
      number={3-4},
       pages={733\ndash 775},
}

\bib{benigni2020optimal}{article}{
      author={Benigni, L.},
      author={Lopatto, P.},
       title={Optimal delocalization for generalized {W}igner matrices},
        date={2020},
     journal={arXiv preprint},
}

\bib{bloemendal2014isotropic}{article}{
      author={Bloemendal, A.},
      author={Erd\H{o}s, L.},
      author={Knowles, A.},
      author={Yau, H.-T.},
      author={Yin, J.},
       title={Isotropic local laws for sample covariance and generalized
  {W}igner matrices},
        date={2014},
     journal={Electron. J. Probab.},
      volume={19},
       pages={53 pp.},
}

\bib{bourgade2018survey}{inproceedings}{
      author={Bourgade, P.},
       title={Random band matrices},
        date={2018},
   booktitle={Proceedings of the {I}nternational {C}ongress of
  {M}athematicians---{R}io de {J}aneiro 2018. {V}ol. {IV}. {I}nvited lectures},
   publisher={World Sci. Publ., Hackensack, NJ},
       pages={2759\ndash 2784},
}

\bib{bourgade2018extreme}{article}{
      author={Bourgade, P.},
       title={Extreme gaps between eigenvalues of {W}igner matrices},
        date={2021},
     journal={Journal of the European Mathematical Society},
}

\bib{bourgade2017eigenvector}{article}{
      author={Bourgade, P.},
      author={Huang, J.},
      author={Yau, H.-T.},
       title={Eigenvector statistics of sparse random matrices},
        date={2017},
     journal={Electron. J. of Probab.},
      volume={22},
}

\bib{bourgade2013eigenvector}{article}{
      author={Bourgade, P.},
      author={Yau, H.-T.},
       title={The eigenvector moment flow and local quantum unique ergodicity},
        date={2017},
        ISSN={0010-3616},
     journal={Comm. Math. Phys.},
      volume={350},
      number={1},
       pages={231\ndash 278},
}

\bib{bourgade2018random}{article}{
      author={Bourgade, P.},
      author={Yau, H.-T.},
      author={Yin, J.},
       title={Random band matrices in the delocalized phase {I}: Quantum unique
  ergodicity and universality},
        date={2020},
     journal={Comm. Pure Appl. Math.},
      volume={73},
      number={7},
       pages={1526\ndash 1596},
}

\bib{brooks2013non}{article}{
      author={Brooks, S.},
      author={Lindenstrauss, E.},
       title={Non-localization of eigenfunctions on large regular graphs},
        date={2013},
     journal={Israel J. Math.},
      volume={193},
      number={1},
       pages={1\ndash 14},
}

\bib{brooks2014joint}{article}{
      author={Brooks, S.},
      author={Lindenstrauss, E.},
       title={Joint quasimodes, positive entropy, and quantum unique
  ergodicity},
        date={2014},
     journal={Invent. Math.},
      volume={198},
      number={1},
       pages={219\ndash 259},
}

\bib{erdHos2020eigenstate}{article}{
      author={Cipolloni, G.},
      author={Erd{\H{o}}s, L.},
      author={Schr{\"o}der, D.},
       title={Eigenstate {T}hermalization {H}ypothesis for {W}igner matrices},
        date={2020},
     journal={arXiv preprint},
}

\bib{cipolloni2021normal}{article}{
      author={Cipolloni, G.},
      author={Erd{\H{o}}s, L.},
      author={Schr{\"o}der, D.},
       title={Normal fluctuation in quantum ergodicity for {W}igner matrices},
        date={2021},
     journal={arXiv preprint},
}

\bib{de1985ergodicite}{article}{
      author={Colin De~Verdière, Y.},
       title={Ergodicit{\'e} et fonctions propres du laplacien},
        date={1985},
     journal={Comm. Math. Phys.},
      volume={102},
      number={3},
       pages={497\ndash 502},
}

\bib{deutsch2018eigenstate}{article}{
      author={Deutsch, J.},
       title={Eigenstate thermalization hypothesis},
        date={2018},
     journal={Rep. Prog. Phys.},
      volume={81},
      number={8},
       pages={082001},
}

\bib{eckhardt1995approach}{article}{
      author={Eckhardt, B.},
      author={Fishman, S.},
      author={Keating, J.},
      author={Agam, O.},
      author={Main, J.},
      author={M{\"u}ller, K.},
       title={Approach to ergodicity in quantum wave functions},
        date={1995},
     journal={Phys. Rev. E},
      volume={52},
      number={6},
       pages={5893},
}

\bib{erdos2009local}{article}{
      author={Erd\H{o}s, L.},
      author={Schlein, B.},
      author={Yau, H.-T.},
       title={Local semicircle law and complete delocalization for {W}igner
  random matrices},
        date={2009},
        ISSN={0010-3616},
     journal={Comm. Math. Phys.},
      volume={287},
      number={2},
       pages={641\ndash 655},
}

\bib{erdos2009semicircle}{article}{
      author={Erd\H{o}s, L.},
      author={Schlein, B.},
      author={Yau, H.-T.},
       title={Semicircle law on short scales and delocalization of eigenvectors
  for {W}igner random matrices},
        date={2009},
        ISSN={0091-1798},
     journal={Ann. Probab.},
      volume={37},
      number={3},
       pages={815\ndash 852},
}

\bib{erdos2010wegner}{article}{
      author={Erd\H{o}s, L.},
      author={Schlein, B.},
      author={Yau, H.-T.},
       title={Wegner estimate and level repulsion for {W}igner random
  matrices},
        date={2010},
        ISSN={1073-7928},
     journal={Int. Math. Res. Not.},
      number={3},
       pages={436\ndash 479},
}

\bib{erdos2011universality}{article}{
      author={Erd\H{o}s, L.},
      author={Schlein, B.},
      author={Yau, H.-T.},
       title={Universality of random matrices and local relaxation flow},
        date={2011},
     journal={Invent. Math.},
      volume={185},
      number={1},
       pages={75\ndash 119},
}

\bib{erdos2017dynamical}{article}{
      author={Erd\H{o}s, L.},
      author={Yau, H.-T.},
       title={A dynamical approach to random matrix theory},
        date={2017},
     journal={Courant Lecture Notes in Mathematics},
      volume={28},
}

\bib{erdos2012bulk}{article}{
      author={Erd{\H{o}}s, {L.}},
      author={Yau, H.-T.},
      author={Yin, J.},
       title={Bulk universality for generalized {W}igner matrices},
        date={2012},
     journal={Probab. Theory Related Fields},
      volume={154},
      number={1-2},
       pages={341\ndash 407},
}

\bib{erdos2012rigidity}{article}{
      author={Erd{\H{o}}s, {L.}},
      author={Yau, H.-T.},
      author={Yin, J.},
       title={Rigidity of eigenvalues of generalized {W}igner matrices},
        date={2012},
     journal={Adv. Math},
      volume={229},
      number={3},
       pages={1435\ndash 1515},
}

\bib{gotze2019local}{article}{
      author={G\"{o}tze, F.},
      author={Naumov, A.},
      author={Tikhomirov, A.},
       title={Local semicircle law under fourth moment condition},
        date={2019},
        ISSN={1572-9230},
     journal={J. Theoret. Probab.},
}

\bib{gotze2018local}{article}{
      author={G\"{o}tze, F.},
      author={Naumov, A.},
      author={Tikhomirov, A.},
      author={Timushev, D.},
       title={On the local semicircular law for {W}igner ensembles},
        date={2018},
        ISSN={1350-7265},
     journal={Bernoulli},
      volume={24},
      number={3},
       pages={2358\ndash 2400},
}

\bib{holowinsky2010sieving}{article}{
      author={Holowinsky, R.},
       title={Sieving for mass equidistribution},
        date={2010},
     journal={Ann. of Math. (2)},
       pages={1499\ndash 1516},
}

\bib{holowinsky2010mass}{article}{
      author={Holowinsky, R.},
      author={Soundararajan, K.},
       title={Mass equidistribution for hecke eigenforms},
        date={2010},
     journal={Ann. of Math. (2)},
       pages={1517\ndash 1528},
}

\bib{knowles2013eigenvector}{article}{
      author={Knowles, A.},
      author={Yin, J.},
       title={Eigenvector distribution of {W}igner matrices},
        date={2013},
     journal={Probab. Theory Related Fields},
      volume={155},
      number={3--4},
       pages={543\ndash 582},
}

\bib{knowles2013isotropic}{article}{
      author={Knowles, A.},
      author={Yin, J.},
       title={The isotropic semicircle law and deformation of {W}igner
  matrices},
        date={2013},
        ISSN={0010-3640},
     journal={Comm. Pure Appl. Math.},
      volume={66},
      number={11},
       pages={1663\ndash 1750},
}

\bib{lindenstrauss2006invariant}{article}{
      author={Lindenstrauss, E.},
       title={Invariant measures and arithmetic quantum unique ergodicity},
        date={2006},
     journal={Ann. of Math. (2)},
       pages={165\ndash 219},
}

\bib{luo1995quantum}{article}{
      author={Luo, W.},
      author={Sarnak, P.},
       title={Quantum ergodicity of eigenfunctions on
  {$\mathrm{PSL}_2(\mathbb{Z})/\mathbb{H}^2$}},
        date={1995},
     journal={Publ. Math. Inst. Hautes Etudes Sci.},
      volume={81},
      number={1},
       pages={207\ndash 237},
}

\bib{marcinek2020high}{article}{
      author={Marcinek, J.},
      author={Yau, H.-T.},
       title={High dimensional normality of noisy eigenvectors},
        date={2020},
     journal={arXiv preprint},
}

\bib{marklof2000quantum}{article}{
      author={Marklof, J.},
      author={Rudnick, Z.},
       title={Quantum unique ergodicity for parabolic maps},
        date={2000},
     journal={Geom. Funct. Anal.},
      volume={10},
      number={6},
       pages={1554\ndash 1578},
}

\bib{o2016eigenvectors}{article}{
      author={O'Rourke, S.},
      author={Vu, V.},
      author={Wang, K.},
       title={Eigenvectors of random matrices: a survey},
        date={2016},
     journal={J. Combin. Theory Ser. A},
      volume={144},
       pages={361\ndash 442},
}

\bib{rudelson2015delocalization}{article}{
      author={Rudelson, M.},
      author={Vershynin, R.},
       title={Delocalization of eigenvectors of random matrices with
  independent entries},
        date={2015},
     journal={Duke Math. J.},
      volume={164},
      number={13},
       pages={2507\ndash 2538},
}

\bib{rudnick1994behaviour}{article}{
      author={Rudnick, Z.},
      author={Sarnak, P.},
       title={The behaviour of eigenstates of arithmetic hyperbolic manifolds},
        date={1994},
     journal={Comm. Math. Phys.},
      volume={161},
      number={1},
       pages={195\ndash 213},
}

\bib{schubert2006upper}{article}{
      author={Schubert, R.},
       title={Upper bounds on the rate of quantum ergodicity},
        date={2006},
     journal={Ann. Henri Poincaré},
      volume={7},
      number={6},
       pages={1085\ndash 1098},
}

\bib{schubert2008rate}{article}{
      author={Schubert, R.},
       title={On the rate of quantum ergodicity for quantised maps},
        date={2008},
     journal={Ann. Henri Poincaré},
      volume={9},
      number={8},
       pages={1455\ndash 1477},
}

\bib{shnirel1974ergodic}{article}{
      author={Shnirelman, A.},
       title={Ergodic properties of eigenfunctions},
        date={1974},
     journal={Russian Math. Surveys},
      volume={29},
      number={6},
       pages={181\ndash 182},
}

\bib{srednicki1994chaos}{article}{
      author={Srednicki, M.},
       title={Chaos and quantum thermalization},
        date={1994},
     journal={Phys. Rev. E},
      volume={50},
      number={2},
       pages={888},
}

\bib{tao2010random}{article}{
      author={Tao, T.},
      author={Vu, V.},
       title={Random matrices: universality of local eigenvalue statistics up
  to the edge},
        date={2010},
        ISSN={0010-3616},
     journal={Comm. Math. Phys.},
      volume={298},
      number={2},
       pages={549\ndash 572},
}

\bib{tao2011random}{article}{
      author={Tao, T.},
      author={Vu, V.},
       title={Random matrices: universality of local eigenvalue statistics},
        date={2011},
        ISSN={0001-5962},
     journal={Acta Math.},
      volume={206},
      number={1},
       pages={127\ndash 204},
}

\bib{tao2012random}{article}{
      author={Tao, T.},
      author={Vu, V.},
       title={Random matrices: universal properties of eigenvectors},
        date={2012},
     journal={Random Matrices Theory Appl.},
      volume={1},
      number={01},
       pages={1150001},
}

\bib{vu2015random}{article}{
      author={Vu, V.},
      author={Wang, K.},
       title={Random weighted projections, random quadratic forms and random
  eigenvectors},
        date={2015},
        ISSN={1042-9832},
     journal={Random Structures Algorithms},
      volume={47},
      number={4},
       pages={792\ndash 821},
}

\bib{zelditch1987uniform}{article}{
      author={Zelditch, S.},
       title={Uniform distribution of eigenfunctions on compact hyperbolic
  surfaces},
        date={1987},
     journal={Duke Math. J.},
      volume={55},
      number={4},
       pages={919\ndash 941},
}

\end{biblist}
\end{bibdiv}
\bibliographystyle{abbrv}
\end{document}